\documentclass[12pt]{amsart}
\usepackage{amsfonts}
\usepackage{amssymb, amstext, amscd, amsmath}
\usepackage{graphicx}
\usepackage{times}
\usepackage{amsthm}
\makeatother
\usepackage{txfonts}
\usepackage{indentfirst}
\usepackage[all]{xy}
\usepackage{tikz}
\usepackage{subcaption}
\usepackage{pifont}
\usepackage{bbm}
\usepackage{mathrsfs}
\usepackage{lscape}
\usepackage{setspace}
\usepackage{color}
\usepackage{ulem}
\usepackage{color}
\usepackage{url}
\usepackage{float}
\everymath{\displaystyle}

\usepackage{geometry}
\renewcommand{\baselinestretch}{1.2} 

\DeclareMathOperator{\sgn}{sgn}
\DeclareMathOperator{\Ind}{Ind}

\usepackage[utf8]{inputenc}
\newtheoremstyle{noparens}%
{}{}%
{\itshape}{}%
{\bfseries}{.}%
{ }%
{\thmname{#1}\thmnumber{ #2}\mdseries\thmnote{ #3}}
\theoremstyle{noparens}
\allowdisplaybreaks
\usepackage{xcolor}   
\usepackage{color}              
\usepackage{indentfirst} 
\usepackage{url}
\usepackage[colorlinks = true,
linkcolor=blue,
urlcolor=blue,
citecolor=blue,
anchorcolor=blue,
backref=page]{hyperref}

\def\red#1{\textcolor{red}{#1}}
\usepackage{color,xcolor} 
\def\blue#1{\textcolor{blue}{#1}}
\usepackage{color,xcolor} 
\def\cyan#1{\textcolor{cyan}{#1}}
\def\purple#1{\textcolor{purple}{#1}}

\newtheorem{theorem}{Theorem}[section]
\newtheorem{corollary}{Corollary}[section]
\newtheorem{proposition}{Proposition}[section]

\newtheorem{remark}{Remark}[section]
\newtheorem{definition}{Definition}[section]
\newtheorem{example}{Example}[section]
\numberwithin{equation}{section}
\numberwithin{equation}{section}

\newtheorem{question}{Question}[section]
\newtheorem{construction}{Construction}[section]

\date{\today}

\begin{document}
	
	\title[The r-coverings and local moves of planar virtual knotoids]{The r-coverings and local moves of planar virtual knotoids}	
	
	\author[Jiacheng An]{Jiacheng An}
	\address{School of Mathematical Sciences, Dalian University of Technology, Dalian 116024, P. R. China}
	\email{anan0312@163.com}
	
	\author[Fengling Li]{Fengling Li}
	\address{School of Mathematical Sciences, Dalian University of Technology, Dalian 116024, P. R. China}
	\email{fenglingli@dlut.edu.cn}
	
	\author[Andrei Vesnin]{Andrei Vesnin}
	\address{Sobolev Institute of Mathematics of the Siberian Branch of the Russian Academy of Sciences, Novosibirsk 630090, Russia}
	\email{vesnin@math.nsc.ru}
	
	
	\thanks{F.\,Li supported in part by a grant of NSFC (No. 12331003) and the Fundamental Research Funds for the Central Universities (No. DUT25LAB302); A.\,V. supported by the state task to the Sobolev Institute of Mathematics (No. FWNF-2026-0031).}
	
	\begin{abstract}
		In this paper, we study planar virtual knotoids, which generalize virtual knots and knotoids both. We consider the problem of whether two given planar virtual knotoids can be transformed into each other via a sequence of local moves and what is the shortest sequence length. By introducing $r$-covering of planar virtual knotoids, we determine the homotopy relationship between different planar virtual knotoids, and obtain lower bounds of the Gordian distance between them. On the basis of these results, we demonstrate calculations of the exact Gordian distance of several given pairs of homotopic planar virtual knotoids.  Furthermore, we investigate $\Delta$-moves and virtual region crossing changes for planar virtual knotoids, and derive lower bounds of the Gordian distance with respect to both moves via $r$-coverings.
	\end{abstract}
	
	\maketitle

	\section{Introduction} \label{sec1}
	
	Crossing change is one of the elementary local moves in classical knot theory. The idea of quantifying the difficulty of untying a knot via crossing changes can be traced back to Tait~\cite{Ta77}, who coined the term beknottedness in~1877 to count the minimal crossing flips needed to simplify a knot diagram, albeit his measure was restricted to a single planar drawing and lacked topological invariance without a complete theory of knot isotopy. Decades later, a rigorous theoretical framework was built by Wendt~\cite{We37} in 1937. He formally proved that every classical knot can be converted into the trivial knot through finitely many crossing changes, and standardized the modern definition of the unknotting number as the minimum number of crossing changes over all isotopic diagrams of a knot. Additionally, Wendt derived explicit lower bounds for the unknotting number by analyzing the first homology groups of cyclic branched coverings of $S^3$ branched along the knot. In~\cite{Mu85}, Murakami further considered the minimal number of crossing changes required to transform a diagram of a knot $K$ into that of another knot $K'$, where the minimum is taken over all diagrams of $K$ from which can obtain diagrams of $K'$. This led him to define the Gordian distance $d_G(K, K')$ between two knots $K$ and $K'$, and to investigate its connection with the double coverings of $S^3$ branched along knots. The broader topological foundations of crossing change, encompassing not only the unknotting number and Gordian distance but also connections to Dehn surgery, sutured manifold theory, and $4$-manifold topology, are systematically surveyed by Scharlemann in~\cite{Sc98}. This expository article organizes decades of scattered results into a unified geometric framework and remains the standard reference for the classical theory of crossing changes.
	
	There also exist numerous local moves that can untie arbitrary knot. Among these moves, the $\Delta$-move was introduced by Murakami and Nakanishi~\cite{MN89}. Based on the $\Delta$-move, they also defined the $\Delta$-unknotting number and $\Delta$-Gordian distance, and established their relations to classical invariants. In addition, Hoste, Nakanishi, and Taniyama~\cite{HNT90} defined the $H(n)$-move and proved that an $H(n)$-move can be realized by an $H(n+1)$-move, see~\cite[Lemma 2]{HNT90}. Besides, using the $\sharp$-move, Murakami~\cite{Mu85} defined the $\sharp$-Gordian distance and revealed several key connections with the signature and the Arf invariant. Both the $H(n)$-move and $\sharp$-move change all crossings within a single region of a diagram. Shimizu~\cite{Sh14} defined the region crossing change and region unknotting number, and established a lower bound for the latter. Gordian complexes of knots by region crossing changes were studied in~\cite{GPV}. This operation was later generalized to link diagrams on closed surfaces of arbitrary genus by Cheng et al.~\cite{CCXZ22} for orientable surfaces and by Cheng et al.~\cite{CLSZ26} for nonorientable surfaces, who gave a unified rank formula for the incidence matrix counting equivalence classes under region crossing changes. These operations are collectively referred to as unknotting operations for knots.
	
	As classical knots were generalized to virtual knots by Kauffman~\cite{Ka21}, the above moves have also been applied to  the virtual knots, but they do not necessarily serve as unknotting operations for virtual knots. Kishino's knot provides an explicit counterexample. Indeed, according to ~\cite{DK09} and ~\cite{Mi08}, crossing changes preserve the non-classical topological data detected by their polynomial invariants, making it impossible to reduce Kishino's knot to the trivial knot using crossing changes alone.  Satoh and Taniguchi~\cite{ST14} proved that even though both the $\Delta$-move and the $\sharp$-move are presented by crossing changes, neither move is an unknotting operation for virtual knots. For these reasons, there exist studies on local moves of virtual knots along two distinct directions. The first is to consider virtual knots modulo a forbidden move, that leads to the class of welded knots. Satoh~\cite{Sa18} proved that crossing changes, $\Delta$-moves, and $\sharp$-moves are all unknotting operations for welded knots. The second direction is to consider virtual knots up to homotopy.  Jeong~\cite{Je23} gave a lower bound on the Gordian distance between two homotopic virtual knots using a two-variable polynomial invariant.
	
	The theory of knotoids was introduced by Turaev~\cite{Tu12} in $2012$. A knotoid diagram $D$ is a generic immersion of an oriented unit interval $[0,1]$ into an oriented surface $\Sigma$, with over/under information assigned to all double points. A knotoid $K$ is the equivalence class $[D]$ of such diagrams under Reidemeister moves and isotopies on $\Sigma$. Knotoids on $S^2$ are called spherical knotoids, while those on $\mathbb{R}^2$ are planar knotoids. The basic concepts of knotoids have been thoroughly studied in~\cite{Tu12}, including the introduction of several knotoid invariants and the monoid of knotoids associated with $\theta$-curves. In $2017$, Gügümcü and Kauffman~\cite{GK17} generalized knotoids to virtual knotoids and constructed several new invariants of the latter. A~comprehensive overview of knotoids, braidoids and their applications was given by Gügümcü, Kauffman and Lambropoulou~\cite{GKL19}. In~\cite{BG21}, Barbensi and Goundaroulis defined the $f$-distance between knotoids as the minimal number of forbidden moves. Adams et al.~\cite{ABCCJRSSW24} introduced the notion of hyperbolicity for both spherical and planar knotoids, and discuss its properties.
	
	As established in~\cite{Tu12}, the natural embedding of $\mathbb R^2$ into $S^2$ induces a surjective map from planar knotoids to spherical knotoids. This mapping is not injective, it follows that the theories of planar and spherical knotoids are essentially distinct. Kodokostas and Lambropoulou~\cite{KL19} investigated a connection between planar knotoids and arcs in $\mathbb R^3$ up to rail isotopy, initially introduced in~\cite{GK17}. Goundaroulis et al.~\cite{GGLDSK17} demonstrated that planar knotoids yield more detailed insights into knotted proteins compared with alternative approaches. A systematic classification of planar knotoids with up to five crossings was carried out by Goundaroulis, Dorier and Stasiak in~\cite{GDS19}. In particular, they proved that the number of prime planar knotoids with five crossings lies between $944$ and $950$, leaving six pairs of diagrams unresolved. Later, Molmaker and van der Veen~\cite{MV23} proved that among the six unresolved pairs of planar knotoids, two can be distinguished using quantum invariants.
	
	In this paper, we focus on planar virtual knotoids. The main questions we consider are whether, for two given planar virtual knotoids, there exists a finite sequence of local moves taking one to the other, and if so, what is the minimal number of such moves required. Two planar virtual knotoids $K$ and $K'$ are said to be \textit{homotopic} if their diagrams can be transformed into each other via a sequence of crossing changes and generalized Reidemeister moves, see Definition~\ref{def2.6}. Denote by $\mathbb Z_+$ the set of positive integers, and let $\mathbb N=\mathbb Z_+\cup\{0\}$. In~\cite{Tu04}, Turaev defined $r$-covering, $r \in \mathbb Z_+$, as an invariant for flat virtual knots. This covering concept has since been extended to virtual knots, see~\cite{NNS20} for details. In $2025$, Higa~\cite{Hi25} used the $r$-coverings of virtual knots to obtain lower bounds of the Gordian distances between homotopic virtual knots. For $r \in \mathbb Z_+$, we generalize the notion of $r$-covering from virtual knots to planar virtual knotoids. For a diagram $D$ of a planar virtual knotoid, its $r$-covering, denoted by $D^{(r)}$, is defined in Definition~\ref{D^{(r)}}. 
	
	We thus obtain the following results.
	
	\begin{theorem}\label{thm-invc}
		Let $D$ be a diagram for a planar virtual knotoid $K$, and $r\in\mathbb Z_+$. $D^{(r)}$ is the $r$-covering of $D$, and $[D^{(r)}]$ denotes the equivalence class of $D^{(r)}$ under generalized Reidemeister moves and planar isotopies. Then $[D^{(r)}]$ is an invariant of $K$.
	\end{theorem}
	
	We denote $[D^{(r)}]$ by $K^{(r)}$. Thus, $K^{(r)}$ is an invariant of $K$.
	
	Next we consider the Gordian distance between two homotopic planar virtual knotoids $K$ and $K'$. The \textit{Gordian distance} $d_G(K,K')$ between their diagrams, see Definition~\ref{def2.7}. 
	
	\begin{theorem}\label{thm-ccov}
		Let $K$ and $K'$ be planar virtual knotoids, and $r \in \mathbb Z_+$. If $K$ and $K'$ are homotopic, then $K^{(r)}$ and $K'^{(r)}$ are homotopic. Moreover, $d_G(K, K') \ge d_G(K^{(r)}, K'^{(r)})$.
	\end{theorem}
	
	In analogy with homotopy and Gordian distance associated with crossing changes, we introduce the $\Delta$-\textit{homotopy} and the $\Delta$-\textit{Gordian distance} $d_\Delta(K,K')$ associated with $\Delta$-moves, see Definition~\ref{def2.8} and Definition~\ref{def2.9}. 
	
	\begin{theorem}\label{thm-decov}
		Let $K$ and $K'$ be planar virtual knotoids, and $r \in \mathbb Z_+$. If $K$ and $K'$ are $\Delta$-homotopic, then $K^{(r)}$ and $K'^{(r)}$ are $\Delta$-homotopic. Moreover, $d_\Delta(K, K') \ge d_\Delta(K^{(r)}, K'^{(r)})$.
	\end{theorem}
	
	Kadokami~\cite{Ka12} introduced the diagrammatic region crossing change of virtual links. We generalize this move to planar virtual knotoids and define the \textit{virtual region crossing change}. Since a crossing change is not an unknotting operation for planar virtual knotoids, the virtual region crossing change, shortly, VRCC-move, is likewise not an unknotting operation. By analogy with the homotopy and Gordian distance defined via crossing changes, we define VRCC-homotopy together with the VRCC-Gordian distance $d_{\text{VRCC}}(K,K')$ induced by VRCC-moves, see Definition~\ref{def2.11} and Definition~\ref{def2.12}.
	
	\begin{theorem}\label{thm-vrcov}
		Let $K$ and $K'$ be planar virtual knotoids, and $r \in \mathbb Z_+$. If $K$ and $K'$ are VRCC-homotopic, then $K^{(r)}$ and $K'^{(r)}$ are VRCC-homotopic. Moreover, $d_\text{VRCC}(K, K') \ge d_\text{VRCC}(K^{(r)}, K'^{(r)})$.
	\end{theorem}
	
	This paper is organized as follows. In Section~\ref{sec2}, we review some fundamental concepts of knotoid theory. Section~\ref{sec3} is devoted to defining the $r$-covering of planar virtual knotoids and proving $K^{(r)}$ is an invariant of $K$, see Theorem~\ref{thm-invc}. In Section~\ref{sec4}, we turn our attention to the interaction between crossing changes and the $r$-coverings in planar virtual knotoids. The relevant results are presented in Theorem~\ref{thm-ccov}. Based on Theorem~\ref{thm-ccov}, we propose three constructions of families of planar virtual knotoids, which enable us to produce pairs of non-homotopic planar virtual knotoids. Throughout Section~\ref{sec5}, we present examples~\ref{ex:3}, \ref{ex:9} and~\ref{ex:8} for computing the Gordian distances between homotopic planar virtual knotoids by virtue of Theorem~\ref{thm-ccov}. Section~\ref{sec6} is dedicated to the relation between $\Delta$-moves and the $r$-coverings in planar virtual knotoids, see Theorem~\ref{thm-decov}. Finally, Section~\ref{sec7} is devoted to establishing the relationship between VRCC-moves and the $r$-coverings in planar virtual knotoids, see Theorem~\ref{thm-vrcov}.
	
	\section{Preliminaries} \label{sec2}
	
	\begin{definition}\label{def2.1}{\rm
			A \textit{knotoid diagram} $D$ in a two-dimensional surface $\Sigma$, possibly with boundary, is defined as a generic immersion of the unit interval $[0,1]$ into interior of $\Sigma$, namely a curve with finitely many isolated transverse double points endowed with standard over-arc and under-arc information, which are called \textit{classical crossings}. The images of $0$ and $1$, which are distinct from each other and any double point of the diagram, are referred to as the endpoints of the knotoid diagram, namely its \textit{tail} and \textit{head}, respectively. A knotoid diagram possesses a natural orientation, from tail to head. A trivial knotoid diagram is an embedding of the unit interval into the surface $\Sigma$. }
	\end{definition}
	
	The three types of \textit{Reidemeister moves} $\mathcal{R} = \{ \Omega_1, \Omega_2, \Omega_3 \}$ defined in Figure~\ref{fig1} and assumed be taken with all possible orientations of arcs act on knotoid diagrams.
	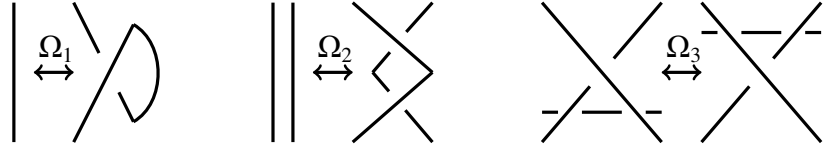
\begin{figure}[htbp]
		\centering
		\tikzset{every picture/.style={line width=0.75pt}}    
		\begin{tikzpicture}[x=0.75pt,y=0.75pt,yscale=-1,xscale=1]
			\draw  [black, very thick]  (100,100) -- (100,170) ;
			\draw  [black, very thick,<->]  (110,135) -- (130,135) ;
			\draw  [black, very thick]  (160,111) to [out=30,in=330] (160,159.8);	 
			\draw  [black, very thick]  (160,111) -- (130,170) ;
			\draw  [black, very thick]  (152.67,145) -- (160,159.8) ;
			\draw  [black, very thick]  (230,100) -- (230,170) ; 
			\draw  [black, very thick]  (240,100) -- (240,170) ; 
			\draw  [black, very thick,<->]  (250,135) -- (270,135) ;
			\draw  [black, very thick]  (270,100) -- (310,135) ; 
			\draw  [black, very thick]  (310,135) -- (270,170) ;
			\draw  [black, very thick]  (365,155) -- (373,155) ; 
			\draw  [black, very thick]  (388.98,142.03) -- (365,170) ; 
			\draw  [black, very thick]  (365,100) -- (425,170) ;
			\draw  [black, very thick]  (400.98,128.03) -- (425,100) ;
			\draw  [black, very thick]  (130,100) -- (142.71,125.31) ;
			\draw  [black, very thick]  (288.68,124.87) -- (280,135) ;
			\draw  [black, very thick]  (310,100) -- (296.87,114.7) ; 
			\draw  [black, very thick]  (280,135) -- (288.19,144.55) ; 
			\draw  [black, very thick]  (296.39,154.01) -- (310,170) ;
			\draw  [black, very thick]  (385,155) -- (405,155) ;
			\draw  [black, very thick]  (417,155) -- (425,155) ;
			\draw  [black, very thick,<->]  (425,135) -- (445,135) ;
			\draw  [black, very thick]  (445,115) -- (453,115) ;
			\draw  [black, very thick]  (468.98,142.03) -- (445,170) ;
			\draw  [black, very thick]  (445,100) -- (505,170) ;
			\draw  [black, very thick]  (480.98,128.03) -- (505,100) ;
			\draw  [black, very thick]  (465,115) -- (485,115) ;
			\draw  [black, very thick]  (497,115) -- (505,115) ;
			\draw (111,115) node [anchor=north west][inner sep=0.75pt]    {$\Omega_1$};
			\draw (251,115) node [anchor=north west][inner sep=0.75pt]    {$\Omega_2$};
			\draw (426,115) node [anchor=north west][inner sep=0.75pt]    {$\Omega_3$};
		\end{tikzpicture}
		\caption{Reidemeister moves $\Omega_1$, $\Omega_2$ and $\Omega_3$ on diagrams.}
		\label{fig1}
	\end{figure}
	
	\begin{definition}{\rm
			Two knotoid diagrams $D$ and $D'$ are said to be \textit{equivalent} if one can be transformed into the other via a sequence of Reidemeister moves and planar isotopies performed away from the tail and head. The corresponding equivalence classes are called \textit{knotoids}. When $\Sigma$ is either $S^2$ or $\mathbb{R}^2$, the corresponding knotoid is referred to as \textit{spherical} or \textit{planar}  accordingly.}
	\end{definition}
	
	Notably, we prohibit moving arcs over or under the two endpoints of the knotoid, as illustrated in Figure~\ref{fig2}. Note that if the moves $\Phi_+$ and $\Phi_-$ are permitted, any knotoid diagram in any surface can be readily transformed into the trivial knotoid diagram.
	\begin{figure}[htbp]
		\begin{center}
			\tikzset{every picture/.style={line width=0.75pt}}       
			\begin{tikzpicture}[x=0.75pt,y=0.75pt,yscale=-1,xscale=1]
				\draw  [black, very thick]  (70,50) -- (70,70) ;
				\draw  [black, very thick]  (70,90) -- (70,110) ;
				\draw  [black, very thick]  (40,80) -- (100,80) ;
				\draw  [black, very thick]  (240,50) -- (240,110) ;
				\draw  [black, very thick]  (160,80) -- (220,80) ;
				\draw  [black, very thick]  (330,50) -- (330,110) ; 
				\draw  [black, very thick]  (300,80) -- (320,80) ; 
				\draw  [black, very thick]  (340,80) -- (360,80) ; 
				\draw  [black, very thick,<->]  (120,80) -- (140,80) ;
				\draw  [black, very thick,<->]  (260,80) -- (280,80) ;
				\draw  [fill={rgb, 255:red, 0; green, 0; blue, 0 }  ,fill opacity=1 ][black, very thick]  (355.48,80) .. controls (355.48,77.5) and (357.5,75.48) .. (360,75.48) .. controls (362.5,75.48) and (364.52,77.5) .. (364.52,80) .. controls (364.52,82.5) and (362.5,84.52) .. (360,84.52) .. controls (357.5,84.52) and (355.48,82.5) .. (355.48,80) -- cycle ;
				\draw  [fill={rgb, 255:red, 0; green, 0; blue, 0 }  ,fill opacity=1 ][black, very thick]  (215.48,80) .. controls (215.48,77.5) and (217.5,75.48) .. (220,75.48) .. controls (222.5,75.48) and (224.52,77.5) .. (224.52,80) .. controls (224.52,82.5) and (222.5,84.52) .. (220,84.52) .. controls (217.5,84.52) and (215.48,82.5) .. (215.48,80) -- cycle ;
				\draw  [fill={rgb, 255:red, 0; green, 0; blue, 0 }  ,fill opacity=1 ][black, very thick]  (95.48,80) .. controls (95.48,77.5) and (97.5,75.48) .. (100,75.48) .. controls (102.5,75.48) and (104.52,77.5) .. (104.52,80) .. controls (104.52,82.5) and (102.5,84.52) .. (100,84.52) .. controls (97.5,84.52) and (95.48,82.5) .. (95.48,80) -- cycle ;
				\draw (120,61) node [anchor=north west][inner sep=0.75pt]    {$\Phi_+$};
				\draw (260,61) node [anchor=north west][inner sep=0.75pt]    {$\Phi_-$};
			\end{tikzpicture}
			\caption{Forbidden knotoid moves $\Phi_+$ and $\Phi_-$.}  \label{fig2}
		\end{center}
	\end{figure}
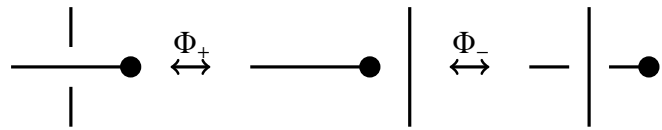
	
	We now recall the bracket polynomial of knotoids defined by Turaev~\cite{Tu12}.

	Let $D$ be a knotoid diagram in a surface $\Sigma$. A \textit{state} of $D$ is a map $s$ assigning to each crossing of $D$ a value in $\{-1, +1\}$. For a state $s$, we apply the A-smoothing at all crossings with $s=+1$ and the B-smoothing at all crossings with $s=-1$. The result is a compact $1$-manifold $D_s \subset \Sigma$ consisting of an embedded segment and several disjoint embedded circles.
	
	\begin{definition}{\rm				
			The \textit{bracket polynomial} of $D$ is defined by
			$$
			\langle D \rangle = \sum_{s \in S(D)} A^{\sigma_s} \left(-A^2 - A^{-2}\right)^{|s|-1},
			$$
			where $S(D)$ denotes the set of all states of $D$, $\sigma_s \in \mathbb{Z}$ is the sum of the values $\pm 1$ assigned by $s$ to the all crossings of $D$, and $|s|$ is the number of connected components of $D_s$.}
	\end{definition}			
	
	For any classical crossing $c$ in a knotoid diagram $D$, a sign function $\operatorname{sgn}(c) \in \{+1, -1\}$ is defined, as illustrated in Figure~\ref{fig3}~(a) and ~(b), and the $writhe$ $w(D)$ of the diagram $D$ is defined by
	$$
	w(D) = \sum_{c \in D} \operatorname{sgn}(c), \label{eq:writhe}
	$$
	where the sum is taken over all classical crossings of $D$.
	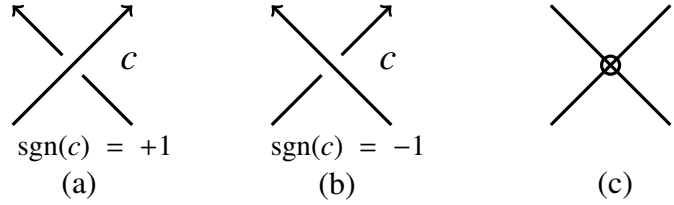
\begin{figure}[htbp]
		\begin{center}
			\tikzset{every picture/.style={line width=0.75pt}}         
			\begin{tikzpicture}[x=0.75pt,y=0.75pt,yscale=-1,xscale=1]
				\draw [black, very thick, ->]  (80,170) -- (140,110) ;
				\draw [black, very thick, ->]  (105,135) -- (80,110) ; 
				\draw [black, very thick]    (115,145) -- (140,170) ;
				\draw [black, very thick, ->]  (270,170) -- (210,110) ;
				\draw  [black, very thick, ->]  (245,135) -- (270,110) ;
				\draw  [black, very thick]  (235,145) -- (210,170) ; 
				\draw [black, very thick]    (410,170) -- (350,110) ;
				\draw [black, very thick]    (350,170) -- (410,110) ; 
				\draw  [black, very thick]  (375.48,140) .. controls (375.48,137.5) and (377.5,135.48) .. (380,135.48) .. controls (382.5,135.48) and (384.52,137.5) .. (384.52,140) .. controls (384.52,142.5) and (382.5,144.52) .. (380,144.52) .. controls (377.5,144.52) and (375.48,142.5) .. (375.48,140) -- cycle ;
				\draw (132.6,132) node [anchor=north west][inner sep=0.75pt]  [font=\large]  {$c$};
				\draw (81,174) node [anchor=north west][inner sep=0.75pt]  [font=\small]  {$\operatorname{sgn}( c) \ =\ +1$};
				\draw (208,174) node [anchor=north west][inner sep=0.75pt]  [font=\small]  {$\operatorname{sgn}( c) \ =\ -1$};
				\draw (262.6,132) node [anchor=north west][inner sep=0.75pt]  [font=\large]  {$c$};
				\draw (103,192) node [anchor=north west][inner sep=0.75pt]    {(a)};
				\draw (232,192) node [anchor=north west][inner sep=0.75pt]    {(b)};
				\draw (372,192) node [anchor=north west][inner sep=0.75pt]    {(c)};
			\end{tikzpicture}
		\end{center}
		\caption{The sign $\operatorname{sgn} (c)$ of a crossing $c$ and a virtual crossing.} \label{fig3}
	\end{figure}
	
	The product
	$$
	\langle D\rangle_\circ = (-A^3)^{-w(D)}\langle D\rangle,
	$$
	defines the \textit{normalized bracket polynomial}, which is a knotoid invariant according to \cite{Tu12}.

	\begin{definition}{\rm
			A \textit{virtual knotoid diagram} is defined to be a knotoid diagram in a surface $\Sigma$ augmented with an additional combinatorial structure, which we refer to as \textit{virtual crossings} (see Figure~\ref{fig3}~(c)).}
	\end{definition}
	
	\begin{definition}{\rm
			Two virtual knotoid diagrams in a surface $\Sigma$ are said to be \textit{equivalent} if they can be transformed into each other by a finite sequence of moves, each of which is either in $\textit{g} \mathcal{R} = \{ \Omega_1, \Omega_2, \Omega_3, \Omega_1^v,$ $\Omega_2^v, \Omega_3^v, \Omega_3^m, \Omega_v \}$, see Figures~\ref{fig1} and~\ref{fig4}, or an isotopy of the surface $\Sigma$. The equivalence classes of such diagrams are called \textit{virtual knotoids}.}
	\end{definition}
	\begin{figure}[htbp]
		\begin{center}
			\tikzset{every picture/.style={line width=0.75pt}}         
			\begin{tikzpicture}[x=0.75pt,y=0.75pt,yscale=-1,xscale=1] 
				\draw  [black, very thick]  (40,20) -- (40,80) ;
				\draw  [black, very thick]  (70,20) -- (100,70) ;
				\draw  [black, very thick,<->]  (50,50) -- (70,50) ; 
				\draw  [black, very thick]  (70,80) -- (100,30) ;
				\draw  [black, very thick]  (170,20) -- (140,50) ; 
				\draw  [black, very thick]  (140,50) -- (170,80) ;
				\draw  [black, very thick]  (140,20) -- (170,50) ;
				\draw  [black, very thick]  (170,50) -- (140,80) ;
				\draw  [black, very thick,<->]  (180,50) -- (200,50) ;
				\draw  [black, very thick]  (210,20) -- (210,80) ;
				\draw  [black, very thick]  (220,20) -- (220,80) ;
				\draw  [black, very thick]  (260,30) -- (320,30) ;
				\draw  [black, very thick]  (270,20) -- (310,80) ;
				\draw  [black, very thick]  (310,20) -- (270,80) ;
				\draw  [black, very thick]  (360,70) -- (420,70) ;
				\draw  [black, very thick]  (370,20) -- (410,80) ;
				\draw  [black, very thick]  (410,20) -- (370,80) ;
				\draw  [black, very thick,<->]  (330,50) -- (350,50) ;
				\draw  [black, very thick]  (40,120) -- (100,120) ;
				\draw  [black, very thick]  (50,110) -- (67.23,135.85) ;
				\draw  [black, very thick]  (90,110) -- (50,170) ;
				\draw  [black, very thick]  (73.69,145.54) -- (90,170) ;
				\draw  [black, very thick]  (140,160) -- (200,160) ;
				\draw  [black, very thick]  (150,110) -- (166.31,134.46) ;
				\draw  [black, very thick]  (190,110) -- (150,170) ;
				\draw  [black, very thick,<->]  (110,140) -- (130,140) ;
				\draw  [black, very thick]  (173.95,145.92) -- (190,170) ;
				\draw  [black, very thick]  (270,110) -- (270,170) ;
				\draw  [black, very thick]  (240,140) -- (300,140) ;
				\draw  [black, very thick]  (420,110) -- (420,170) ;
				\draw  [black, very thick]  (340,140) -- (400,140) ;
				\draw  [black, very thick,<->]  (310,140) -- (330,140) ;
				\draw  [black, very thick]  (100,30) to [out=30,in=330] (100,70);
				\draw  [fill={rgb, 255:red, 0; green, 0; blue, 0 }  ,fill opacity=1 ][black, very thick]  (395.48,140) .. controls (395.48,137.5) and (397.5,135.48) .. (400,135.48) .. controls (402.5,135.48) and (404.52,137.5) .. (404.52,140) .. controls (404.52,142.5) and (402.5,144.52) .. (400,144.52) .. controls (397.5,144.52) and (395.48,142.5) .. (395.48,140) -- cycle ;
				\draw  [fill={rgb, 255:red, 0; green, 0; blue, 0 }  ,fill opacity=1 ][black, very thick]  (295.48,140) .. controls (295.48,137.5) and (297.5,135.48) .. (300,135.48) .. controls (302.5,135.48) and (304.52,137.5) .. (304.52,140) .. controls (304.52,142.5) and (302.5,144.52) .. (300,144.52) .. controls (297.5,144.52) and (295.48,142.5) .. (295.48,140) -- cycle ;
				\draw  [black, very thick]  (152.03,159.76) .. controls (152.03,157.27) and (154.05,155.24) .. (156.55,155.24) .. controls (159.04,155.24) and (161.07,157.27) .. (161.07,159.76) .. controls (161.07,162.26) and (159.04,164.28) .. (156.55,164.28) .. controls (154.05,164.28) and (152.03,162.26) .. (152.03,159.76) -- cycle ;
				\draw  [black, very thick]  (398.6,70.05) .. controls (398.6,67.55) and (400.62,65.53) .. (403.12,65.53) .. controls (405.61,65.53) and (407.64,67.55) .. (407.64,70.05) .. controls (407.64,72.54) and (405.61,74.57) .. (403.12,74.57) .. controls (400.62,74.57) and (398.6,72.54) .. (398.6,70.05) -- cycle ;
				\draw  [black, very thick]  (372.31,69.76) .. controls (372.31,67.27) and (374.34,65.24) .. (376.83,65.24) .. controls (379.33,65.24) and (381.35,67.27) .. (381.35,69.76) .. controls (381.35,72.26) and (379.33,74.28) .. (376.83,74.28) .. controls (374.34,74.28) and (372.31,72.26) .. (372.31,69.76) -- cycle ;
				\draw  [black, very thick]  (265.48,140) .. controls (265.48,137.5) and (267.5,135.48) .. (270,135.48) .. controls (272.5,135.48) and (274.52,137.5) .. (274.52,140) .. controls (274.52,142.5) and (272.5,144.52) .. (270,144.52) .. controls (267.5,144.52) and (265.48,142.5) .. (265.48,140) -- cycle ;
				\draw  [black, very thick]  (178.88,160.05) .. controls (178.88,157.55) and (180.91,155.53) .. (183.4,155.53) .. controls (185.9,155.53) and (187.92,157.55) .. (187.92,160.05) .. controls (187.92,162.54) and (185.9,164.57) .. (183.4,164.57) .. controls (180.91,164.57) and (178.88,162.54) .. (178.88,160.05) -- cycle ;
				\draw  [black, very thick]  (78.88,120.05) .. controls (78.88,117.55) and (80.91,115.53) .. (83.4,115.53) .. controls (85.9,115.53) and (87.92,117.55) .. (87.92,120.05) .. controls (87.92,122.54) and (85.9,124.57) .. (83.4,124.57) .. controls (80.91,124.57) and (78.88,122.54) .. (78.88,120.05) -- cycle ;
				\draw  [black, very thick]  (52.31,120.05) .. controls (52.31,117.55) and (54.34,115.53) .. (56.83,115.53) .. controls (59.33,115.53) and (61.35,117.55) .. (61.35,120.05) .. controls (61.35,122.54) and (59.33,124.57) .. (56.83,124.57) .. controls (54.34,124.57) and (52.31,122.54) .. (52.31,120.05) -- cycle ;
				\draw  [black, very thick]  (385.48,50) .. controls (385.48,47.5) and (387.5,45.48) .. (390,45.48) .. controls (392.5,45.48) and (394.52,47.5) .. (394.52,50) .. controls (394.52,52.5) and (392.5,54.52) .. (390,54.52) .. controls (387.5,54.52) and (385.48,52.5) .. (385.48,50) -- cycle ;
				\draw  [black, very thick]  (285.48,50) .. controls (285.48,47.5) and (287.5,45.48) .. (290,45.48) .. controls (292.5,45.48) and (294.52,47.5) .. (294.52,50) .. controls (294.52,52.5) and (292.5,54.52) .. (290,54.52) .. controls (287.5,54.52) and (285.48,52.5) .. (285.48,50) -- cycle ;
				\draw  [black, very thick]  (298.6,29.76) .. controls (298.6,27.27) and (300.62,25.24) .. (303.12,25.24) .. controls (305.61,25.24) and (307.64,27.27) .. (307.64,29.76) .. controls (307.64,32.26) and (305.61,34.28) .. (303.12,34.28) .. controls (300.62,34.28) and (298.6,32.26) .. (298.6,29.76) -- cycle ;
				\draw  [black, very thick]  (272.31,29.62) .. controls (272.31,27.12) and (274.34,25.1) .. (276.83,25.1) .. controls (279.33,25.1) and (281.35,27.12) .. (281.35,29.62) .. controls (281.35,32.12) and (279.33,34.14) .. (276.83,34.14) .. controls (274.34,34.14) and (272.31,32.12) .. (272.31,29.62) -- cycle ;
				\draw  [black, very thick]  (150.48,65) .. controls (150.48,62.5) and (152.5,60.48) .. (155,60.48) .. controls (157.5,60.48) and (159.52,62.5) .. (159.52,65) .. controls (159.52,67.5) and (157.5,69.52) .. (155,69.52) .. controls (152.5,69.52) and (150.48,67.5) .. (150.48,65) -- cycle ;
				\draw  [black, very thick]  (150.48,35) .. controls (150.48,32.5) and (152.5,30.48) .. (155,30.48) .. controls (157.5,30.48) and (159.52,32.5) .. (159.52,35) .. controls (159.52,37.5) and (157.5,39.52) .. (155,39.52) .. controls (152.5,39.52) and (150.48,37.5) .. (150.48,35) -- cycle ;
				\draw  [black, very thick]  (83.74,49.76) .. controls (83.74,47.27) and (85.76,45.24) .. (88.26,45.24) .. controls (90.76,45.24) and (92.78,47.27) .. (92.78,49.76) .. controls (92.78,52.26) and (90.76,54.28) .. (88.26,54.28) .. controls (85.76,54.28) and (83.74,52.26) .. (83.74,49.76) -- cycle ;
				\draw (50.5,28) node [anchor=north west][inner sep=0.75pt]    {$\Omega_{1}^{v}$};
				\draw (180.5,28) node [anchor=north west][inner sep=0.75pt]    {$\Omega_{2}^{v}$};
				\draw (330.5,28) node [anchor=north west][inner sep=0.75pt]    {$\Omega_{3}^{v}$};
				\draw (110.5,118) node [anchor=north west][inner sep=0.75pt]    {$\Omega_{3}^{m}$};
				\draw (310.5,120) node [anchor=north west][inner sep=0.75pt]    {$\Omega_{v}$};
			\end{tikzpicture}
			\caption{Virtual Reidemeister moves $\Omega_1^v$, $\Omega_2^v$, $\Omega_3^v$, the mixed move $\Omega_3^{m}$ and $\Omega_{v}$ move.} \label{fig4}
		\end{center}
	\end{figure}
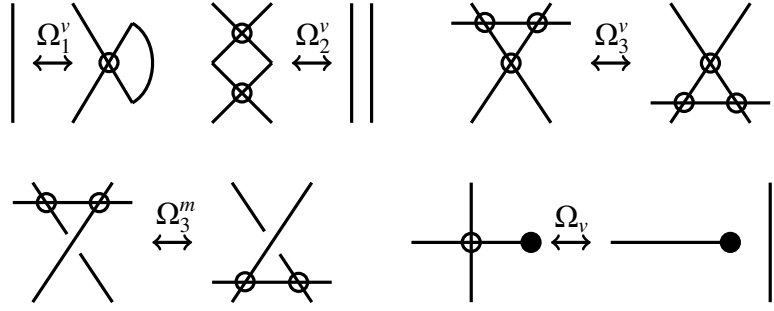
	
	A \textit{crossing change} is a local move that swaps the under-arc and over-arc at a single crossing of a diagram as presented in Figure~\ref{fig5}. 
	\begin{figure}[htbp]
		\begin{center}
			\tikzset{every picture/.style={line width=0.75pt}}         
			\begin{tikzpicture}[x=0.75pt,y=0.75pt,yscale=-1,xscale=1]
				\draw [black, very thick]  (80,490.67) -- (139.67,430) ;
				\draw [black, very thick]  (105,455.67) -- (80,430.67) ; 
				\draw [black, very thick]  (115,465.67) -- (140,490.67) ;
				\draw [black, very thick]  (310,490) -- (250,430) ;
				\draw [black, very thick]  (285,455) -- (310,430) ;
				\draw [black, very thick]  (275,465) -- (250,490) ;
				\draw [black, very thick,<->]  (170,460) -- (210,460) ;
			\end{tikzpicture}
			\caption{A crossing change.} \label{fig5}
		\end{center} 
	\end{figure}
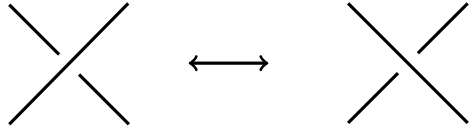
	
	\begin{definition}\label{def2.6}{\rm
			Let $K$ and $K'$ be planar virtual knotoids. $K$ and $K'$ are said to be \textit{homotopic} if a diagram of one can be transformed into a diagram of the other via a finite sequence of crossing changes, generalized Reidemeister moves and planar isotopies.}
	\end{definition}
	
	Planar knotoids possess multiple homotopy classes. For example, according to~\cite{FLV25}, one can see that $\chi_1$ and $\chi_2$ whose diagrams are presented in Figure~\ref{fig6}, are homotopic, but neither is homotopic to the trivial knotoid. They correspond to two of the four unifoils from Figure 1 in~\cite{Tu12}, where $\chi_2 = U$ and $\chi_1 = \operatorname{mir}(U)$, where $\operatorname{mir}(U)$ denotes taking of the mirror image of $U$.  	
	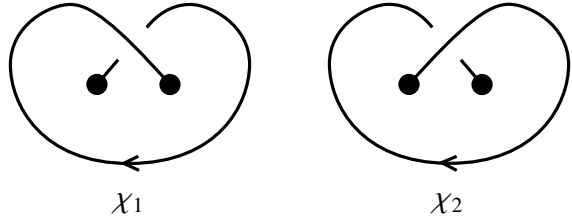
\begin{figure}[htbp]
		\begin{center}
			\tikzset{every picture/.style={line width=0.75pt}}       
			\begin{tikzpicture}[x=0.75pt,y=0.75pt,yscale=-1,xscale=1]
				\draw [black, very thick]    (78.36,71.74) .. controls (82.76,66.14) and (90.16,60.14) .. (100.16,60.14) .. controls (110.16,60.14) and (130.16,70.34) .. (129.96,90.34) .. controls (129.76,110.34) and (114.39,139.71) .. (69.96,139.94) .. controls (25.53,140.17) and (9.76,110.14) .. (9.96,90.14) .. controls (10.16,70.14) and (29.96,60.14) .. (39.96,60.14) .. controls (49.96,60.14) and (62.56,69.54) .. (90.04,100.34) ;
				\draw [black, very thick]    (53.38,100.27) .. controls (57.68,95.17) and (59.74,92.95) .. (64,87.88) ;
				\draw  [fill={rgb, 255:red, 0; green, 0; blue, 0 }  ,fill opacity=1 ][black, very thick]  (85.52,100.34) .. controls (85.52,97.84) and (87.54,95.82) .. (90.04,95.82) .. controls (92.54,95.82) and (94.56,97.84) .. (94.56,100.34) .. controls (94.56,102.84) and (92.54,104.86) .. (90.04,104.86) .. controls (87.54,104.86) and (85.52,102.84) .. (85.52,100.34) -- cycle ;
				\draw  [fill={rgb, 255:red, 0; green, 0; blue, 0 }  ,fill opacity=1 ][black, very thick]  (48.86,100.27) .. controls (48.86,97.77) and (50.88,95.75) .. (53.38,95.75) .. controls (55.88,95.75) and (57.9,97.77) .. (57.9,100.27) .. controls (57.9,102.77) and (55.88,104.79) .. (53.38,104.79) .. controls (50.88,104.79) and (48.86,102.77) .. (48.86,100.27) -- cycle ;
				\draw [black, very thick]    (221.61,71.67) .. controls (217.2,66.07) and (209.8,60.07) .. (199.8,60.07) .. controls (189.8,60.08) and (169.81,70.28) .. (170.01,90.28) .. controls (170.22,110.28) and (185.6,139.65) .. (230.03,139.86) .. controls (274.46,140.07) and (290.22,110.04) .. (290.01,90.04) .. controls (289.8,70.04) and (270,60.05) .. (260,60.05) .. controls (250,60.06) and (237.4,69.46) .. (209.94,100.27) ;
				\draw [black, very thick]    (246.6,100.34) .. controls (242.3,95.24) and (240.24,93.02) .. (235.98,87.95) ;
				\draw  [fill={rgb, 255:red, 0; green, 0; blue, 0 }  ,fill opacity=1 ][black, very thick]  (242.08,100.34) .. controls (242.08,97.84) and (244.1,95.82) .. (246.6,95.82) .. controls (249.09,95.82) and (251.12,97.84) .. (251.12,100.34) .. controls (251.12,102.84) and (249.09,104.86) .. (246.6,104.86) .. controls (244.1,104.86) and (242.08,102.84) .. (242.08,100.34) -- cycle ;
				\draw  [fill={rgb, 255:red, 0; green, 0; blue, 0 }  ,fill opacity=1 ][black, very thick]  (205.42,100.27) .. controls (205.42,97.77) and (207.44,95.75) .. (209.94,95.75) .. controls (212.43,95.75) and (214.46,97.77) .. (214.46,100.27) .. controls (214.46,102.77) and (212.43,104.79) .. (209.94,104.79) .. controls (207.44,104.79) and (205.42,102.77) .. (205.42,100.27) -- cycle ;
				\draw  [black, very thick]  (233.88,143.44) -- (226.36,139.53) -- (233.85,136.34) ;
				\draw  [black, very thick]  (74.51,143.42) -- (66.8,139.88) -- (74.13,136.32) ;
				\node at (70,160) {$\chi_1$};
				\node at (230,160) {$\chi_2$};
		\end{tikzpicture}
			\caption{The diagrams of knotoids $\chi_1$ and $\chi_2$.}  \label{fig6}
		\end{center}
	\end{figure}	
	
	\begin{definition}\label{def2.7}{\rm
			Let $K$ and $K'$ be homotopic planar virtual knotoids. The minimal number of crossing changes required to transform a diagram of $K$ into a diagram of $K'$, where the minimum is taken over all diagrams of $K$ and $K'$, is called the \textit{Gordian distance} between $K$ and $K'$, denoted by $d_G(K, K')$. If $K$ is homotopic to the trivial planar knotoid $O$, then $d_G(K, O)$ is called the \textit{unknotting number} of $K$.}
	\end{definition}
	
	In 1989, Murakami and Nakanishi~\cite{MN89} proposed the notion of the $\Delta$-\textit{move} on link diagrams and proved that $\Delta$-move is an unknotting operation for any classical knot. Although any knot can be unknotted by $\Delta$-moves, this does not hold for knotoids. A $\Delta$-move on a knotoid diagram is the replacement, in a disk disjoint from the endpoints and all other crossings, of one of the local three-strand tangles in Figure~\ref{fig7} by the other; all oriented variants are allowed. 
	
	\begin{definition}\label{def2.8}{\rm
			Let $K$ and $K'$ be planar virtual knotoids. $K$ and $K'$ are said to be $\Delta$-\textit{homotopic} if a diagram of one can be transformed into a diagram of the other via a finite sequence of $\Delta$-moves, generalized Reidemeister moves and planar isotopies.}
	\end{definition}
	
	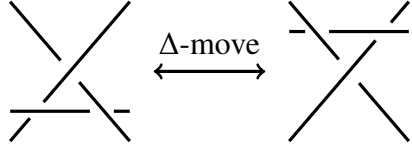
\begin{figure}[htbp]
		\begin{center}
			\tikzset{every picture/.style={line width=0.75pt}} 
			\begin{tikzpicture}[x=0.75pt,y=0.75pt,yscale=-1,xscale=1]
				\draw  [black, very thick]  (89.88,288.47) -- (80,300) ;
				\draw  [black, very thick]  (80,230) -- (105.27,259.48) ;
				\draw  [black, very thick]  (132,285) -- (140,285) ;
				\draw  [black, very thick,<->]  (152,265) -- (208,265) ;
				\draw  [black, very thick]  (220,245) -- (228,245) ;
				\draw  [black, very thick]  (220,230) -- (244.88,259.02) ;
				\draw  [black, very thick]  (220,300) -- (263.33,249.45) ;
				\draw  [black, very thick]  (240,245) -- (280,245) ;
				\draw  [black, very thick]  (80,285) -- (120,285) ;
				\draw  [black, very thick]  (114.62,270.39) -- (140,300) ;
				\draw  [black, very thick]  (140,230) -- (96.77,280.44) ;
				\draw  [black, very thick]  (271.01,240.49) -- (280,230) ;
				\draw  [black, very thick]  (255.03,270.87) -- (280,300) ;
				\node at (180,251) {$\Delta$-move};
			\end{tikzpicture}
			\caption{A $\Delta$-move.}  \label{fig7}
		\end{center}
	\end{figure}
	
	\begin{definition}\label{def2.9}{\rm
			Let $K$ and $K'$ be $\Delta$-homotopic planar virtual knotoids. The $\Delta$-\textit{Gordian  distance} between $K$ and $K'$, denoted by $d_\Delta(K, K')$, is the minimal number of $\Delta$-moves needed to convert any diagram of $K$ into a diagram of $K'$, where the minimum is taken over all diagrams of $K$ and~$K'$.}
	\end{definition}
	
	Shimizu~\cite{Sh14} introduced the notion of region crossing changes on knot diagrams and demonstrated that this local move serves as an unknotting operation for knots. A $region$ of a knot diagram is a connected component of the complement in $\mathbb{R}^2$ of the knot diagram, where each crossing has been replaced by a vertex. A \textit{region crossing change}, shortly, RCC-move, at a region $R$ is defined as reversing every crossing lying on the boundary $\partial R$ of $R$. A straightforward observation is that applying RCC-move twice at the same region $R$ in a diagram $D$ yields $D$ itself. Kadokami~\cite{Ka12} introduced the \textit{diagrammatic region crossing change} for virtual link diagrams, which consists of performing a crossing change on every classical crossing in a given region while leaving virtual crossings unchanged. We generalize this operation to planar virtual knotoids and define the \textit{virtual region crossing change}, or, shortly, VRCC-move.
	
	Let $D$ denote a virtual knotoid diagram on $\mathbb{R}^2$. Let $|D|$ be the $4$-valent graph constructed by replacing every classical crossing and virtual crossing in $D$ with a vertex. Each connected component of $\mathbb{R}^2 \setminus |D|$ is referred to as a \textit{region} of $D$.
	
	\begin{definition}{\rm
			A \textit{virtual region crossing change}, or shortly, \textit{VRCC-move} at a region $R$ of a planar virtual knotoid diagram is defined to be performing one crossing change at each distinct classical crossing incident to} the boundary $\partial R$ of $R$, while virtual crossings remain unaltered. 
	\end{definition}
	
	As shown in Figure~\ref{fig28}, applying a single VRCC-move at region $R$ transforms a planar virtual knotoid diagram $D$ into a diagram $D'$.
	
	\begin{definition}\label{def2.11}{\rm
			Let $K$ and $K'$ be planar virtual knotoids. $K$ and $K'$ are said to be VRCC-\textit{homotopic} if the diagram of one can be transformed into that of the other by a finite sequence of VRCC-moves, generalized Reidemeister moves and planar isotopies.}
	\end{definition}
	
	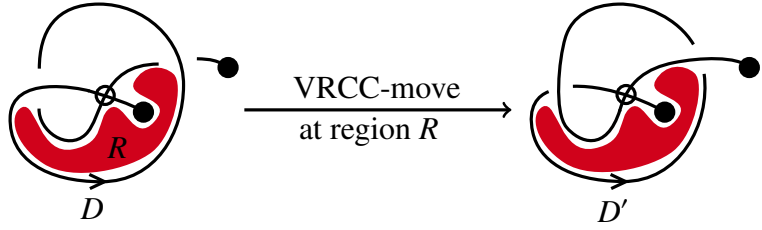
\begin{figure}[htbp]
		\begin{center}
			\tikzset{every picture/.style={line width=0.75pt}}
			\begin{tikzpicture}[x=0.75pt,y=0.75pt,yscale=-1,xscale=1]
				\draw [black, very thick]    (368.25,83.75) .. controls (372.29,68.27) and (382.25,62) .. (395.33,59.33) .. controls (408.42,56.67) and (427.5,55.83) .. (433.35,61.75) ; 
				\draw [color={rgb, 255:red, 0; green, 0; blue, 0 }  ,draw opacity=1 ][black, very thick]    (101.58,73.33) .. controls (108.17,75.25) and (125.33,80.58) .. (129.38,84.13) ;
				\draw [black, very thick]    (368.5,82.83) .. controls (361.58,105) and (347.83,99.25) .. (341.58,91.67) .. controls (335.33,84.08) and (337.67,58.08) .. (341,46.92) .. controls (344.33,35.75) and (354.25,30.33) .. (367.5,29.92) .. controls (380.75,29.5) and (397.58,37.08) .. (405.33,50.17) ;
				\draw [black, very thick]    (410.67,64.92) .. controls (413.94,90.35) and (398.03,119.17) .. (369.9,118.95) .. controls (341.76,118.72) and (327.74,107.89) .. (324.74,94.43) .. controls (321.74,80.98) and (327.67,75.75) .. (332.67,73.83) ;
				\draw  [black, very thick]  (363.74,115.4) -- (371.41,119.03) -- (364.04,122.5) ;
				\draw [color={rgb, 255:red, 0; green, 0; blue, 0 }  ,draw opacity=1 ][black, very thick]    (344.83,71.17) .. controls (365.25,70.5) and (387.27,81.58) .. (390.9,84.13) ;
				\draw  [fill={rgb, 255:red, 0; green, 0; blue, 0 }  ,fill opacity=1 ][black, very thick]  (428.83,61.75) .. controls (428.83,59.25) and (430.86,57.23) .. (433.35,57.23) .. controls (435.85,57.23) and (437.87,59.25) .. (437.87,61.75) .. controls (437.87,64.25) and (435.85,66.27) .. (433.35,66.27) .. controls (430.86,66.27) and (428.83,64.25) .. (428.83,61.75) -- cycle ;
				\draw  [fill={rgb, 255:red, 0; green, 0; blue, 0 }  ,fill opacity=1 ][black, very thick]  (386.38,84.13) .. controls (386.38,81.63) and (388.4,79.61) .. (390.9,79.61) .. controls (393.39,79.61) and (395.42,81.63) .. (395.42,84.13) .. controls (395.42,86.62) and (393.39,88.65) .. (390.9,88.65) .. controls (388.4,88.65) and (386.38,86.62) .. (386.38,84.13) -- cycle ; 
				\draw [black, very thick]    (77,82.35) .. controls (76.69,93.69) and (91.18,103.76) .. (100.24,95.41) .. controls (109.29,87.06) and (107.18,60.12) .. (137,60) ;
				\draw [black, very thick]    (77.13,62.3) .. controls (77.16,43.74) and (85.01,29.92) .. (107.18,29.91) .. controls (129.35,29.9) and (149.63,47.25) .. (149.5,70.75) .. controls (149.38,94.25) and (135.75,119) .. (108.36,119.06) .. controls (80.97,119.12) and (61.75,107.88) .. (62.5,86.63) .. controls (63.25,65.38) and (96.39,71.91) .. (102.25,73.5) ; 
				\draw [black, very thick]    (156.17,57.35) .. controls (163,57.5) and (166.75,58.08) .. (171.83,61.75) ;
				\draw  [fill={rgb, 255:red, 0; green, 0; blue, 0 }  ,fill opacity=1 ][black, very thick]  (167.31,61.75) .. controls (167.31,59.25) and (169.34,57.23) .. (171.83,57.23) .. controls (174.33,57.23) and (176.35,59.25) .. (176.35,61.75) .. controls (176.35,64.25) and (174.33,66.27) .. (171.83,66.27) .. controls (169.34,66.27) and (167.31,64.25) .. (167.31,61.75) -- cycle ;
				\draw  [fill={rgb, 255:red, 0; green, 0; blue, 0 }  ,fill opacity=1 ][black, very thick]  (124.85,84.13) .. controls (124.85,81.63) and (126.88,79.6) .. (129.38,79.6) .. controls (131.87,79.6) and (133.9,81.63) .. (133.9,84.13) .. controls (133.9,86.62) and (131.87,88.65) .. (129.38,88.65) .. controls (126.88,88.65) and (124.85,86.62) .. (124.85,84.13) -- cycle ;
				\draw  [black, very thick]  (102.22,115.4) -- (109.88,119.03) -- (102.52,122.5) ;
				\draw  [color={rgb, 255:red, 208; green, 2; blue, 27 }  ,draw opacity=1 ][fill={rgb, 255:red, 208; green, 2; blue, 27 }  ,fill opacity=1 ] (68.5,82.25) .. controls (75,78.75) and (76.71,106.73) .. (94.13,103.25) .. controls (111.54,99.77) and (105.98,87.48) .. (112.67,82.67) .. controls (119.36,77.85) and (124.25,101.75) .. (135.33,89.08) .. controls (146.42,76.42) and (120.75,77.17) .. (125.25,69) .. controls (129.75,60.83) and (139,60.75) .. (144.08,65.33) .. controls (149.17,69.92) and (144.42,94.83) .. (133.58,103.58) .. controls (122.75,112.33) and (93.92,119) .. (79.38,109.5) .. controls (64.83,100) and (62,85.75) .. (68.5,82.25) -- cycle ; 
				\draw  [black, very thick]  (105.67,75.75) .. controls (105.67,73.25) and (107.69,71.23) .. (110.19,71.23) .. controls (112.69,71.23) and (114.71,73.25) .. (114.71,75.75) .. controls (114.71,78.25) and (112.69,80.27) .. (110.19,80.27) .. controls (107.69,80.27) and (105.67,78.25) .. (105.67,75.75) -- cycle ;
				\draw  [black, very thick]  (366.98,75.29) .. controls (366.98,72.79) and (369,70.77) .. (371.5,70.77) .. controls (373.99,70.77) and (376.02,72.79) .. (376.02,75.29) .. controls (376.02,77.79) and (373.99,79.81) .. (371.5,79.81) .. controls (369,79.81) and (366.98,77.79) .. (366.98,75.29) -- cycle ; 
				\draw  [color={rgb, 255:red, 208; green, 2; blue, 27 }  ,draw opacity=1 ][fill={rgb, 255:red, 208; green, 2; blue, 27 }  ,fill opacity=1 ] (329.92,81.85) .. controls (336.42,78.35) and (338.12,106.33) .. (355.54,102.85) .. controls (372.96,99.37) and (367.39,87.08) .. (374.08,82.27) .. controls (380.77,77.45) and (385.67,101.35) .. (396.75,88.68) .. controls (407.83,76.02) and (382.17,76.77) .. (386.67,68.6) .. controls (391.17,60.43) and (400.42,60.35) .. (405.5,64.93) .. controls (410.58,69.52) and (405.83,94.43) .. (395,103.18) .. controls (384.17,111.93) and (355.33,118.6) .. (340.79,109.1) .. controls (326.25,99.6) and (323.42,85.35) .. (329.92,81.85) -- cycle ;
				\draw  [black, very thick,->]  (180,83) -- (314,83) ;
				\draw (203,65) node [anchor=north west][inner sep=0.75pt]   [align=left] {VRCC-move};
				\draw (206,85) node [anchor=north west][inner sep=0.75pt]   [align=left] {at region $R$};
				\draw (96,126) node [anchor=north west][inner sep=0.75pt]    {$D$};
				\draw (355.5,126) node [anchor=north west][inner sep=0.75pt]    {$D'$};
				\draw (110,95) node [anchor=north west][inner sep=0.75pt]   [align=left] {$\displaystyle R$};
			\end{tikzpicture}
			\caption{$D'$ obtained from $D$ by the VRCC-move at region $R$.}  \label{fig28}
		\end{center}
	\end{figure}
	
	\begin{definition}\label{def2.12}{\rm
			Let $K$ and $K'$ be VRCC-homotopic planar virtual knotoids. The VRCC-\textit{Gordian distance} between $K$ and $K'$, denoted by $d_\text{VRCC}(K, K')$, is defined as the minimal number of VRCC-moves required to deform a diagram of $K$ into a diagram of $K'$, where the minimum is taken over all diagrams of $K$ and $K'$.}
	\end{definition}
	
	Gauss diagrams provide a concise combinatorial representation of knotoid diagrams by capturing all crossing interactions through directed chords. Both classical knotoids and virtual knotoids can be represented by means of Gauss diagrams, see~\cite{GPV00, PV94}.
	
	The Gauss diagram $G(D)$ of a knotoid diagram $D$ is an counterclockwise-oriented arc, with chords connecting the preimages of crossings in $D$. The starting point of $G(D)$ is called its \textit{tail}, and the endpoint is called its \textit{head}. Each chord in $G(D)$ corresponds bijectively to a crossing $c \in D$, and is also denoted by $c$. The orientation of a chord points from the over-arc to the under-arc of the corresponding crossing. We further mark the initial point of each chord $c$ by $\operatorname{sgn}(c)$, as shown in Figure~\ref{fig8}.
	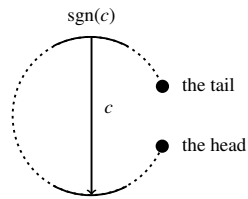
\begin{figure}[htbp] 
		\centering
		\scalebox{0.6}{
			\tikzset{every picture/.style={line width=0.75pt}}
			\begin{tikzpicture}[x=0.75pt,y=0.75pt,yscale=-1,xscale=1]
				\draw  [draw opacity=0][dash pattern={on 1.69pt off 2.76pt}][black, very thick]  (189.12,131.19) .. controls (179.37,155.19) and (156.14,172.08) .. (129.03,172.08) .. controls (93.13,172.08) and (64.03,142.46) .. (64.03,105.93) .. controls (64.03,69.39) and (93.13,39.78) .. (129.03,39.78) .. controls (156.31,39.78) and (179.67,56.88) .. (189.31,81.13) -- (129.03,105.93) -- cycle ; \draw  [dash pattern={on 1.69pt off 2.76pt}][black, very thick]  (189.12,131.19) .. controls (179.37,155.19) and (156.14,172.08) .. (129.03,172.08) .. controls (93.13,172.08) and (64.03,142.46) .. (64.03,105.93) .. controls (64.03,69.39) and (93.13,39.78) .. (129.03,39.78) .. controls (156.31,39.78) and (179.67,56.88) .. (189.31,81.13) ;   
				\draw  [fill={rgb, 255:red, 0; green, 0; blue, 0 }  ,fill opacity=1 ][line width=1.5]  (184.6,131.19) .. controls (184.6,128.69) and (186.62,126.67) .. (189.12,126.67) .. controls (191.62,126.67) and (193.64,128.69) .. (193.64,131.19) .. controls (193.64,133.68) and (191.62,135.71) .. (189.12,135.71) .. controls (186.62,135.71) and (184.6,133.68) .. (184.6,131.19) -- cycle ;
				\draw  [fill={rgb, 255:red, 0; green, 0; blue, 0 }  ,fill opacity=1 ][line width=1.5]  (184.79,81.13) .. controls (184.79,78.63) and (186.81,76.61) .. (189.31,76.61) .. controls (191.8,76.61) and (193.83,78.63) .. (193.83,81.13) .. controls (193.83,83.63) and (191.8,85.65) .. (189.31,85.65) .. controls (186.81,85.65) and (184.79,83.63) .. (184.79,81.13) -- cycle ;
				\draw  [black, very thick,->]  (130,40) -- (130,172) ;
				\draw  [draw opacity=0][black, very thick]  (98.03,47.51) .. controls (107.26,42.59) and (117.82,39.8) .. (129.04,39.8) .. controls (140.47,39.8) and (151.21,42.7) .. (160.55,47.79) -- (129.04,104.36) -- cycle ; \draw  [black, very thick]  (98.03,47.51) .. controls (107.26,42.59) and (117.82,39.8) .. (129.04,39.8) .. controls (140.47,39.8) and (151.21,42.7) .. (160.55,47.79) ;  
				\draw  [draw opacity=0][black, very thick]  (159.46,164.69) .. controls (150.45,169.48) and (140.16,172.2) .. (129.24,172.2) .. controls (117.88,172.2) and (107.2,169.26) .. (97.94,164.1) -- (129.24,107.9) -- cycle ; \draw  [black, very thick]  (159.46,164.69) .. controls (150.45,169.48) and (140.16,172.2) .. (129.24,172.2) .. controls (117.88,172.2) and (107.2,169.26) .. (97.94,164.1) ;  
				\node at (130,36) [above] {\text{sgn}($c$)};
				\node at (135,100) [right] {$c$};
				\node at (200,80) [right] {the tail};
				\node at (200,130) [right]{the head};
		\end{tikzpicture}}
		\caption{Elements of a Gauss diagram $G(D)$.}
		\label{fig8}
	\end{figure}
	
	Specifically, the Gauss diagram of a virtual knotoid is constructed in the same way as for a classical knotoid with the modification that all virtual crossings are ignored. We can consider a weak equivalence relation for virtual knotoids related to string homotopy in thickened surfaces. Note that the endpoints of all chords lie on the arc between the tail and head of the Gauss diagram, rather than being distributed over the entire circle. This is the key distinction between the Gauss diagram of a virtual knotoid and that of a virtual knot.
	
	The three Reidemeister moves $\mathcal{R} = \{ \Omega_1, \Omega_2, \Omega_3 \}$ on Gauss diagrams are illustrated in Figures~\ref{fig25},~\ref{fig29} and~\ref{fig30}, while the crossing change and $\Delta$-move on Gauss diagrams are shown in Figures~\ref{fig9} and~\ref{fig10}, respectively.
	\begin{figure}[htbp]
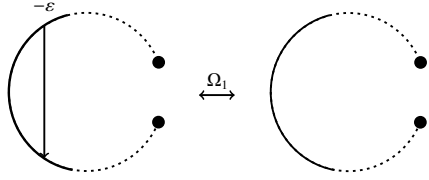

		\centering
		\scalebox{0.6}{
			\tikzset{every picture/.style={line width=0.75pt}}       
}
		\caption{Reidemeister move $\Omega_1$ on a Gauss diagram.} \label{fig25}
	\end{figure}
	
	\begin{figure}[htbp]
		\centering
		\scalebox{0.6}{
			\tikzset{every picture/.style={line width=0.75pt}}        
			%
}
		\caption{Reidemeister move $\Omega_2$ on Gauss diagrams.}
		\label{fig29}
	\end{figure}
	\begin{figure}[htbp]
		\begin{subfigure}{0.54\textwidth}
			\centering
			\scalebox{0.6}{
				\tikzset{every picture/.style={line width=0.75pt}}
				%
}
		\end{subfigure}
		\caption{Reidemeister move $\Omega_3$ on Gauss diagrams.}
		\label{fig30}
	\end{figure}
	\begin{figure}[htbp]
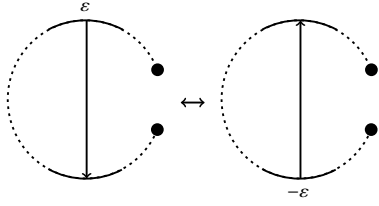

		\centering
		\scalebox{0.6}{
			\tikzset{every picture/.style={line width=0.75pt}}        
			%
}
		\caption{A crossing change on a Gauss diagram.}
		\label{fig9}
	\end{figure}
	
	\begin{figure}[htbp]
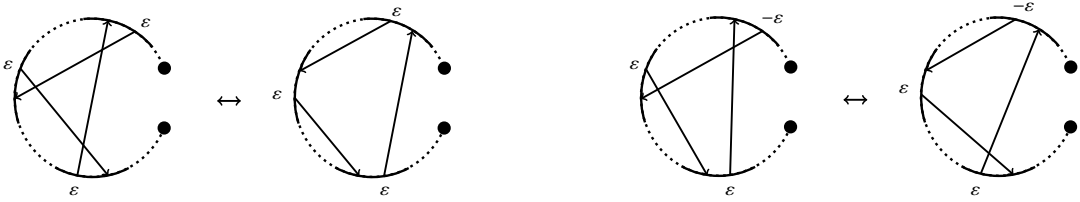

		\begin{center}
			\begin{subfigure}{0.45\textwidth}
				\centering
				\scalebox{0.6}{
					\tikzset{every picture/.style={line width=0.75pt}}         
					%
}				
			\end{subfigure}
			\caption{$\Delta$-moves on Gauss diagrams.}
			\label{fig10}
		\end{center}
	\end{figure}
	
	The \textit{index} $\Ind(c)$, defined for virtual knots above, was used in~\cite{Ch16} and~\cite{Hi25}, and also referred to as the degree of a chord in~\cite{Je23}. We extend this definition to virtual knotoids.
	
	\begin{definition}{\rm
			Let $C(G(D))$ be the set of chords in the Gauss diagram $G(D)$ of a virtual knotoid diagram $D$, we denote by $\ell(c) = \{\ell_1(c), \ell_2(c), \dots, \ell_m(c)\}$ the set of chords that cross $c$ and point toward the left-hand side of $c$, and by $r(c) = \{r_1(c), r_2(c), \dots, r_n(c)\}$ the set of chords that cross $c$ and point toward the right-hand side of $c$, see Figure~\ref{fig33}.
			\begin{figure}[htbp]
				\begin{subfigure}{0.45\textwidth}
					\raggedleft
					\scalebox{0.6}{
						\tikzset{every picture/.style={line width=0.75pt}}
						\begin{tikzpicture}[x=0.75pt,y=0.75pt,yscale=-1,xscale=1] 
							\draw  [draw opacity=0][black, very thick]  (615.4,364.25) .. controls (605.65,388.25) and (582.42,405.14) .. (555.31,405.14) .. controls (519.41,405.14) and (490.31,375.52) .. (490.31,338.99) .. controls (490.31,302.45) and (519.41,272.84) .. (555.31,272.84) .. controls (582.59,272.84) and (605.95,289.94) .. (615.59,314.19) -- (555.31,338.99) -- cycle ; \draw  [black, very thick]  (615.4,364.25) .. controls (605.65,388.25) and (582.42,405.14) .. (555.31,405.14) .. controls (519.41,405.14) and (490.31,375.52) .. (490.31,338.99) .. controls (490.31,302.45) and (519.41,272.84) .. (555.31,272.84) .. controls (582.59,272.84) and (605.95,289.94) .. (615.59,314.19) ;
							\draw  [fill={rgb, 255:red, 0; green, 0; blue, 0 }  ,fill opacity=1 ][black, very thick]  (610.88,364.25) .. controls (610.88,361.75) and (612.91,359.73) .. (615.4,359.73) .. controls (617.9,359.73) and (619.92,361.75) .. (619.92,364.25) .. controls (619.92,366.74) and (617.9,368.77) .. (615.4,368.77) .. controls (612.91,368.77) and (610.88,366.74) .. (610.88,364.25) -- cycle ; 
							\draw  [fill={rgb, 255:red, 0; green, 0; blue, 0 }  ,fill opacity=1 ][black, very thick]  (611.07,314.19) .. controls (611.07,311.69) and (613.09,309.67) .. (615.59,309.67) .. controls (618.09,309.67) and (620.11,311.69) .. (620.11,314.19) .. controls (620.11,316.69) and (618.09,318.71) .. (615.59,318.71) .. controls (613.09,318.71) and (611.07,316.69) .. (611.07,314.19) -- cycle ;
							\draw  [black, very thick, ->]  (500.45,303.27) -- (593.91,285.82) ;
							\draw  [black, very thick, <-]  (491.18,349.82) -- (590.45,394.73) ;
							\draw  [black, very thick, <-]  (531.27,277.58) -- (554.09,405.09) ;
							\draw (543.64,324.36) node [anchor=north west][inner sep=0.75pt]    {$c$};
							\draw (602,277) node [anchor=north west][inner sep=0.75pt]    {$r_{i}( c)$};
							\draw (458,339) node [anchor=north west][inner sep=0.75pt]    {$\ell_{j}( c)$};
					\end{tikzpicture}}			
				\end{subfigure}
				\hfill
				\begin{subfigure}{0.45\textwidth}
					\raggedright
					\scalebox{0.6}{
						\tikzset{every picture/.style={line width=0.75pt}}
						\begin{tikzpicture}[x=0.75pt,y=0.75pt,yscale=-1,xscale=1]
							\draw  [draw opacity=0][black, very thick]  (537.85,161.8) .. controls (528.09,185.8) and (504.86,202.69) .. (477.75,202.69) .. controls (441.86,202.69) and (412.75,173.07) .. (412.75,136.54) .. controls (412.75,100.01) and (441.86,70.39) .. (477.75,70.39) .. controls (505.04,70.39) and (528.39,87.5) .. (538.03,111.74) -- (477.75,136.54) -- cycle ; \draw  [black, very thick]  (537.85,161.8) .. controls (528.09,185.8) and (504.86,202.69) .. (477.75,202.69) .. controls (441.86,202.69) and (412.75,173.07) .. (412.75,136.54) .. controls (412.75,100.01) and (441.86,70.39) .. (477.75,70.39) .. controls (505.04,70.39) and (528.39,87.5) .. (538.03,111.74) ;
							\draw  [fill={rgb, 255:red, 0; green, 0; blue, 0 }  ,fill opacity=1 ][black, very thick]  (533.33,161.8) .. controls (533.33,159.31) and (535.35,157.28) .. (537.85,157.28) .. controls (540.34,157.28) and (542.37,159.31) .. (542.37,161.8) .. controls (542.37,164.3) and (540.34,166.32) .. (537.85,166.32) .. controls (535.35,166.32) and (533.33,164.3) .. (533.33,161.8) -- cycle ;
							\draw  [fill={rgb, 255:red, 0; green, 0; blue, 0 }  ,fill opacity=1 ][black, very thick]  (533.51,111.74) .. controls (533.51,109.25) and (535.54,107.22) .. (538.03,107.22) .. controls (540.53,107.22) and (542.55,109.25) .. (542.55,111.74) .. controls (542.55,114.24) and (540.53,116.26) .. (538.03,116.26) .. controls (535.54,116.26) and (533.51,114.24) .. (533.51,111.74) -- cycle ;
							\draw  [black, very thick, ->]  (444.5,80) -- (510.56,80) ; 
							\draw  [black, very thick, ->]  (423.89,100) -- (532.11,100) ; 
							\draw  [black, very thick, <-]  (421.89,170) -- (533.89,170) ; 
							\draw  [black, very thick, <-]  (440.11,190) -- (516.11,190) ; 
							\draw  [black, very thick, <-]  (480,70.78) -- (480,202.44) ;
							\draw (426,75) node [anchor=north west][inner sep=0.75pt]  [font=\scriptsize]  {$+$};
							\draw (408,96) node [anchor=north west][inner sep=0.75pt]  [font=\scriptsize]  {$-$};
							\draw (482.44,118.67) node [anchor=north west][inner sep=0.75pt]    {$c$};
							\draw (540,166.5) node [anchor=north west][inner sep=0.75pt]  [font=\scriptsize]  {$+$};
							\draw (525,187) node [anchor=north west][inner sep=0.75pt]  [font=\scriptsize]  {$-$};
							\draw (526.67,72) node [anchor=north west][inner sep=0.75pt]    {$r^{+}( c)$};
							\draw (542.44,92) node [anchor=north west][inner sep=0.75pt]    {$r^{-}( c)$};
							\draw (380,162) node [anchor=north west][inner sep=0.75pt]    {$\ell^+(c)$};
							\draw (396.44,182) node [anchor=north west][inner sep=0.75pt]    {$\ell^-(c)$};		
					\end{tikzpicture}}
				\end{subfigure}
				\caption{Chords cross the chord $c$.}
				\label{fig33}
			\end{figure}
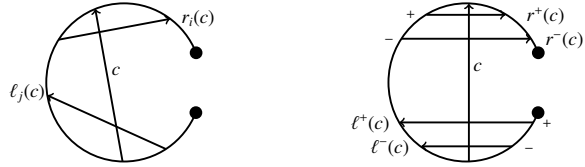
			
			Given a virtual knotoid diagram $D$ and a fixed chord $c \in C(D)$, we define the following sets:
			$$
			\begin{aligned}
				r^+(c) &= \{c' \in r(c) \mid \operatorname{sgn}(c') = 1\}, \\
				r^-(c) &= \{c' \in r(c) \mid \operatorname{sgn}(c') = -1\}, \\
				\ell^+(c) &= \{c' \in \ell(c) \mid \operatorname{sgn}(c') = 1\}, \\
				\ell^-(c) &= \{c' \in \ell(c) \mid \operatorname{sgn}(c') = -1\}.
			\end{aligned}
			$$ 
			The index (degree) of $c$ is defined by equation:
			$$
			\Ind(c) = |r^+(c)| - |r^-(c)| - |\ell^+(c)| + |\ell^-(c)|,
			$$
			where $|X|$ denotes the cardinality of a set $X$. In particular, if a chord $c$ is isolated, meaning that it has no intersection with any other chord, then $\Ind(c)=0$.
		}
	\end{definition}
	
	\begin{definition}{\rm
			Let $K$ be a planar virtual knotoid, $D$ a diagram of $K$, and $G$ the Gauss diagram corresponding to $D$. The \textit{index} of any crossing in $D$ equals the index of the corresponding chord in $G$.}
	\end{definition}
	
	In~\cite{FLV25}, a three-variable transcendental invariant of planar knotoids was defined. The invariant is given by: 
	$$
	H_D(t, y, z) = \sum_{\substack{c \in C(G(D)) \\ n \in \mathbb{N}}} \operatorname{sgn}(c) \left( t^{\operatorname{Ind}_c^n(z)} - 1 \right) y^n.
	$$
	
	This invariant has the following property.   
	\begin{theorem}~\cite{FLV25}\label{thm-H}
		Let $K$ and $K'$ be two homotopic planar knotoids. Then 
		$$
		H_K(t, y, z) - H_{K'}(t, y, z) = \sum_{n \in \mathbb{N}} \left( \sum_{m \in \mathbb{N}} {a_{n_m}} \left( t^{z^m_n} + t^{-(z^{-1})^m_n} - 2 \right) \right) y^n,
		$$
		where $a_{n_m} \in \mathbb{Z}$. For all $n \in \mathbb{N}$, the Gordian distance satisfies
		$$
		d_G(K, K') \geq \sum_{m \in \mathbb{N}} |a_{n_m}|.
		$$
	\end{theorem}
	
	\section{R-covering: proof of Theorem~\ref{thm-invc}.} \label{sec3}
	
	In this section, we define the $r$-covering of planar virtual knotoids and present the proof of Theorem~\ref{thm-invc}.
	
	Turaev~\cite{Tu04} first introduced the concept of the $r$-covering of flat virtual knots. In $2020$, Nakamura, Nakanishi and Satoh~\cite{NNS20} generalized $r$-coverings to virtual knots. We then extend the concept to planar virtual knotoids.
	
	\begin{definition}\label{D^{(r)}}{\rm
			Let $D$ be a diagram of a planar virtual knotoid $K$ and $r\in \mathbb Z_+$. Then the $r$-covering of $D$, denoted by $D^{(r)}$, is the diagram obtained from $D$ by virtualizing all crossings whose indices are not divisible by $r$, where virtualization is the operation that converts a classical crossing to a virtual crossing}. 
	\end{definition}
	
	\begin{remark}{\rm
			Let $G$ be a Gauss diagram of a planar virtual knotoid $K$, and $r \in \mathbb Z_+$. Then $G^{(r)}$, the Gauss diagram of $D^{(r)}$, is obtained from $G$ by removing all chords with indices are not dividible by $r$.}
	\end{remark}
	
	\begin{proof}[Proof of Theorem~\ref{thm-invc}.]
		Let $D$ be a diagram of a planar virtual knotoid $K$, and let $G$ denote its associated Gauss diagram. We shall prove that $D^{(r)}$ is an invariant of $K$ under moves in $\textit{g}\mathcal{R} = \{ \Omega_1, \Omega_2, \Omega_3, \Omega_1^v,$ $\Omega_2^v, \Omega_3^v, \Omega_3^m, \Omega_v \}$. Now suppose $D'$ is another diagram of $K$ with associated Gauss diagram $G'$. If $D$ and $D'$ differ by one move in $\textit{g}\mathcal{R}$, we show that $D^{(r)}$ and $D'^{(r)}$ belong to the same equivalence class.
		
		These moves in $\textit{g}\mathcal{R}$ leave the indices of all crossings unaffected except those involved in the move itself. Consequently, for every crossing not altered by the move in diagrams $D$ and $D'$, the virtualization induced by the $r$-covering is identical. 
		
		\noindent\textbf{Case (a).}
		Suppose that diagrams $D$ and $D'$ differ by an $\Omega_1$-move. Then their Gauss diagrams $G$ and $G'$ differ by exactly one chord having index $0$ as shown in Figure~\ref{fig25}. Thus, $D$ and $D'$ differ by exactly one crossing having index $0$. Since the $r$-covering does not virtualize $0$-index crossings, $D^{(r)}$ and $D'^{(r)}$ differ only by a single crossing corresponding the $\Omega_1$-move and so belong to the same equivalence class.
		
		\noindent\textbf{Case (b).}
		Suppose that diagrams $D$ and $D'$ differ by an $\Omega_2$-move. Then their Gauss diagrams $G$ and $G'$ differ by two chords with indices as shown in Figure~\ref{fig29}, where indices are equal. Thus, $D$ and $D'$ differ by two crossings whose indices are equal. Therefore, both  corresponding crossings will be preserved or removed under the $r$-covering. Hence $D^{(r)}$ and  $D'^{(r)}$ differ at most by these two  crossings. As a result, $D^{(r)}$ and $D'^{(r)}$ differ by at most one $\Omega_2$-move and so belong to the same equivalence class.
		
		\noindent\textbf{Case (c).}
		Suppose that diagrams $D$ and $D'$ differ by an $\Omega_3$-move. Then their diagrams $G$ and $G'$ are differ by three chords with indices as shown in Figure~\ref{fig30}.  According to Figure~\ref{fig30} (a), let $x$, $y$, $z$ denote the three chords in $G$, and let $x'$, $y'$, $z'$ be their corresponding chords in $G'$. We assume that $\operatorname{Ind}(x)=\operatorname{Ind}(x')=i$, $\operatorname{Ind}(y)=\operatorname{Ind}(y')=j$, $\operatorname{Ind}(z)=\operatorname{Ind}(z')=k$. Analogously to the proof of Lemma 4.1 in~\cite{Ch17}, it is not difficult to show that $i=j+k$. The same argument holds for Figure~\ref{fig30} (b), (c), and (d).  Hence the three crossings distinguishing $D$ from $D'$ via this $\Omega_3$-move satisfy the same relation that the index of one crossing equals the sum of the indices of the other two crossings. Consequently, there are three possible cases: all three indices are divisible by $r$, all three are not divisible by $r$, or exactly two indices are not divisible by $r$ while the remaining one is divisible by $r$. The case where one index is not divisible by $r$ and the other two are divisible by $r$ never occurs. Under the $r$-covering, the three  crossings are either all retained, all removed, or only two of them are removed. If all three crossings are kept, $D^{(r)}$ and $D'^{(r)}$ differ by exactly one $\Omega_3$-move. If all three crossings or just two of them are removed,  $D^{(r)}$ is equivalent with $D'^{(r)}$. Therefore, $D^{(r)}$ and $D'^{(r)}$ differ by at most one $\Omega_3$-move and so belong to the same equivalence class.
		
		\noindent\textbf{Case (d).}
		Suppose that $D$ and $D'$ differ by one of the moves $\Omega_1^v, \Omega_2^v, \Omega_3^v, \Omega_3^m$ and $\ \Omega_v$, Virtual crossings are disregarded in Gauss diagrams and each of these moves involves at most one classical crossing. Such moves preserve the indices of all chords in Gauss diagrams and thus leave the indices of all crossings in diagrams unchanged. Since the $r$-covering depends on indices, $D^{(r)}$ and $D'^{(r)}$ differ by at most one of the moves $\Omega_1^v, \Omega_2^v, \Omega_3^v, \Omega_3^m$ and $\Omega_v$, and so belong to the same equivalence class.
		
		In summary, $D^{(r)}$ is invariant under $\textit{g}\mathcal{R}$. While $[D^{(r)}]$ denotes the equivalence class of the diagram $D^{(r)}$ modulo $\textit{g}\mathcal{R}$ and planar isotopies, $[D^{(r)}]$ is therefore well-defined. Hence $[D^{(r)}]$ is an invariant of $K$. This complete the proof of Theorem~\ref{thm-invc}. \end{proof}
	
	We denote $K^{(r)} = [D^{(r)}]$. Thus, $K^{(r)}$ is an invariant of $K$. 
	
	\section{Gordian distance: proof of Theorem~\ref{thm-ccov}} \label{sec4}
	
	In this section, we commence with the proof of Theorem~\ref{thm-ccov}. Based on this theorem, we then give two rules for constructing planar virtual knotoids, by which one can construct many pairs of non-homotopic planar virtual knotoids.
	
	\begin{proof}[Proof of Theorem~\ref{thm-ccov}.]
		Since $K$ and $K'$ are homotopic,  set $d_G(K, K') = n$. There exists a sequence of planar virtual knotoids $K_0, \dots, K_n$ satisfying $K_0=K$, $K_n=K'$, where a diagram of $K_{i-1}$ differs from a diagram of $K_i$ by a single crossing change ($i = 1 , \dots , n$). Let $D_{i-1}$, $D_i$ denote the diagrams of $K_{i-1}$, $K_i$ and let $G_{i-1}$, ${G_i}$ be their corresponding Gauss diagrams, respectively, such that the crossing change at a  crossing $c$ for $D_{i-1}$ converts it into $D_{i}$. We denote the crossing of $D_{i}$ corresponding to $c$ by $c'$. We observe from the crossing change on a Gauss diagram illustrated in Figure~\ref{fig9} that this crossing change only flips the sign of the index for the affected chord, leaving the indices of all other chords unchanged. Therefore, when we proceed to apply the crossing change at $c$, the indices of all crossings of  $D_{i-1}$ remain unchanged with the exception of $c$, and $\Ind(c') = -\Ind(c)$. Thus, all crossings of  $D_{i-1}^{(r)}$ and $D_{i}^{(r)}$ coincide with the exception of the crossings corresponding to $c$ and $c'$. We then distinguish two cases for the relationship between $\Ind(c)$ and $r$.
		
		\noindent\textbf{Case (a).}
		The $\Ind(c)$ is divisible by $r$. The $r$-covering preserves the crossing $c$ in $D_{i-1}$ and the  crossing $c'$ in $D_{i}$. Then, a crossing change from  $D_{i-1}$ to $D_{i}$ corresponds to a crossing change from $D_{i-1}^{(r)}$ to $D_{i}^{(r)}$. Therefore, $D_{i-1}^{(r)}$ and $D_i^{(r)}$ differ by a crossing change. 
		
		\noindent\textbf{Case (b).}
		The $\Ind(c)$ is not divisible by $r$. The $r$-covering virtualizes the crossing $c$ in $D_{i-1}$ and the crossing $c'$ in $D_{i}$. Hence, $D_{i-1}^{(r)}$ coincides with $D_i^{(r)}$.
		
		Thus, we obtain a sequence of planar virtual knotoid diagrams $D_0^{(r)}, \dots, D_n^{(r)}$ satisfying $D_0^{(r)}=D^{(r)}$, $D_n^{(r)}=D'^{(r)}$, where $D_{i-1}^{(r)}$ and $D_i^{(r)}$ differ by a crossing change or are identical ($i = 1 , \dots , n$). Therefore, $K^{(r)}$ is homotopic to $K'^{(r)}$. Moreover, $d_G(K, K') \geq d_G(K^{(r)}, K'^{(r)})$.
	\end{proof}
	
	Theorem~\ref{thm-ccov} enables us to further check the homotopy of two planar virtual knotoids. We now present  three constructions based on Theorem~\ref{thm-ccov}, which allow us to construct pairs of non-homotopic planar virtual knotoids.
	
	\begin{construction}\label{4.1}
		{\rm We define a family $\mathcal{U}$ of planar virtual knotoids. A planar virtual knotoid $K$ belongs to $\mathcal{U}$ if it admits a diagram $D$ whose Gauss diagram $G(D)$ satisfies the following conditions. 
			\begin{enumerate}
				\item[{(1)}]
				For some $n\in\mathbb{N}$, the Gauss diagram $G(D)$ contains chords $c_1,c_2,\dots,c_{2n+1}$. When $n=0$, $G(D)$ has only one chord, and $D$ represents $\chi_1$ or $\chi_2$ in Figure~\ref{fig6}; 
				\item[{(2)}]
				As one travels along $G(D)$ from the tail to the head, the chord endpoints occur in the order $c_1,c_2,\dots,c_{2n+1},c_2,\dots,c_{2n+1},c_1$; 
				\item[{(3)}]
				If $n\geq1$, then, as one travels along $G(D)$ from the tail to the head, consider the first endpoint encountered for each chord in $\{c_2,\dots,c_{2n+1}\}$. There exists $\varepsilon\in\{+1,-1\}$ such that the first endpoints of $c_2,c_4,\dots,c_{2n}$ have sign $\varepsilon$, whereas the first endpoints of $c_3,c_5,\dots,c_{2n+1}$ have sign $-\varepsilon$ (see Figure~\ref{fig11} (b));
				\item[{(4)}]
				Virtualizing every crossing of $D$ except the crossing corresponding to $c_1$ produces a diagram equivalent to the diagram of $\chi_1$ or $\chi_2$ in Figure~\ref{fig6}. 
			\end{enumerate}
		}
	\end{construction}
	
	Figure~\ref{fig11} shows the diagram and Gauss diagram of a knotoid $u$. It is straightforward to verify that these diagrams satisfy the conditions specified in Construction~\ref{4.1}, which implies that $u\in\mathcal U$. 
	\begin{figure}[htbp]
		\begin{center}
			\begin{subfigure}{0.34\textwidth}
				\centering
				\tikzset{every picture/.style={line width=0.75pt}}         
				\begin{tikzpicture}[x=0.75pt,y=0.75pt,yscale=-1,xscale=1]
					\draw [black, very thick]    (127.87,106.37) .. controls (132.53,117.7) and (116.17,126.73) .. (101.26,122.69) .. controls (86.34,118.65) and (80.53,106.74) .. (87.53,91.07) .. controls (94.53,75.4) and (116.53,64.03) .. (144.53,61.7) .. controls (172.53,59.37) and (208.2,74.7) .. (209.53,117.37) .. controls (210.87,160.03) and (203.87,187.7) .. (177.87,192.03) .. controls (151.87,196.37) and (126.53,187.03) .. (127.87,169.7) .. controls (129.2,152.37) and (145.87,149.37) .. (159.2,157.03) ;
					\draw [black, very thick]    (106.13,115.2) .. controls (113.2,99.37) and (127.87,92.03) .. (144.87,94.7) ;
					\draw [black, very thick]    (97.87,90.03) .. controls (104.87,82.03) and (114.53,80.37) .. (119.53,90.7) ;
					\draw [black, very thick]    (179.87,98.03) .. controls (185.87,88.03) and (177.25,61.8) .. (162.06,81.58) .. controls (146.87,101.37) and (134.63,161.49) .. (147.59,170.48) .. controls (160.55,179.46) and (180.53,143.7) .. (179.2,135.7) ;
					\draw [black, very thick]    (163.87,99.37) .. controls (175.6,102.2) and (182.2,112.03) .. (181.53,118.03) ;
					\draw [black, very thick]    (172.47,113.2) .. controls (165.53,131.37) and (191.2,123.37) .. (198.2,137.03) .. controls (205.2,150.7) and (188.2,172.03) .. (174.2,165.37) ;
					\draw [black, very thick]    (150.06,120.08) -- (145.13,126.72) -- (143.24,118.45) ;
					\draw  [fill={rgb, 255:red, 0; green, 0; blue, 0 }  ,fill opacity=1 ][black, very thick]  (101.61,115.2) .. controls (101.61,112.7) and (103.64,110.68) .. (106.13,110.68) .. controls (108.63,110.68) and (110.65,112.7) .. (110.65,115.2) .. controls (110.65,117.7) and (108.63,119.72) .. (106.13,119.72) .. controls (103.64,119.72) and (101.61,117.7) .. (101.61,115.2) -- cycle ;
					\draw  [fill={rgb, 255:red, 0; green, 0; blue, 0 }  ,fill opacity=1 ][black, very thick]  (93.35,90.03) .. controls (93.35,87.54) and (95.37,85.51) .. (97.87,85.51) .. controls (100.36,85.51) and (102.39,87.54) .. (102.39,90.03) .. controls (102.39,92.53) and (100.36,94.55) .. (97.87,94.55) .. controls (95.37,94.55) and (93.35,92.53) .. (93.35,90.03) -- cycle ;
					\draw  [black, very thick]  (137.69,153.77) .. controls (137.69,151.27) and (139.71,149.25) .. (142.21,149.25) .. controls (144.71,149.25) and (146.73,151.27) .. (146.73,153.77) .. controls (146.73,156.27) and (144.71,158.29) .. (142.21,158.29) .. controls (139.71,158.29) and (137.69,156.27) .. (137.69,153.77) -- cycle ;
					\draw (150,200) node [anchor=north west][inner sep=0.75pt]  {$(a)$};		
				\end{tikzpicture}
			\end{subfigure}
			\qquad 
			\begin{subfigure}{0.34\textwidth}
				\centering
				\tikzset{every picture/.style={line width=0.75pt}}         
				\begin{tikzpicture}[x=0.75pt,y=0.75pt,yscale=-1,xscale=1]
					\draw  [draw opacity=0][black, very thick]  (226.44,355.99) .. controls (216.68,379.99) and (193.45,396.88) .. (166.34,396.88) .. controls (130.44,396.88) and (101.34,367.26) .. (101.34,330.73) .. controls (101.34,294.19) and (130.44,264.58) .. (166.34,264.58) .. controls (193.63,264.58) and (216.98,281.68) .. (226.62,305.93) -- (166.34,330.73) -- cycle ; 
					\draw  [black, very thick]  (226.44,355.99) .. controls (216.68,379.99) and (193.45,396.88) .. (166.34,396.88) .. controls (130.44,396.88) and (101.34,367.26) .. (101.34,330.73) .. controls (101.34,294.19) and (130.44,264.58) .. (166.34,264.58) .. controls (193.63,264.58) and (216.98,281.68) .. (226.62,305.93) ;  
					\draw  [fill={rgb, 255:red, 0; green, 0; blue, 0 }  ,fill opacity=1 ][black, very thick]  (221.92,355.99) .. controls (221.92,353.49) and (223.94,351.47) .. (226.44,351.47) .. controls (228.93,351.47) and (230.96,353.49) .. (230.96,355.99) .. controls (230.96,358.48) and (228.93,360.51) .. (226.44,360.51) .. controls (223.94,360.51) and (221.92,358.48) .. (221.92,355.99) -- cycle ;
					\draw  [fill={rgb, 255:red, 0; green, 0; blue, 0 }  ,fill opacity=1 ][black, very thick]  (222.1,305.93) .. controls (222.1,303.43) and (224.13,301.41) .. (226.62,301.41) .. controls (229.12,301.41) and (231.14,303.43) .. (231.14,305.93) .. controls (231.14,308.43) and (229.12,310.45) .. (226.62,310.45) .. controls (224.13,310.45) and (222.1,308.43) .. (222.1,305.93) -- cycle ;
					\draw  [red, very thick, ->]  (200.56,386.78) -- (214.32,286.12) ;
					\draw  [black, very thick, ->]  (102.97,343.51) -- (192.69,270.66) ;
					\draw  [black, very thick, ->]  (160.69,264.94) -- (113.83,369.23) ;
					\draw  [black, very thick, ->]  (139.83,391.23) -- (126.3,279.2) ;
					\draw  [black, very thick, ->]  (106.97,303.8) -- (172.11,396.09) ;
					\draw  [black, very thick]  (104.86,325.54) -- (101.53,333.12) -- (97.85,325.47) ;
					\draw (192,249) node [anchor=north west][inner sep=0.75pt]  [font=\large]  {$c_{2}$};
					\draw (95,378) node [anchor=north west][inner sep=0.75pt]  [font=\large]  {$c_{3}$};
					\draw (113,253) node [anchor=north west][inner sep=0.75pt]  [font=\large]  {$c_{4}$};
					\draw (154.85,251.5) node [anchor=north west][inner sep=0.75pt]  [font=\scriptsize]  {$-$};
					\draw (86.51,341.1) node [anchor=north west][inner sep=0.75pt]  [font=\scriptsize]  {$-$};
					\draw (133.31,397.23) node [anchor=north west][inner sep=0.75pt]  [font=\scriptsize]  {$-$};
					\draw (88.5,296.1) node [anchor=north west][inner sep=0.75pt]  [font=\scriptsize]  {$-$};
					\draw (170,408) node [anchor=north west][inner sep=0.75pt]  [font=\large]  {$c_{5}$};
					\draw (212,265) node [anchor=north west][inner sep=0.75pt]  [font=\large]  {$c_{1}$};
					\draw (71,338) node [anchor=north west][inner sep=0.75pt]  [font=\large]  {$c_{2}$};
					\draw (152,240) node [anchor=north west][inner sep=0.75pt]  [font=\large]  {$c_{3}$};
					\draw (130,403) node [anchor=north west][inner sep=0.75pt]  [font=\large]  {$c_{4}$};
					\draw (73,292) node [anchor=north west][inner sep=0.75pt]  [font=\large]  {$c_{5}$};
					\draw (192,400) node [anchor=north west][inner sep=0.75pt]  [font=\large]  {$c_{1}$};
					\draw (195.86,392.6) node [anchor=north west][inner sep=0.75pt]  [font=\scriptsize]  {$+$};
					\draw (211,275) node [anchor=north west][inner sep=0.75pt]  [font=\scriptsize]  {$-$};
					\draw (193,260) node [anchor=north west][inner sep=0.75pt]  [font=\scriptsize]  {$+$};
					\draw (104,372) node [anchor=north west][inner sep=0.75pt]  [font=\scriptsize]  {$+$};
					\draw (120,265) node [anchor=north west][inner sep=0.75pt]  [font=\scriptsize]  {$+$};
					\draw (170,400) node [anchor=north west][inner sep=0.75pt]  [font=\scriptsize]  {$+$};
					\draw (160,422) node [anchor=north west][inner sep=0.75pt]    {(b)};
				\end{tikzpicture}					
			\end{subfigure}
			\caption{Diagram and Gauss diagram of planar virtual knotoid $u \in \mathcal U$.}
			\label{fig11}
		\end{center}
	\end{figure}
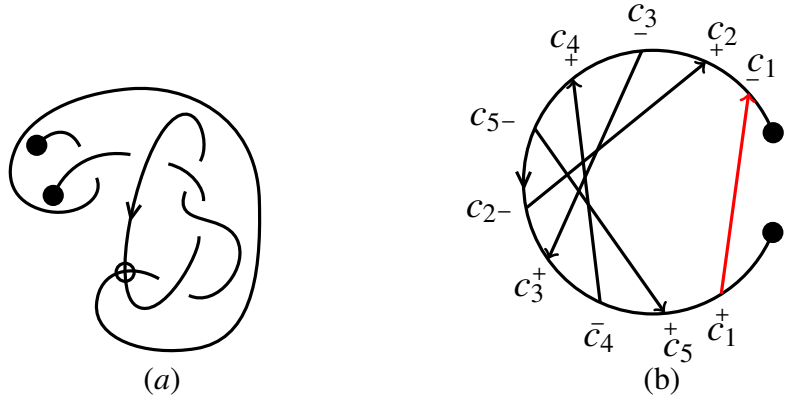
	
	\begin{construction}\label{4.2} {\rm 
			We define a family $\mathcal{V}$ of planar virtual knotoids. A planar virtual knotoid $K$ belongs to $\mathcal{V}$ if it admits a diagram $D$ whose Gauss diagram $G(D)$ satisfies the following conditions.
			\begin{enumerate}
				\item[{(1)}] For some $n\in\mathbb{N}$, the Gauss diagram $G(D)$ contains chords $c_1,c_2,\dots,c_{2n+1}$. When $n=0$, $G(D)$ has only one chord, and $D$ represents $\chi_1$ or $\chi_2$ in Figure~\ref{fig6}; 
				\item[{(2)}] As one travels along $G(D)$ from the tail to the head, the chord endpoints occur in the order $c_1,c_2,\dots,c_{2n+1},c_1,\dots,c_{2n+1}$; 
				\item[{(3)}]
				If $n\geq1$, then, as one travels along $G(D)$ from the tail to the head, consider the first endpoint encountered for each chord in $\{c_1,\dots,c_{2n+1}\}$. There exists $\varepsilon\in\{+1,-1\}$ such that the first endpoints of $c_1,c_3,\dots,c_{2n-1}$ have sign $\varepsilon$, whereas the first endpoints of $c_2,c_4,\dots,c_{2n}$ and the first endpoint of $c_{2n+1}$ have sign $-\varepsilon$ (see Figure~\ref{fig27} (b));
				\item[{(4)}]
				Virtualizing every crossing of $D$ except the crossing corresponding to $c_{2n+1}$ produces a diagram equivalent to the diagram of $\chi_1$ or $\chi_2$ in Figure~\ref{fig6}. 
			\end{enumerate}
		}
	\end{construction}
	
	Depicted in Figure~\ref{fig27} are the diagram and Gauss diagram associated with knotoid $v$. The conditions from Construction~\ref{4.2} are clearly fulfilled, so $v$ is an element of $\mathcal V$.
	\begin{figure}[htbp]
		\begin{center}
			\begin{subfigure}{0.34\textwidth}
				\centering
				\tikzset{every picture/.style={line width=0.75pt}}
				\begin{tikzpicture}[x=0.75pt,y=0.75pt,yscale=-1,xscale=1]
					\draw [black, very thick]    (78.44,103.41) .. controls (84.81,87.35) and (89.96,82.43) .. (107.12,80.29) ;
					\draw [black, very thick]    (121.08,78.39) .. controls (152.93,73.88) and (171.55,73.37) .. (192.37,88.13) ;
					\draw [black, very thick]    (148.76,170.23) .. controls (223.24,171.52) and (239.66,123.09) .. (206.34,99.24) ;
					\draw [black, very thick]    (158.49,64.47) .. controls (158.22,59.22) and (131.99,32.36) .. (120.18,55.43) .. controls (108.37,78.51) and (101.99,170.52) .. (126.47,170.36) ;
					\draw [black, very thick]    (157.56,88.74) .. controls (155.72,97.61) and (162.09,117.03) .. (198.11,93.46) .. controls (234.12,69.9) and (268.52,99.72) .. (233.19,116.33) ;
					\draw [black, very thick]    (136.85,147.85) .. controls (124.6,208.97) and (225.79,205.35) .. (191.74,164.69) .. controls (157.68,124.03) and (185.12,125.58) .. (191.49,126.36) .. controls (197.86,127.13) and (206.68,128.69) .. (209.87,126.98) ;
					\draw [black, very thick]    (163.06,98.99) -- (168.77,104.98) -- (160.3,105.44) ;
					\draw  [fill={rgb, 255:red, 0; green, 0; blue, 0 }  ,fill opacity=1 ][black, very thick]  (132.33,147.85) .. controls (132.33,145.36) and (134.36,143.33) .. (136.85,143.33) .. controls (139.35,143.33) and (141.37,145.36) .. (141.37,147.85) .. controls (141.37,150.35) and (139.35,152.37) .. (136.85,152.37) .. controls (134.36,152.37) and (132.33,150.35) .. (132.33,147.85) -- cycle ;
					\draw  [black, very thick]  (186.6,163.94) .. controls (186.6,161.44) and (188.62,159.42) .. (191.12,159.42) .. controls (193.62,159.42) and (195.64,161.44) .. (195.64,163.94) .. controls (195.64,166.43) and (193.62,168.46) .. (191.12,168.46) .. controls (188.62,168.46) and (186.6,166.43) .. (186.6,163.94) -- cycle ;
					\draw  [fill={rgb, 255:red, 0; green, 0; blue, 0 }  ,fill opacity=1 ][black, very thick]  (73.92,103.41) .. controls (73.92,100.91) and (75.95,98.89) .. (78.44,98.89) .. controls (80.94,98.89) and (82.96,100.91) .. (82.96,103.41) .. controls (82.96,105.91) and (80.94,107.93) .. (78.44,107.93) .. controls (75.95,107.93) and (73.92,105.91) .. (73.92,103.41) -- cycle ;
					\draw (152,211) node [anchor=north west][inner sep=0.75pt]  {$(a)$};
				\end{tikzpicture}
			\end{subfigure}
			\qquad 
			\begin{subfigure}{0.34\textwidth}
				\centering
				\tikzset{every picture/.style={line width=0.75pt}} 
				\begin{tikzpicture}[x=0.75pt,y=0.75pt,yscale=-1,xscale=1]
					\draw  [draw opacity=0][black, very thick]  (217.74,354.36) .. controls (207.98,378.36) and (184.75,395.25) .. (157.64,395.25) .. controls (121.74,395.25) and (92.64,365.64) .. (92.64,329.1) .. controls (92.64,292.57) and (121.74,262.95) .. (157.64,262.95) .. controls (184.93,262.95) and (208.28,280.06) .. (217.92,304.31) -- (157.64,329.1) -- cycle ; 
					\draw  [black, very thick]  (217.74,354.36) .. controls (207.98,378.36) and (184.75,395.25) .. (157.64,395.25) .. controls (121.74,395.25) and (92.64,365.64) .. (92.64,329.1) .. controls (92.64,292.57) and (121.74,262.95) .. (157.64,262.95) .. controls (184.93,262.95) and (208.28,280.06) .. (217.92,304.31) ;
					\draw  [fill={rgb, 255:red, 0; green, 0; blue, 0 }  ,fill opacity=1 ][black, very thick]  (213.4,304.31) .. controls (213.4,301.81) and (215.43,299.79) .. (217.92,299.79) .. controls (220.42,299.79) and (222.44,301.81) .. (222.44,304.31) .. controls (222.44,306.8) and (220.42,308.83) .. (217.92,308.83) .. controls (215.43,308.83) and (213.4,306.8) .. (213.4,304.31) -- cycle ; 
					\draw  [fill={rgb, 255:red, 0; green, 0; blue, 0 }  ,fill opacity=1 ][black, very thick]  (213.22,354.36) .. controls (213.22,351.87) and (215.24,349.84) .. (217.74,349.84) .. controls (220.23,349.84) and (222.26,351.87) .. (222.26,354.36) .. controls (222.26,356.86) and (220.23,358.88) .. (217.74,358.88) .. controls (215.24,358.88) and (213.22,356.86) .. (213.22,354.36) -- cycle ;
					\draw  [red, very thick, ->]  (208.15,370.64) -- (94.57,313.05) ;
					\draw  [black, very thick, ->]  (94.42,343.75) -- (213.99,296.24) ;
					\draw  [black, very thick, ->]  (140.4,392.96) -- (140.8,265.61) ;
					\draw  [black, very thick, ->]  (109.27,285.31) -- (193.27,384.32) ;
					\draw  [black, very thick, ->]  (190.71,271.97) -- (119.27,382.52) ;
					\draw  [black, very thick]  (95.97,329.13) -- (93.2,336.93) -- (88.97,329.57) ;
					\draw (218,276) node [anchor=north west][inner sep=0.75pt]  [font=\large]  {$c_{1}$};
					\draw (102,392) node [anchor=north west][inner sep=0.75pt]  [font=\large]  {$c_{2}$};
					\draw (130,241) node [anchor=north west][inner sep=0.75pt]  [font=\large]  {$c_{3}$};
					\draw (197,392) node [anchor=north west][inner sep=0.75pt]  [font=\large]  {$c_{4}$};		
					\draw (66,338) node [anchor=north west][inner sep=0.75pt]  [font=\large]  {$c_{1}$};
					\draw (190,251) node [anchor=north west][inner sep=0.75pt]  [font=\large]  {$c_{2}$};
					\draw (130,405) node [anchor=north west][inner sep=0.75pt]  [font=\large]  {$c_{3}$};
					\draw (84,264) node [anchor=north west][inner sep=0.75pt]  [font=\large]  {$c_{4}$};
					\draw (221,373) node [anchor=north west][inner sep=0.75pt]  [font=\large]  {$c_{5}$};
					\draw (190,261) node [anchor=north west][inner sep=0.75pt]  [font=\scriptsize]  {$-$};
					\draw (80,341) node [anchor=north west][inner sep=0.75pt]  [font=\scriptsize]  {$-$};
					\draw (212.48,371.14) node [anchor=north west][inner sep=0.75pt]  [font=\scriptsize]  {$+$};
					\draw (136,398) node [anchor=north west][inner sep=0.75pt]  [font=\scriptsize]  {$-$};
					\draw (97,275) node [anchor=north west][inner sep=0.75pt]  [font=\scriptsize]  {$-$};
					\draw (67,301) node [anchor=north west][inner sep=0.75pt]  [font=\large]  {$c_{5}$};
					\draw (82,306) node [anchor=north west][inner sep=0.75pt]  [font=\scriptsize]  {$-$};
					\draw (216,288) node [anchor=north west][inner sep=0.75pt]  [font=\scriptsize]  {$+$};
					\draw (110,385) node [anchor=north west][inner sep=0.75pt]  [font=\scriptsize]  {$+$};
					\draw (135,254) node [anchor=north west][inner sep=0.75pt]  [font=\scriptsize]  {$+$};
					\draw (192,386) node [anchor=north west][inner sep=0.75pt]  [font=\scriptsize]  {$+$};
					\draw (145,415) node [anchor=north west][inner sep=0.75pt]    {(b)};
				\end{tikzpicture}
			\end{subfigure}		
			\caption{Diagram and Gauss diagram of planar virtual knotoid $v \in \mathcal V$.}
			\label{fig27}
		\end{center}
	\end{figure}
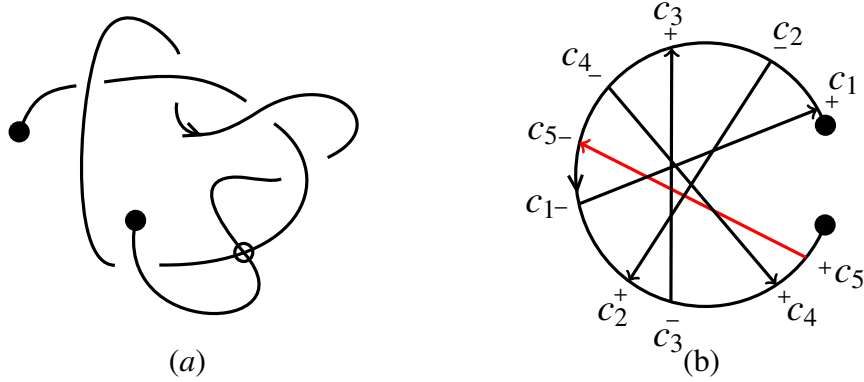
	
	\begin{construction}\label{4.3}\
		{\rm 
			We define a family $\mathcal{W}$ of planar virtual knotoids. A planar virtual knotoid $K$ belongs to $\mathcal{W}$ if it admits a diagram $D$ whose Gauss diagram $G(D)$ satisfies the following conditions.
			\begin{enumerate}
				\item[{(1)}]
				For some $n\in\mathbb{N}$, the Gauss diagram $G(D)$ contains chords $c_1,c_2,\dots,c_{2n}$. When $n=0$, $G(D)$ is chord-free and $D$ represents the trivial knotoid;
				\item[{(2)}]
				If $n\geq1$, then, as one travels along $G(D)$ from the tail to the head, the chord endpoints occur in the order $c_1,c_2,\dots,c_{2n},c_1,\dots,c_{2n}$; 
				\item[{(3)}]
				If $n\geq1$, then, as one travels along $G(D)$ from the tail to the head, consider the first endpoint encountered for each chord in $\{c_1,\dots,c_{2n}\}$. There exists $\varepsilon\in\{+1,-1\}$ such that the first endpoints of $c_1,c_3,\dots,c_{2n-1}$ have sign $\varepsilon$, whereas the first endpoints of $c_2,c_4,\dots,c_{2n}$ have sign $-\varepsilon$ (see Figure~\ref{fig26} (b)).
			\end{enumerate}
		}
	\end{construction}
	
	Figure~\ref{fig26} shows the diagram and Gauss diagram of a knotoid $w$. One can easily check that the requirements of Construction~\ref{4.3} are satisfied, so $w\in\mathcal W$.
	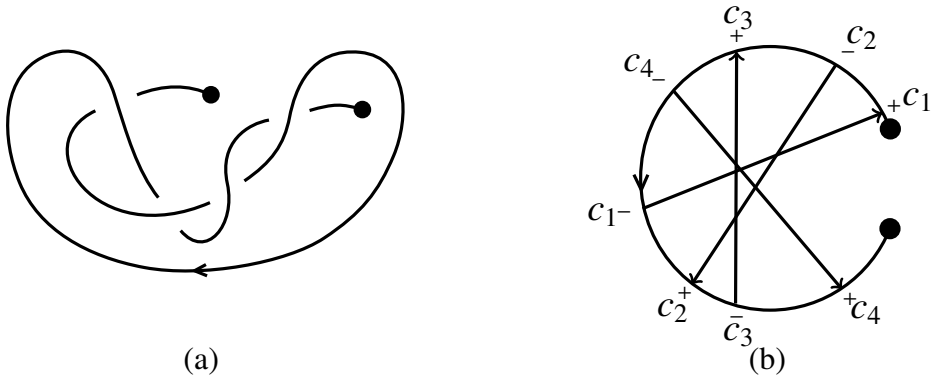
\begin{figure}[htbp]
		\begin{center}
			\begin{subfigure}{0.34\textwidth}
				\centering
				\tikzset{every picture/.style={line width=0.75pt}}
				\begin{tikzpicture}[x=0.75pt,y=0.75pt,yscale=-0.8,xscale=0.8]
					\draw [black, very thick]    (444.65,383.44) .. controls (461.53,376.88) and (476.79,376.81) .. (491.03,383.38) ; 
					\draw [black, very thick]    (418.55,393.45) .. controls (374.26,414.31) and (418.91,482.04) .. (490.41,451.04) ;
					\draw [black, very thick]    (458.03,447.93) .. controls (444.69,430.93) and (438.37,408.51) .. (430.35,382.97) .. controls (422.33,357.43) and (401.33,348.77) .. (381,362.43) .. controls (360.67,376.1) and (359.08,409.61) .. (371.09,438.73) .. controls (383.09,467.85) and (413.33,485.1) .. (449.33,491.77) .. controls (485.33,498.43) and (538.6,488.8) .. (561.97,473.29) .. controls (585.33,457.77) and (595.98,444.54) .. (605.67,421.43) .. controls (615.36,398.33) and (615.77,359.97) .. (584.67,356.77) .. controls (553.56,353.56) and (543.67,380.1) .. (540.33,397.43) .. controls (537,414.77) and (529,425.77) .. (512,437.43) ; 
					\draw [black, very thick]    (472.03,468.59) .. controls (488.04,487.51) and (506.88,465.34) .. (501.3,439.12) .. controls (495.73,412.9) and (514.41,401.51) .. (527.41,399.26) ;
					\draw  [black, very thick]  (488.59,496.97) -- (480.73,493.78) -- (487.88,489.9) ;
					\draw  [fill={rgb, 255:red, 0; green, 0; blue, 0 }  ,fill opacity=1 ][black, very thick]  (581.48,392.67) .. controls (581.48,390.18) and (583.5,388.15) .. (586,388.15) .. controls (588.5,388.15) and (590.52,390.18) .. (590.52,392.67) .. controls (590.52,395.17) and (588.5,397.19) .. (586,397.19) .. controls (583.5,397.19) and (581.48,395.17) .. (581.48,392.67) -- cycle ; 
					\draw  [fill={rgb, 255:red, 0; green, 0; blue, 0 }  ,fill opacity=1 ][black, very thick]  (486.51,383.38) .. controls (486.51,380.89) and (488.54,378.86) .. (491.03,378.86) .. controls (493.53,378.86) and (495.55,380.89) .. (495.55,383.38) .. controls (495.55,385.88) and (493.53,387.9) .. (491.03,387.9) .. controls (488.54,387.9) and (486.51,385.88) .. (486.51,383.38) -- cycle ;
					\draw [black, very thick]    (552.88,393.26) .. controls (568,388.9) and (573.11,388.9) .. (586,392.67) ;
					\draw (472,540) node [anchor=north west][inner sep=0.75pt]    {(a)};
				\end{tikzpicture}
			\end{subfigure}
			\qquad \qquad 
			\begin{subfigure}{0.34\textwidth}
				\centering
				\tikzset{every picture/.style={line width=0.75pt}}
				\begin{tikzpicture}[x=0.75pt,y=0.75pt,yscale=-1,xscale=1]
					\draw  [draw opacity=0][black, very thick]  (428.24,354.13) .. controls (418.48,378.13) and (395.25,395.02) .. (368.14,395.02) .. controls (332.24,395.02) and (303.14,365.4) .. (303.14,328.87) .. controls (303.14,292.34) and (332.24,262.72) .. (368.14,262.72) .. controls (395.43,262.72) and (418.78,279.83) .. (428.42,304.07) -- (368.14,328.87) -- cycle ; \draw  [black, very thick]  (428.24,354.13) .. controls (418.48,378.13) and (395.25,395.02) .. (368.14,395.02) .. controls (332.24,395.02) and (303.14,365.4) .. (303.14,328.87) .. controls (303.14,292.34) and (332.24,262.72) .. (368.14,262.72) .. controls (395.43,262.72) and (418.78,279.83) .. (428.42,304.07) ;  
					\draw [black, very thick, <-]    (424.49,296) -- (304.1,344.08) ;
					\draw [black, very thick, <-]    (351.3,265.38) -- (350.77,391.68) ;
					\draw [black, very thick, ->]    (319.77,285.08) -- (404.1,384.41) ;
					\draw [black, very thick, <-]    (328.77,382.08) -- (401.21,271.74) ;
					\draw  [fill={rgb, 255:red, 0; green, 0; blue, 0 }  ,fill opacity=1 ][black, very thick]  (423.72,354.13) .. controls (423.72,351.63) and (425.74,349.61) .. (428.24,349.61) .. controls (430.73,349.61) and (432.76,351.63) .. (432.76,354.13) .. controls (432.76,356.63) and (430.73,358.65) .. (428.24,358.65) .. controls (425.74,358.65) and (423.72,356.63) .. (423.72,354.13) -- cycle ;
					\draw  [fill={rgb, 255:red, 0; green, 0; blue, 0 }  ,fill opacity=1 ][black, very thick]  (423.9,304.07) .. controls (423.9,301.58) and (425.93,299.55) .. (428.42,299.55) .. controls (430.92,299.55) and (432.94,301.58) .. (432.94,304.07) .. controls (432.94,306.57) and (430.92,308.59) .. (428.42,308.59) .. controls (425.93,308.59) and (423.9,306.57) .. (423.9,304.07) -- cycle ;
					\draw (433,282) node [anchor=north west][inner sep=0.75pt]  [font=\large]  {$c_{1}$};
					\draw (310,388) node [anchor=north west][inner sep=0.75pt]  [font=\large]  {$c_{2}$};
					\draw (343,241) node [anchor=north west][inner sep=0.75pt]  [font=\large]  {$c_{3}$};
					\draw (407,389) node [anchor=north west][inner sep=0.75pt]  [font=\large]  {$c_{4}$};
					\draw (275,340) node [anchor=north west][inner sep=0.75pt]  [font=\large]  {$c_{1}$};
					\draw (405,251) node [anchor=north west][inner sep=0.75pt]  [font=\large]  {$c_{2}$};
					\draw (343,400) node [anchor=north west][inner sep=0.75pt]  [font=\large]  {$c_{3}$};
					\draw (293,265) node [anchor=north west][inner sep=0.75pt]  [font=\large]  {$c_{4}$};
					\draw (402,261) node [anchor=north west][inner sep=0.75pt]  [font=\scriptsize]  {$-$};
					\draw (289.5,341) node [anchor=north west][inner sep=0.75pt]  [font=\scriptsize]  {$-$};
					\draw (346,394) node [anchor=north west][inner sep=0.75pt]  [font=\scriptsize]  {$-$};
					\draw (306.98,275) node [anchor=north west][inner sep=0.75pt]  [font=\scriptsize]  {$-$};
					\draw (425,288) node [anchor=north west][inner sep=0.75pt]  [font=\scriptsize]  {$+$};
					\draw (320,382) node [anchor=north west][inner sep=0.75pt]  [font=\scriptsize]  {$+$};
					\draw (346,253) node [anchor=north west][inner sep=0.75pt]  [font=\scriptsize]  {$+$};
					\draw (403,385) node [anchor=north west][inner sep=0.75pt]  [font=\scriptsize]  {$+$};
					\draw  [black, very thick]  (306.78,327.06) -- (303.47,334.65) -- (299.77,327.02) ;
					\draw (356,412) node [anchor=north west][inner sep=0.75pt]    {(b)};
				\end{tikzpicture}
			\end{subfigure}		
			\caption{Diagram and Gauss diagram of planar virtual knotoid $w \in \mathcal W$.}
			\label{fig26}
		\end{center}
	\end{figure}
	
	\begin{theorem}\label{thm-UVW}
		For any planar virtual knotoids $u\in\mathcal{U}$, $v\in\mathcal{V}$ and $w\in\mathcal{W}$, $u$ is not homotopic to $w$, and $v$ is not homotopic to $w$. 
	\end{theorem}
	
	\begin{proof}
		We first prove that $u$ is non-homotopic to $w$. Let $D$ denote a diagram of $u$, and let $G(D)$ be its associated Gauss diagram. By conditions~(2) and~(3) of Construction~\ref{4.1} and the definition of the chord index, the first chord $c_1$ in $G(D)$ has chord index of $0$, while every remaining chord possesses a chord index whose absolute value equals $1$. For $w$, by contrast, every chord in its Gauss diagram has chord index of absolute value $1$. These indices follow in the same way from conditions~(2) and~(3) of Construction~\ref{4.3}. When we take the $2$-covering of $u$ in $\mathcal{U}$, the resulting diagram $D^{(2)}$ is  obtained by virtualizing all crossings except the one corresponding to $c_1$, and hence, by condition~(4) of Construction~\ref{4.1}, equivalent to the diagram of $\chi_1$ or $\chi_2$ in Figure~\ref{fig6}. In contrast, applying the $2$-covering to $w$ in $\mathcal{W}$ results in the trivial knotoid.
		Since both $\chi_1$ and $\chi_2$ are not homotopic to the trivial knotoid, it follows from Theorem~\ref{thm-ccov} that $u$ is not homotopic to~$w$.
		
		We next prove that $v$ is not homotopic to $w$. Let $D'$ be a diagram of $v$, and let $G(D')$ denote its corresponding Gauss diagram. By conditions~(2) and~(3) of Construction~\ref{4.2} and the definition of the chord index, the $(2n+1)$-th chord $c_{2n+1}$ in $G(D')$ carries an index of $0$, whereas the index of every other chord has an absolute value equal to $2$. The $3$-covering of $v$ produces a diagram $D'^{(3)}$ obtained by virtualizing all crossings except the one corresponding to $c_{2n+1}$, and hence, by condition~(4) of Construction~\ref{4.2}, equivalent to one of the diagrams of $\chi_1$ and $\chi_2$  in Figure~\ref{fig6}. Meanwhile, applying the $3$-covering to $w$ yields the trivial knotoid. Since $\chi_1$ and $\chi_2$ are both not homotopic to the trivial knotoid, it follows from Theorem~\ref{thm-ccov} that $v$ is not homotopic to $w$.
	\end{proof}
	
	The following corollaries provide explicit non-equivalent planar virtual knotoids within the same homotopy class. 
	\begin{corollary}\label{cor1}
		Each of the families $\mathcal{U}$ and $\mathcal{V}$ contains at least one planar virtual knotoid that is homotopic but not equivalent to $\chi_1$. More precisely, the planar virtual knotoids $u$ and $v$ shown in Figures~\ref{fig11} and~\ref{fig27}, respectively, have these properties. 
	\end{corollary}
	
	\begin{proof} 
		The Gauss diagrams in Figures~\ref{fig11} and~\ref{fig27} show that $u\in\mathcal{U}$ and $v\in\mathcal{V}$, respectively. For the diagram of $u$, we perform crossing changes at the crossings corresponding to chords $c_2$ and $c_5$ in its Gauss diagram. This yields the diagram shown in Figure~\ref{fig31}. 
		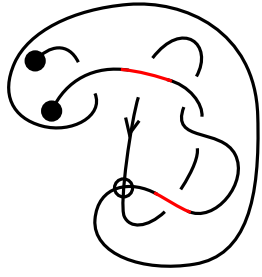
\begin{figure}[htbp]
			\centering
			\tikzset{every picture/.style={line width=0.75pt}}        
			\begin{tikzpicture}[x=0.75pt,y=0.75pt,yscale=-1,xscale=1]
				\draw [black, very thick]    (136.87,498.2) .. controls (141.53,509.54) and (125.17,518.57) .. (110.26,514.53) .. controls (95.34,510.49) and (89.53,498.57) .. (96.53,482.91) .. controls (103.53,467.24) and (125.53,455.87) .. (153.53,453.54) .. controls (181.53,451.2) and (217.2,466.54) .. (218.53,509.2) .. controls (219.87,551.87) and (212.87,579.54) .. (186.87,583.87) .. controls (160.87,588.2) and (135.53,578.87) .. (136.87,561.54) .. controls (138.2,544.2) and (154.87,541.2) .. (168.2,548.87) ; 
				\draw [black, very thick]    (115.13,507.04) .. controls (122.2,491.2) and (136.87,483.87) .. (153.87,486.54) ; 
				\draw [black, very thick]    (106.87,481.87) .. controls (113.87,473.87) and (123.53,472.2) .. (128.53,482.54) ;
				\draw [black, very thick]    (158.64,500.1) .. controls (155.19,511.01) and (148.84,553.73) .. (151.56,559.91) .. controls (154.29,566.09) and (163.02,565.91) .. (171.93,557.55) ;
				\draw [black, very thick]    (172.87,491.2) .. controls (184.6,494.04) and (191.2,503.87) .. (190.53,509.87) ; 
				\draw [black, very thick]    (181.47,505.04) .. controls (174.53,523.2) and (200.2,515.2) .. (207.2,528.87) .. controls (214.2,542.54) and (197.2,563.87) .. (183.2,557.2) ;
				\draw  [black, very thick]  (159.06,511.92) -- (154.13,518.56) -- (152.24,510.29) ;
				\draw  [fill={rgb, 255:red, 0; green, 0; blue, 0 }  ,fill opacity=1 ][black, very thick]  (110.61,507.04) .. controls (110.61,504.54) and (112.64,502.52) .. (115.13,502.52) .. controls (117.63,502.52) and (119.65,504.54) .. (119.65,507.04) .. controls (119.65,509.53) and (117.63,511.56) .. (115.13,511.56) .. controls (112.64,511.56) and (110.61,509.53) .. (110.61,507.04) -- cycle ;
				\draw  [fill={rgb, 255:red, 0; green, 0; blue, 0 }  ,fill opacity=1 ][black, very thick]  (102.35,481.87) .. controls (102.35,479.37) and (104.37,477.35) .. (106.87,477.35) .. controls (109.36,477.35) and (111.39,479.37) .. (111.39,481.87) .. controls (111.39,484.37) and (109.36,486.39) .. (106.87,486.39) .. controls (104.37,486.39) and (102.35,484.37) .. (102.35,481.87) -- cycle ; 
				\draw  [black, very thick]  (146.69,545.61) .. controls (146.69,543.11) and (148.71,541.09) .. (151.21,541.09) .. controls (153.71,541.09) and (155.73,543.11) .. (155.73,545.61) .. controls (155.73,548.1) and (153.71,550.13) .. (151.21,550.13) .. controls (148.71,550.13) and (146.69,548.1) .. (146.69,545.61) -- cycle ;
				\draw [black, very thick]    (187.78,489.69) .. controls (196.65,473.55) and (180.94,460.67) .. (165.75,480.45) ; 
				\draw [black, very thick]    (179.89,546.38) .. controls (186.57,536.14) and (188.79,529.26) .. (188.2,525.72) ; 
				\draw [color={rgb, 255:red, 208; green, 2; blue, 27 }  ,draw opacity=1 ][red, very thick]    (149.84,486.17) .. controls (157.94,486.67) and (169.34,489.77) .. (175.44,492.07) ; 
				\draw [color={rgb, 255:red, 208; green, 2; blue, 27 }  ,draw opacity=1 ][red, very thick]    (166.6,548.01) .. controls (169.18,549.19) and (180.85,556.76) .. (184.93,557.84) ;
			\end{tikzpicture}
			\caption{Two crossing changes are performed on the diagram of $u$.}
			\label{fig31}
		\end{figure}
		It is straightforward to verify that applying two $\Omega_2$-moves followed by one $\Omega_1^v$-move to this diagram produces the diagram of $\chi_1$ given in Figure~\ref{fig6}. Thus, $u$ is homotopic to $\chi_1$.
		
		Similarly, for the diagram of $v$, we perform crossing changes at the crossings corresponding to chords $c_1$ and $c_3$ in its Gauss diagram. After these crossing changes, the resulting diagram is equivalent to the diagram of $\chi_1$ in Figure~\ref{fig6}. Thus, $v$ is homotopic to $\chi_1$.
		
		Indeed, by calculating the odd writhe invariant $J$ defined in~\cite{GK17} for $u$ and $\chi_1$, one can see that
		$$
		J(u) = -4 \quad \text{and} \quad J(\chi_1) = 0.
		$$
		Since the odd writhe is an invariant of virtual knotoids, $u$ and $\chi_1$ are not equivalent. Furthermore, computing the values of the three-variable transcendental invariant~\cite{FLV25} associated with $v$ and $\chi_1$, we obtain
		$$
		H_{v}(t, y, z) = (-2)(t^2 + t^{-2} - 2)y^2 \quad \text{and} \quad H_{\chi_1}(t, y, z) = 0.
		$$
		Since $H$ is also an invariant of virtual knotoids, these distinct values show that $v$ and $\chi_1$ are not equivalent.
	\end{proof}
	
	\begin{corollary}\label{cor4}
		The family $\mathcal{W}$ contains at least one planar virtual knotoid that is homotopic but not equivalent to the trivial knotoid. In particular, the planar virtual knotoid $w$ shown in Figure~\ref{fig26} has these properties. 
	\end{corollary}
	
	\begin{proof}
		The Gauss diagram in Figure~\ref{fig26} shows that $w\in\mathcal{W}$. For the diagram of $w$, we perform crossing changes at the crossings corresponding to chords $c_2$ and $c_3$ in its Gauss diagram. The resulting diagram is equivalent to the trivial knotoid, so $w$ is homotopic to the trivial knotoid. A direct computation of the odd writhe invariant $J$ shows that 
		$$
		J(w) = -4 \quad \text{and} \quad J(O) = 0.
		$$
		Since  $J$ is an invariant of virtual knotoids, $w$ is not equivalent to the trivial knotoid. 
	\end{proof}
	
	\begin{corollary}\label{cor2}
		Let $K$ be a planar virtual knotoid, and $r \in \mathbb Z_+$ with $r>1$. If $K^{(r)}$ is equivalent to either $\chi_1$ or $\chi_2$ shown in Figure~\ref{fig6}, then $K$ is not homotopic to the trivial knotoid.
	\end{corollary}
	
	\begin{proof} 
		Suppose, to the contrary, that $K$ is homotopic to the trivial knotoid $O$. By Theorem~\ref{thm-ccov}, $K^{(r)}$ is homotopic to $O^{(r)}=O$. This contradicts the hypothesis, since neither $\chi_1$ nor $\chi_2$ is homotopic to~$O$.
	\end{proof}
	
	\begin{corollary}\label{cor3}
		Let $K$ be a planar virtual knotoid, and $r \in \mathbb Z_+$ with $r>1$. If $K^{(r)}$ is equivalent to the trivial knotoid, then $K$ is not homotopic to $\chi_1$ in Figure~\ref{fig6}.
	\end{corollary}
	
	\begin{proof}
		Suppose, to the contrary, that $K$ is homotopic to $\chi_1$. The unique crossing of $\chi_1$ has index $0$, so $\chi_1^{(r)}=\chi_1$. By Theorem~\ref{thm-ccov}, $K^{(r)}$ is homotopic to $\chi_1^{(r)}=\chi_1$. This contradicts the hypothesis that $K^{(r)}$ is equivalent to the trivial knotoid, because the trivial knotoid is not homotopic to $\chi_1$.
	\end{proof}
	
	With Theorem~\ref{thm-ccov}, we can also calculate the Gordian distances among homotopic planar virtual knotoids, as illustrated by the examples in the next section. We complete the section with the following questions.
	
	\begin{question}
		Does either $\mathcal{U}$ or $\mathcal{V}$ contain a planar virtual knotoid that is not homotopic to $\chi_1$?
	\end{question}
	
	\begin{question} 
		Does $\mathcal{W}$ contain a planar virtual knotoid that is not homotopic to the trivial knotoid? 
	\end{question}

	\section{Gordian distance: examples}\label{sec5}
	
	In this section, we give several examples for computing the Gordian distances between homotopic planar virtual knotoids. 
	
	\begin{example} \label{ex:3} {\rm
			The diagrams and Gauss diagrams of $K_1$ and $K_2$ are illustrated in Figure~\ref{fig15}~(a) and ~(b). 
			\begin{figure}[htbp]
				\begin{center}
					\begin{subfigure}{0.34\textwidth}
						\centering
						\tikzset{every picture/.style={line width=0.75pt}} 
						\begin{tikzpicture}[x=0.75pt,y=0.75pt,yscale=-0.8,xscale=0.8]
							\draw [black, very thick]  (112.09,116.67) .. controls (126.79,127.25) and (135.18,140) .. (137.44,155.52) ;
							\draw [black, very thick]  (89.46,100.27) .. controls (47.82,74.53) and (15.42,148.9) .. (80.4,191.91) ;
							\draw [black, very thick]  (180.1,164.01) .. controls (178.15,177.51) and (152.74,235.21) .. (101.12,205.89) ;
							\draw [black, very thick]  (64.6,163.15) .. controls (71.57,142.71) and (85.99,122.17) .. (103.94,104.52) .. controls (121.89,86.88) and (174.94,59.38) .. (178.65,138.4) ;
							\draw [black, very thick]  (55.66,186.08) .. controls (48.54,209.82) and (77.39,213.51) .. (96.33,194.53) .. controls (115.27,175.56) and (135,185.01) .. (143.98,194.67) ;
							\draw [black, very thick]  (166.12,216.35) .. controls (185.8,234.39) and (223.71,200.56) .. (225.11,155.29) .. controls (226.5,110.03) and (174.52,69.11) .. (153.13,68.74) .. controls (131.75,68.38) and (131.19,92.34) .. (139.21,109.13) .. controls (147.23,125.93) and (174.65,158.34) .. (211.76,158.61) .. controls (248.87,158.89) and (268.7,137.95) .. (228.27,125.79) ;
							\draw [black, very thick]  (208.37,119.95) .. controls (197.87,115.75) and (188.82,119.25) .. (190.69,131.75) ;
							\draw [black, very thick]  (207.72,197.89) -- (214.77,193.17) -- (213.97,201.27) ;
							\draw  [fill={rgb, 255:red, 0; green, 0; blue, 0 }  ,fill opacity=1 ][black, very thick]  (188.23,127.97) .. controls (190.32,126.6) and (193.12,127.2) .. (194.48,129.29) .. controls (195.84,131.38) and (195.25,134.18) .. (193.16,135.54) .. controls (191.07,136.9) and (188.27,136.31) .. (186.91,134.22) .. controls (185.54,132.13) and (186.14,129.33) .. (188.23,127.97) -- cycle ;
							\draw [black, very thick]  (222.52,154.13) .. controls (224.62,152.76) and (227.42,153.36) .. (228.78,155.45) .. controls (230.14,157.54) and (229.55,160.34) .. (227.46,161.7) .. controls (225.36,163.06) and (222.56,162.47) .. (221.2,160.38) .. controls (219.84,158.29) and (220.43,155.49) .. (222.52,154.13) -- cycle ;
							\draw [fill={rgb, 255:red, 0; green, 0; blue, 0 }  ,fill opacity=1 ][black, very thick]  (134.97,151.73) .. controls (137.07,150.37) and (139.87,150.96) .. (141.23,153.05) .. controls (142.59,155.15) and (142,157.95) .. (139.91,159.31) .. controls (137.81,160.67) and (135.01,160.08) .. (133.65,157.99) .. controls (132.29,155.89) and (132.88,153.09) .. (134.97,151.73) -- cycle ;
							\draw [black, very thick]  (132.14,83.1) .. controls (134.24,81.73) and (137.04,82.33) .. (138.4,84.42) .. controls (139.76,86.51) and (139.17,89.31) .. (137.08,90.67) .. controls (134.98,92.03) and (132.18,91.44) .. (130.82,89.35) .. controls (129.46,87.26) and (130.05,84.46) .. (132.14,83.1) -- cycle ;
							\draw (115,236) node [anchor=north west][inner sep=0.75pt] {$K_1$};
						\end{tikzpicture}	
					\end{subfigure}
					\qquad
					\begin{subfigure}{0.34\textwidth}
						\centering
						\tikzset{every picture/.style={line width=0.75pt}} 
						\begin{tikzpicture}[x=0.75pt,y=0.75pt,yscale=-0.8,xscale=0.8]
							\draw [color={rgb, 255:red, 0; green, 0; blue, 0 }  ,draw opacity=1 ][black, very thick]    (385.27,115.95) .. controls (399.94,126.58) and (408.28,139.36) .. (410.49,154.88) ; 
							\draw [color={rgb, 255:red, 0; green, 0; blue, 0 }  ,draw opacity=1 ][black, very thick]    (362.69,99.47) .. controls (330,79.12) and (299.81,122.36) .. (327.62,165.86) ;
							\draw [color={rgb, 255:red, 0; green, 0; blue, 0 }  ,draw opacity=1 ][black, very thick]    (453.12,163.51) .. controls (451.12,177) and (425.52,234.63) .. (374,205.14) ;
							\draw [color={rgb, 255:red, 0; green, 0; blue, 0 }  ,draw opacity=1 ][black, very thick]    (337.62,162.28) .. controls (344.66,141.85) and (359.15,121.36) .. (377.16,103.78) .. controls (395.16,86.19) and (448.31,58.87) .. (451.75,137.9) ;
							\draw [color={rgb, 255:red, 0; green, 0; blue, 0 }  ,draw opacity=1 ][black, very thick]    (338.06,161.05) .. controls (326.14,192.12) and (327.62,199.67) .. (333.57,203.48) .. controls (339.51,207.29) and (348.38,205.58) .. (353.83,203.63) ;
							\draw [color={rgb, 255:red, 0; green, 0; blue, 0 }  ,draw opacity=1 ][black, very thick]    (438.97,215.81) .. controls (458.59,233.92) and (496.61,200.21) .. (498.15,154.95) .. controls (499.7,109.69) and (447.85,68.59) .. (426.47,68.16) .. controls (405.09,67.73) and (404.45,91.68) .. (412.41,108.5) .. controls (420.37,125.33) and (447.68,157.83) .. (484.79,158.22) .. controls (521.9,158.62) and (541.81,137.75) .. (501.42,125.45) ;
							\draw [color={rgb, 255:red, 0; green, 0; blue, 0 }  ,draw opacity=1 ][black, very thick]    (481.53,119.55) .. controls (471.05,115.31) and (461.99,118.79) .. (463.82,131.3) ;
							\draw  [color={rgb, 255:red, 0; green, 0; blue, 0 }  ,draw opacity=1 ][black, very thick]  (480.63,197.48) -- (487.7,192.79) -- (486.86,200.89) ;
							\draw  [color={rgb, 255:red, 0; green, 0; blue, 0 }  ,draw opacity=1 ][fill={rgb, 255:red, 0; green, 0; blue, 0 }  ,fill opacity=1 ][black, very thick]  (461.37,127.5) .. controls (463.46,126.14) and (466.26,126.75) .. (467.62,128.84) .. controls (468.97,130.94) and (468.37,133.74) .. (466.27,135.09) .. controls (464.18,136.45) and (461.38,135.85) .. (460.02,133.75) .. controls (458.67,131.65) and (459.27,128.85) .. (461.37,127.5) -- cycle ;
							\draw  [color={rgb, 255:red, 0; green, 0; blue, 0 }  ,draw opacity=1 ][black, very thick]  (495.58,153.77) .. controls (497.67,152.42) and (500.47,153.02) .. (501.83,155.11) .. controls (503.18,157.21) and (502.58,160.01) .. (500.48,161.36) .. controls (498.39,162.72) and (495.59,162.12) .. (494.23,160.02) .. controls (492.88,157.92) and (493.48,155.13) .. (495.58,153.77) -- cycle ;
							\draw  [color={rgb, 255:red, 0; green, 0; blue, 0 }  ,draw opacity=1 ][fill={rgb, 255:red, 0; green, 0; blue, 0 }  ,fill opacity=1 ][black, very thick]  (408.03,151.09) .. controls (410.13,149.73) and (412.93,150.33) .. (414.28,152.43) .. controls (415.64,154.53) and (415.04,157.33) .. (412.94,158.68) .. controls (410.84,160.04) and (408.05,159.43) .. (406.69,157.34) .. controls (405.34,155.24) and (405.94,152.44) .. (408.03,151.09) -- cycle ;
							\draw  [color={rgb, 255:red, 0; green, 0; blue, 0 }  ,draw opacity=1 ][black, very thick]  (405.43,82.44) .. controls (407.53,81.09) and (410.33,81.69) .. (411.68,83.79) .. controls (413.04,85.88) and (412.44,88.68) .. (410.34,90.04) .. controls (408.24,91.39) and (405.44,90.79) .. (404.09,88.69) .. controls (402.73,86.6) and (403.34,83.8) .. (405.43,82.44) -- cycle ;
							\draw [color={rgb, 255:red, 0; green, 0; blue, 0 }  ,draw opacity=1 ][black, very thick]    (340.95,181.65) .. controls (352.31,189.79) and (362.47,199.75) .. (375.76,206.04) ;
							\draw [color={rgb, 255:red, 0; green, 0; blue, 0 }  ,draw opacity=1 ][black, very thick]    (370.37,192.86) .. controls (378.52,185.21) and (397.3,173.96) .. (416.9,194.06) ;
							\draw (398,236) node [anchor=north west][inner sep=0.75pt] {$K_2$};
						\end{tikzpicture}
					\end{subfigure}
					\begin{subfigure}{0.34\textwidth}
						\centering
						\tikzset{every picture/.style={line width=0.75pt}} 
						\begin{tikzpicture}[x=0.75pt,y=0.75pt,yscale=-1,xscale=1]
							\draw  [draw opacity=0][black, very thick]  (194.14,389.58) .. controls (184.38,413.58) and (161.15,430.47) .. (134.04,430.47) .. controls (98.15,430.47) and (69.04,400.85) .. (69.04,364.32) .. controls (69.04,327.79) and (98.15,298.17) .. (134.04,298.17) .. controls (161.33,298.17) and (184.68,315.28) .. (194.32,339.52) -- (134.04,364.32) -- cycle ; \draw  [line width=1.5]  (194.14,389.58) .. controls (184.38,413.58) and (161.15,430.47) .. (134.04,430.47) .. controls (98.15,430.47) and (69.04,400.85) .. (69.04,364.32) .. controls (69.04,327.79) and (98.15,298.17) .. (134.04,298.17) .. controls (161.33,298.17) and (184.68,315.28) .. (194.32,339.52) ;  
							\draw  [fill={rgb, 255:red, 0; green, 0; blue, 0 }  ,fill opacity=1 ][black, very thick]  (189.62,389.58) .. controls (189.62,387.08) and (191.64,385.06) .. (194.14,385.06) .. controls (196.63,385.06) and (198.66,387.08) .. (198.66,389.58) .. controls (198.66,392.08) and (196.63,394.1) .. (194.14,394.1) .. controls (191.64,394.1) and (189.62,392.08) .. (189.62,389.58) -- cycle ;
							\draw  [fill={rgb, 255:red, 0; green, 0; blue, 0 }  ,fill opacity=1 ][black, very thick]  (189.8,339.52) .. controls (189.8,337.03) and (191.83,335) .. (194.32,335) .. controls (196.82,335) and (198.84,337.03) .. (198.84,339.52) .. controls (198.84,342.02) and (196.82,344.04) .. (194.32,344.04) .. controls (191.83,344.04) and (189.8,342.02) .. (189.8,339.52) -- cycle ;
							\draw [black, very thick, ->]    (128.82,430.27) -- (185.31,405.06) ;
							\draw [black, very thick, ->]    (161.55,424.27) -- (77.5,331.86) ;
							\draw [black, very thick, ->]    (70.28,351.16) -- (185.73,324.27) ;
							\draw [black, very thick, ->]    (93.55,312.62) -- (102.64,422.25) ;
							\draw [red, very thick, ->]    (166.28,306.82) -- (70.28,377) ;
							\draw [red, very thick, ->]    (82.09,404.07) -- (137,298.44) ;
							\draw (49,371.76) node [anchor=north west][inner sep=0.75pt]  [font=\large]  {$c_{2}$};
							\draw (129.86,281) node [anchor=north west][inner sep=0.75pt]  [font=\large]  {$c_{3}$};
							\draw (90.6,426.52) node [anchor=north west][inner sep=0.75pt]  [font=\large]  {$c_{4}$};
							\draw (165.73,294.37) node [anchor=north west][inner sep=0.75pt]  [font=\scriptsize]  {$-$};
							\draw (53.22,344.69) node [anchor=north west][inner sep=0.75pt]  [font=\scriptsize]  {$-$};
							\draw (71.02,407.92) node [anchor=north west][inner sep=0.75pt]  [font=\scriptsize]  {$-$};
							\draw (84.2,298.13) node [anchor=north west][inner sep=0.75pt]  [font=\scriptsize]  {$-$};
							\draw (57.61,317.97) node [anchor=north west][inner sep=0.75pt]  [font=\large]  {$c_{5}$};
							\draw (188.8,312.16) node [anchor=north west][inner sep=0.75pt]  [font=\large]  {$c_{1}$};
							\draw (119.47,435.74) node [anchor=north west][inner sep=0.75pt]  [font=\scriptsize]  {$+$};
							\draw (162.65,430) node [anchor=north west][inner sep=0.75pt]  [font=\scriptsize]  {$+$};
							\draw (191.85,398.76) node [anchor=north west][inner sep=0.75pt]  [font=\large]  {$c_{6}$};
							\draw (120,445) node [anchor=north west][inner sep=0.75pt]    {(a)};
						\end{tikzpicture}	
					\end{subfigure}
					\qquad 
					\begin{subfigure}{0.34\textwidth}
						\centering
						\tikzset{every picture/.style={line width=0.75pt}} 
						\begin{tikzpicture}[x=0.75pt,y=0.75pt,yscale=-1,xscale=1]
							\draw  [draw opacity=0][line width=1.5]  (477.14,389.58) .. controls (467.38,413.58) and (444.15,430.47) .. (417.04,430.47) .. controls (381.15,430.47) and (352.04,400.85) .. (352.04,364.32) .. controls (352.04,327.79) and (381.15,298.17) .. (417.04,298.17) .. controls (444.33,298.17) and (467.68,315.28) .. (477.32,339.52) -- (417.04,364.32) -- cycle ; \draw  [line width=1.5]  (477.14,389.58) .. controls (467.38,413.58) and (444.15,430.47) .. (417.04,430.47) .. controls (381.15,430.47) and (352.04,400.85) .. (352.04,364.32) .. controls (352.04,327.79) and (381.15,298.17) .. (417.04,298.17) .. controls (444.33,298.17) and (467.68,315.28) .. (477.32,339.52) ;  
							\draw  [fill={rgb, 255:red, 0; green, 0; blue, 0 }  ,fill opacity=1 ][line width=1.5]  (472.62,389.58) .. controls (472.62,387.08) and (474.64,385.06) .. (477.14,385.06) .. controls (479.63,385.06) and (481.66,387.08) .. (481.66,389.58) .. controls (481.66,392.08) and (479.63,394.1) .. (477.14,394.1) .. controls (474.64,394.1) and (472.62,392.08) .. (472.62,389.58) -- cycle ;
							\draw  [fill={rgb, 255:red, 0; green, 0; blue, 0 }  ,fill opacity=1 ][line width=1.5]  (472.8,339.52) .. controls (472.8,337.03) and (474.83,335) .. (477.32,335) .. controls (479.82,335) and (481.84,337.03) .. (481.84,339.52) .. controls (481.84,342.02) and (479.82,344.04) .. (477.32,344.04) .. controls (474.83,344.04) and (472.8,342.02) .. (472.8,339.52) -- cycle ;
							\draw [black, very thick, ->]    (411.82,430.27) -- (468.31,405.06) ;
							\draw [black, very thick, ->]    (444.55,424.27) -- (360.5,331.86) ;
							\draw [black, very thick, ->]    (353.28,351.16) -- (468.73,324.27) ;
							\draw [black, very thick, ->]    (376.55,312.62) -- (385.64,422.25) ;
							\draw [red, very thick, <-]    (449.28,306.82) -- (353.28,377) ;
							\draw [red, very thick, <-]    (365.09,404.07) -- (420,298.44) ;
							\draw (337.5,378) node [anchor=north west][inner sep=0.75pt]  [font=\scriptsize]  {$+$};
							\draw (415,285) node [anchor=north west][inner sep=0.75pt]  [font=\scriptsize]  {$+$};
							\draw (378,426.52) node [anchor=north west][inner sep=0.75pt]  [font=\large]  {$c_{4}'$};
							\draw (448.73,289) node [anchor=north west][inner sep=0.75pt]  [font=\large]  {$c_{2}'$};
							\draw (336.22,348) node [anchor=north west][inner sep=0.75pt]  [font=\scriptsize]  {$-$};
							\draw (350,407.92) node [anchor=north west][inner sep=0.75pt]  [font=\large]  {$c_{3}'$};
							\draw (367.2,298.13) node [anchor=north west][inner sep=0.75pt]  [font=\scriptsize]  {$-$};
							\draw (340.61,315) node [anchor=north west][inner sep=0.75pt]  [font=\large]  {$c_{5}'$};
							\draw (472,311) node [anchor=north west][inner sep=0.75pt]  [font=\large]  {$c_{1}'$};
							\draw (402.47,435.74) node [anchor=north west][inner sep=0.75pt]  [font=\scriptsize]  {$+$};
							\draw (445.65,430) node [anchor=north west][inner sep=0.75pt]  [font=\scriptsize]  {$+$};
							\draw (474.85,395) node [anchor=north west][inner sep=0.75pt]  [font=\large]  {$c_{6}'$};
							\draw (405,445) node [anchor=north west][inner sep=0.75pt]    {(b)};
						\end{tikzpicture}	
					\end{subfigure}		
					\caption{Diagrams and Gauss diagrams of planar virtual knotoids $ K_1$ and $K_2$.}
					\label{fig15}
				\end{center}
			\end{figure}
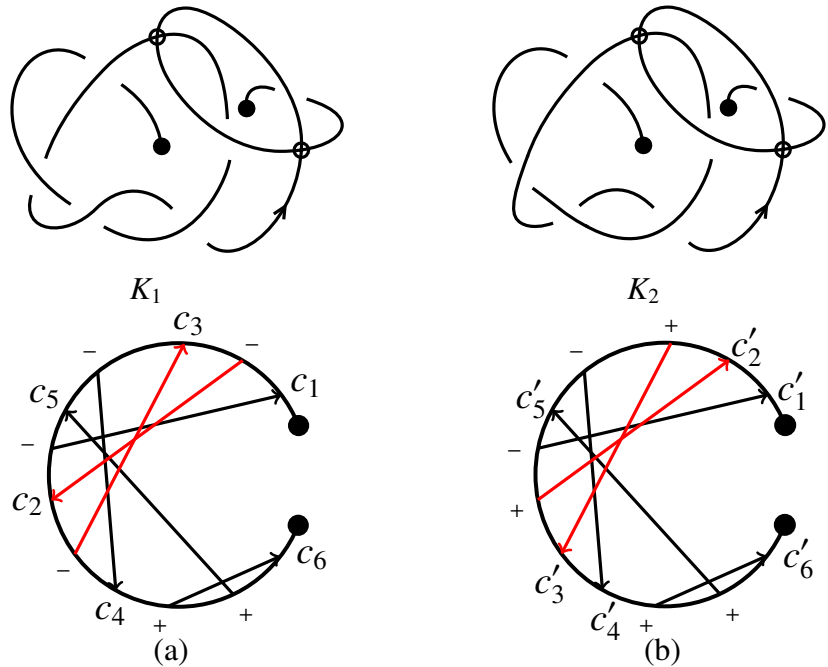
			Direct computation yields the chord indices of the Gauss diagram of $K_1$,
			$$
			\Ind(c_1)=-2,\quad \Ind(c_2)=2,\quad \Ind(c_3)=-2,\quad \Ind(c_4)=2,\quad \Ind(c_5)=1, \quad \Ind(c_6)=-1;
			$$
			we then calculate those of the Gauss diagram of $K_2$,
			$$
			\Ind(c_1')=-2,\quad \Ind(c_2')=-2,\quad \Ind(c_3')=2,\quad \Ind(c_4')=2,\quad \Ind(c_5')=1, \quad \Ind(c_6')=-1.
			$$
			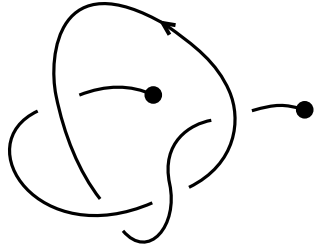
\begin{figure}[htbp]
				\begin{center}
					\tikzset{every picture/.style={line width=0.7pt}} 
					\begin{tikzpicture}[x=0.75pt,y=0.75pt,yscale=-0.8,xscale=0.8]
						\draw [black, very thick]    (97.65,623.44) .. controls (114.53,616.88) and (129.79,616.81) .. (144.03,623.38) ;
						\draw [black, very thick]    (71.55,633.45) .. controls (27.26,654.31) and (71.91,722.04) .. (143.41,691.04) ;
						\draw [black, very thick]    (151,578.9) .. controls (170.56,592.67) and (183.89,603.12) .. (191.67,620.9) .. controls (199.44,638.67) and (196.33,666.67) .. (166.43,681.31) ;
						\draw [black, very thick]    (110.69,688.6) .. controls (97.36,671.6) and (88.02,648.31) .. (83.02,623.64) .. controls (78.02,598.97) and (83.91,539.51) .. (152.16,579.51) ;
						\draw [black, very thick]    (125.03,708.59) .. controls (141.04,727.51) and (159.88,705.34) .. (154.3,679.12) .. controls (148.73,652.9) and (167.41,641.51) .. (180.41,639.26) ;
						\draw  [line width=1.5]  (154.19,585.58) -- (149.96,578.23) -- (157.99,579.57) ;
						\draw  [fill={rgb, 255:red, 0; green, 0; blue, 0 }  ,fill opacity=1 ][black, very thick]  (234.48,632.67) .. controls (234.48,630.18) and (236.5,628.15) .. (239,628.15) .. controls (241.5,628.15) and (243.52,630.18) .. (243.52,632.67) .. controls (243.52,635.17) and (241.5,637.19) .. (239,637.19) .. controls (236.5,637.19) and (234.48,635.17) .. (234.48,632.67) -- cycle ;
						\draw  [fill={rgb, 255:red, 0; green, 0; blue, 0 }  ,fill opacity=1 ][black, very thick]  (139.51,623.38) .. controls (139.51,620.89) and (141.54,618.86) .. (144.03,618.86) .. controls (146.53,618.86) and (148.55,620.89) .. (148.55,623.38) .. controls (148.55,625.88) and (146.53,627.9) .. (144.03,627.9) .. controls (141.54,627.9) and (139.51,625.88) .. (139.51,623.38) -- cycle ;
						\draw [black, very thick]    (205.88,633.26) .. controls (221,628.9) and (226.11,628.9) .. (239,632.67) ;
					\end{tikzpicture}
					\caption{The diagram of $K_1^{(2)}$.}
					\label{fig16}
				\end{center}
			\end{figure}		
			However, the diagram of $K_1^{(2)}$ is presented in Figure~\ref{fig16}, and $K_2^{(2)}$ is the trivial knotoid. The diagrams of $K_1^{(2)}$ and $K_2^{(2)}$ also differ by two crossing changes and two $\Omega_2$-moves, so $d_G(K_1^{(2)}, K_2^{(2)}) \le 2$. By computing the values of the three-variable transcendental invariant associated with $K_1^{(2)}$ and $K_2^{(2)}$, we obtain
			$$
			H_{K_1^{(2)}}(t, y, z) = (-2)(t + t^{-1} - 2)y, \qquad H_{K_2^{(2)}}(t, y, z) = 0.
			$$
			By Theorem~\ref{thm-H}, we conclude that $d_G(K_1^{(2)}, K_2^{(2)})\geq 2$. Thus, $d_G(K_1^{(2)}, K_2^{(2)}) = 2$. By Theorem~\ref{thm-ccov}, we obtain $d_G(K_1, K_2)\geq d_G(K_1^{(2)}, K_2^{(2)}) = 2$. We can also observe that the diagrams of $K_1$ and $K_2$ differ by two crossing changes, so $d_G(K_1, K_2) \le 2$. 
			Hence, $d_G(K_1, K_2) = d_G(K_1^{(2)}, K_2^{(2)}) = 2$.}
	\end{example}
	
	We next examine a more involved example.
	
	\begin{example} \label{ex:9} {\rm
			The diagrams and Gauss diagrams of $K_3$ and $K_4$ are illustrated in Figure~\ref{fig23}~(a) and ~(b). 
			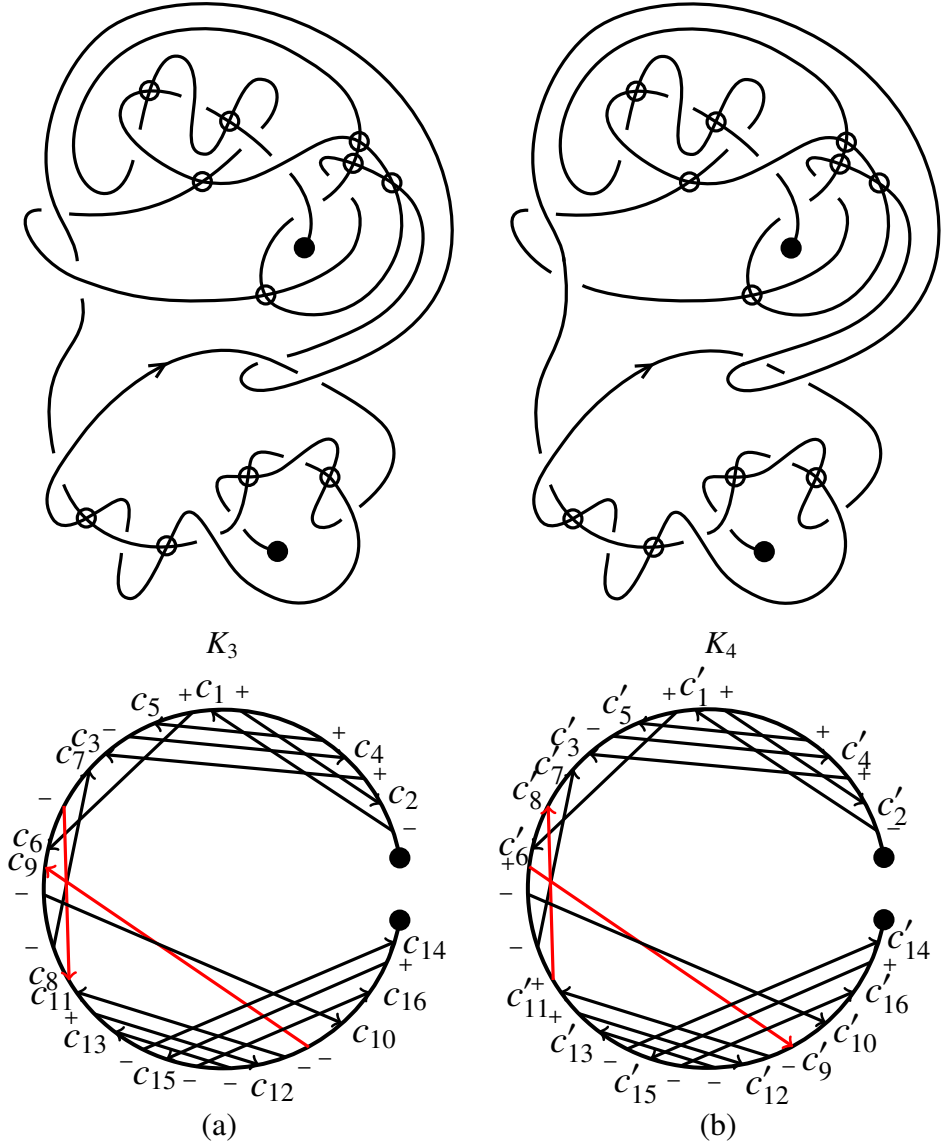
\begin{figure}[htbp]
				\begin{center}
					\begin{subfigure}{0.34\textwidth}
						\centering
						\tikzset{every picture/.style={line width=0.75pt}} 
						\begin{tikzpicture}[x=0.75pt,y=0.75pt,yscale=-1,xscale=1]
							\draw [black, very thick]    (189,1021.17) .. controls (199.37,1032.7) and (202.79,1048.33) .. (196.16,1058.5) ;
							\draw [black, very thick]    (146.97,988.39) .. controls (158.25,994.32) and (171.2,1003.19) .. (179.62,1012.35) ;
							\draw [black, very thick]    (134.46,982.46) .. controls (126.2,979.01) and (105.07,976.47) .. (103.81,990.82) .. controls (102.54,1005.18) and (133.18,1023.83) .. (146.95,1025.65) .. controls (160.72,1027.47) and (168.22,1025.23) .. (184.81,1016.47) .. controls (201.41,1007.71) and (209.12,1000.74) .. (220.5,1003.82) .. controls (231.87,1006.89) and (255.21,1037.56) .. (240.04,1066.52) .. controls (224.87,1095.49) and (196.68,1095.27) .. (182.52,1088.09) .. controls (168.36,1080.91) and (173.59,1055.51) .. (190.59,1043.71) ;
							\draw [black, very thick]    (110.28,1014.49) .. controls (101.04,1051.2) and (65.28,1022.17) .. (87.5,981.11) .. controls (109.71,940.05) and (160.7,944.09) .. (188.47,959.3) .. controls (216.25,974.51) and (221.22,984.41) .. (223.59,995.85) .. controls (225.96,1007.29) and (219.88,1025.57) .. (205.49,1034.67) ;
							\draw [black, very thick]    (174.65,998.6) .. controls (185.69,987.59) and (184.83,973.82) .. (173.43,973.93) .. controls (162.03,974.04) and (156.4,1011.49) .. (146.54,1011.53) .. controls (136.67,1011.57) and (139.41,999.91) .. (140.93,985.64) .. controls (142.45,971.38) and (138.69,961.1) .. (130.46,962.44) .. controls (122.23,963.79) and (116.81,982.98) .. (113.71,999.21) ;
							\draw [black, very thick]    (164.39,1009.08) .. controls (139.6,1033.47) and (115.45,1044.56) .. (78.15,1041.51) ;
							\draw [black, very thick]    (64.28,1039.55) .. controls (46.05,1038.53) and (60.91,1066.01) .. (82.92,1074.22) .. controls (104.94,1082.43) and (120.71,1086.79) .. (158.01,1084.64) .. controls (195.3,1082.48) and (217.04,1066.42) .. (220.91,1058.68) .. controls (224.79,1050.94) and (226.01,1039.64) .. (219.82,1033.52) ;
							\draw [black, very thick]    (187.43,1112.84) .. controls (193.04,1111.87) and (209.63,1105.66) .. (222.95,1099.18) .. controls (236.28,1092.7) and (244.18,1084.63) .. (249.32,1075.74) .. controls (254.45,1066.85) and (259.95,1050.55) .. (250.88,1036.22) .. controls (241.8,1021.88) and (210.87,1011.49) .. (206.27,1013.4) .. controls (201.67,1015.31) and (206.05,1019.77) .. (208.37,1022.25) ;
							\draw [black, very thick]    (172.81,1116.55) .. controls (169.18,1117.69) and (160.05,1123.17) .. (166.84,1128.17) .. controls (173.63,1133.18) and (189.4,1127.82) .. (210.28,1119.92) .. controls (231.16,1112.02) and (260.99,1098.48) .. (268.56,1066.08) .. controls (276.14,1033.68) and (259.79,982.91) .. (211.99,954.59) .. controls (164.2,926.28) and (115.08,935.12) .. (96.34,949.51) .. controls (77.6,963.91) and (64.92,993.08) .. (66.7,1018.12) .. controls (68.47,1043.15) and (80.94,1049.82) .. (82.25,1065.23) ; 
							\draw [black, very thick]    (84.61,1083.98) .. controls (89.94,1111.46) and (59.5,1109.26) .. (70.46,1161.65) ;
							\draw [black, very thick]    (75.38,1177.28) .. controls (85.71,1204.05) and (114.73,1213.91) .. (142.23,1206.08) ;
							\draw [black, very thick]    (185.77,1163.29) .. controls (177.45,1160.74) and (172.23,1159.51) .. (169.7,1164.2) .. controls (167.17,1168.88) and (167.9,1199.36) .. (153.98,1201.74) ;
							\draw [black, very thick]    (196.73,1166.83) .. controls (209.52,1172) and (210.27,1174.77) .. (215.37,1182.7) .. controls (220.46,1190.62) and (227.86,1206.15) .. (220.09,1220.51) .. controls (212.32,1234.88) and (195.96,1241.09) .. (179.63,1233.59) .. controls (163.3,1226.1) and (159.86,1219.24) .. (152.65,1207.8) .. controls (145.43,1196.35) and (140.87,1188.5) .. (136.14,1191.91) .. controls (131.42,1195.32) and (118.18,1235.36) .. (108.64,1234.55) .. controls (99.09,1233.73) and (103.97,1218.33) .. (104.42,1211.68) ;
							\draw [black, very thick]    (204.91,1126.57) .. controls (215.23,1132.17) and (234.38,1142.88) .. (238.59,1151.46) .. controls (242.8,1160.05) and (243.97,1172.39) .. (223.81,1188.6) ; 
							\draw [black, very thick]    (214.82,1194.89) .. controls (208.61,1200.56) and (199.41,1200.36) .. (200.2,1193.47) .. controls (200.99,1186.58) and (216.54,1163.93) .. (212.12,1156.83) .. controls (207.69,1149.74) and (197.32,1160.48) .. (185.73,1168.82) .. controls (174.13,1177.17) and (158.12,1171.12) .. (153.95,1174.17) .. controls (149.78,1177.21) and (154.24,1185.94) .. (158.69,1190.8) ;
							\draw [black, very thick]    (165.75,1199.84) .. controls (170.05,1205.7) and (174.73,1209.2) .. (182.69,1210.88) ;
							\draw [black, very thick]    (107.1,1201.14) .. controls (107.8,1195.65) and (110.45,1185.41) .. (104.51,1184.83) .. controls (98.57,1184.24) and (77.28,1206.08) .. (69.92,1194.5) .. controls (62.56,1182.92) and (70.05,1171.95) .. (84.32,1152.79) .. controls (98.59,1133.63) and (118.3,1119.64) .. (132.55,1114.5) .. controls (146.8,1109.35) and (158.24,1108.92) .. (169.04,1111.56) .. controls (179.83,1114.21) and (183.46,1116.8) .. (192.83,1121.23) ;
							\draw  [black, very thick]  (117.81,1117.76) -- (126.25,1116.88) -- (121.76,1123.67) ;
							\draw  [fill={rgb, 255:red, 0; green, 0; blue, 0 }  ,fill opacity=1 ][black, very thick]  (191.64,1058.5) .. controls (191.64,1056) and (193.67,1053.98) .. (196.16,1053.98) .. controls (198.66,1053.98) and (200.68,1056) .. (200.68,1058.5) .. controls (200.68,1061) and (198.66,1063.02) .. (196.16,1063.02) .. controls (193.67,1063.02) and (191.64,1061) .. (191.64,1058.5) -- cycle ;
							\draw  [black, very thick]  (113.91,979.63) .. controls (113.91,977.13) and (115.93,975.11) .. (118.43,975.11) .. controls (120.92,975.11) and (122.95,977.13) .. (122.95,979.63) .. controls (122.95,982.13) and (120.92,984.15) .. (118.43,984.15) .. controls (115.93,984.15) and (113.91,982.13) .. (113.91,979.63) -- cycle ;
							\draw  [fill={rgb, 255:red, 0; green, 0; blue, 0 }  ,fill opacity=1 ][black, very thick]  (178.17,1210.88) .. controls (178.17,1208.38) and (180.19,1206.36) .. (182.69,1206.36) .. controls (185.18,1206.36) and (187.21,1208.38) .. (187.21,1210.88) .. controls (187.21,1213.37) and (185.18,1215.4) .. (182.69,1215.4) .. controls (180.19,1215.4) and (178.17,1213.37) .. (178.17,1210.88) -- cycle ; 
							\draw  [black, very thick]  (154.18,995.02) .. controls (154.18,992.52) and (156.21,990.5) .. (158.7,990.5) .. controls (161.2,990.5) and (163.22,992.52) .. (163.22,995.02) .. controls (163.22,997.52) and (161.2,999.54) .. (158.7,999.54) .. controls (156.21,999.54) and (154.18,997.52) .. (154.18,995.02) -- cycle ;
							\draw  [black, very thick]  (219.18,1005.19) .. controls (219.18,1002.69) and (221.21,1000.67) .. (223.7,1000.67) .. controls (226.2,1000.67) and (228.22,1002.69) .. (228.22,1005.19) .. controls (228.22,1007.68) and (226.2,1009.71) .. (223.7,1009.71) .. controls (221.21,1009.71) and (219.18,1007.68) .. (219.18,1005.19) -- cycle ;
							\draw  [black, very thick]  (216.13,1016.13) .. controls (216.13,1013.63) and (218.15,1011.61) .. (220.65,1011.61) .. controls (223.14,1011.61) and (225.17,1013.63) .. (225.17,1016.13) .. controls (225.17,1018.63) and (223.14,1020.65) .. (220.65,1020.65) .. controls (218.15,1020.65) and (216.13,1018.63) .. (216.13,1016.13) -- cycle ;
							\draw  [black, very thick]  (235.57,1025.91) .. controls (235.57,1023.41) and (237.6,1021.39) .. (240.09,1021.39) .. controls (242.59,1021.39) and (244.61,1023.41) .. (244.61,1025.91) .. controls (244.61,1028.4) and (242.59,1030.43) .. (240.09,1030.43) .. controls (237.6,1030.43) and (235.57,1028.4) .. (235.57,1025.91) -- cycle ;
							\draw  [black, very thick]  (163.57,1173.3) .. controls (163.57,1170.8) and (165.6,1168.78) .. (168.09,1168.78) .. controls (170.59,1168.78) and (172.61,1170.8) .. (172.61,1173.3) .. controls (172.61,1175.79) and (170.59,1177.82) .. (168.09,1177.82) .. controls (165.6,1177.82) and (163.57,1175.79) .. (163.57,1173.3) -- cycle ;
							\draw  [black, very thick]  (204.46,1173.63) .. controls (204.46,1171.13) and (206.48,1169.11) .. (208.98,1169.11) .. controls (211.48,1169.11) and (213.5,1171.13) .. (213.5,1173.63) .. controls (213.5,1176.13) and (211.48,1178.15) .. (208.98,1178.15) .. controls (206.48,1178.15) and (204.46,1176.13) .. (204.46,1173.63) -- cycle ;
							\draw  [black, very thick]  (82.24,1194.13) .. controls (82.24,1191.63) and (84.26,1189.61) .. (86.76,1189.61) .. controls (89.25,1189.61) and (91.28,1191.63) .. (91.28,1194.13) .. controls (91.28,1196.63) and (89.25,1198.65) .. (86.76,1198.65) .. controls (84.26,1198.65) and (82.24,1196.63) .. (82.24,1194.13) -- cycle ;
							\draw  [black, very thick]  (122.57,1208.07) .. controls (122.57,1205.58) and (124.6,1203.55) .. (127.09,1203.55) .. controls (129.59,1203.55) and (131.61,1205.58) .. (131.61,1208.07) .. controls (131.61,1210.57) and (129.59,1212.59) .. (127.09,1212.59) .. controls (124.6,1212.59) and (122.57,1210.57) .. (122.57,1208.07) -- cycle ;
							\draw  [black, very thick]  (172.29,1082.19) .. controls (172.29,1079.69) and (174.32,1077.67) .. (176.81,1077.67) .. controls (179.31,1077.67) and (181.33,1079.69) .. (181.33,1082.19) .. controls (181.33,1084.68) and (179.31,1086.71) .. (176.81,1086.71) .. controls (174.32,1086.71) and (172.29,1084.68) .. (172.29,1082.19) -- cycle ;
							\draw  [black, very thick]  (140.46,1025.3) .. controls (140.46,1022.8) and (142.48,1020.78) .. (144.98,1020.78) .. controls (147.48,1020.78) and (149.5,1022.8) .. (149.5,1025.3) .. controls (149.5,1027.79) and (147.48,1029.82) .. (144.98,1029.82) .. controls (142.48,1029.82) and (140.46,1027.79) .. (140.46,1025.3) -- cycle ;
							\draw (145,1250) node [anchor=north west][inner sep=0.75pt] {$K_3$};
						\end{tikzpicture}
					\end{subfigure}
					\qquad
					\begin{subfigure}{0.34\textwidth}
						\centering
						\tikzset{every picture/.style={line width=0.75pt}} 
						\begin{tikzpicture}[x=0.75pt,y=0.75pt,yscale=-1,xscale=1]
							\draw [black, very thick]    (470.41,1117.04) .. controls (481.24,1113.16) and (508.63,1105.65) .. (521.95,1099.17) .. controls (535.28,1092.69) and (543.18,1084.62) .. (548.32,1075.73) .. controls (553.45,1066.84) and (558.95,1050.55) .. (549.88,1036.21) .. controls (540.8,1021.88) and (509.87,1011.48) .. (505.27,1013.39) .. controls (500.67,1015.3) and (505.05,1019.77) .. (507.37,1022.25) ;
							\draw [black, very thick]    (488,1021.17) .. controls (498.37,1032.69) and (501.79,1048.32) .. (495.16,1058.49) ;
							\draw [black, very thick]    (445.97,988.38) .. controls (457.25,994.32) and (470.2,1003.18) .. (478.62,1012.35) ;
							\draw [black, very thick]    (433.46,982.45) .. controls (425.2,979) and (404.07,976.47) .. (402.81,990.82) .. controls (401.54,1005.17) and (432.18,1023.83) .. (445.95,1025.64) .. controls (459.72,1027.46) and (467.22,1025.23) .. (483.81,1016.47) .. controls (500.41,1007.71) and (508.12,1000.74) .. (519.5,1003.81) .. controls (530.87,1006.89) and (554.21,1037.55) .. (539.04,1066.52) .. controls (523.87,1095.48) and (495.68,1095.27) .. (481.52,1088.08) .. controls (467.36,1080.9) and (472.59,1055.5) .. (489.59,1043.71) ;
							\draw [black, very thick]    (409.28,1014.48) .. controls (400.04,1051.2) and (364.28,1022.17) .. (386.5,981.1) .. controls (408.71,940.04) and (459.7,944.09) .. (487.47,959.3) .. controls (515.25,974.51) and (520.22,984.41) .. (522.59,995.84) .. controls (524.96,1007.28) and (518.88,1025.56) .. (504.49,1034.66) ;
							\draw [black, very thick]    (473.65,998.59) .. controls (484.69,987.59) and (483.83,973.81) .. (472.43,973.93) .. controls (461.03,974.04) and (455.4,1011.48) .. (445.54,1011.52) .. controls (435.67,1011.56) and (438.41,999.9) .. (439.93,985.64) .. controls (441.45,971.37) and (437.69,961.09) .. (429.46,962.44) .. controls (421.23,963.78) and (415.81,982.98) .. (412.71,999.21) ;
							\draw [black, very thick]    (463.39,1009.08) .. controls (438.6,1033.47) and (414.45,1044.55) .. (377.15,1041.5) ;
							\draw [black, very thick]    (363.28,1039.55) .. controls (347.09,1038.64) and (357,1060.21) .. (374.84,1070.8) ;
							\draw [black, very thick]    (471.81,1116.54) .. controls (468.18,1117.68) and (459.05,1123.16) .. (465.84,1128.17) .. controls (472.63,1133.17) and (488.4,1127.81) .. (509.28,1119.91) .. controls (530.16,1112.01) and (559.99,1098.48) .. (567.56,1066.07) .. controls (575.14,1033.67) and (558.79,982.9) .. (510.99,954.59) .. controls (463.2,926.27) and (414.08,935.11) .. (395.34,949.51) .. controls (376.6,963.9) and (363.92,993.08) .. (365.7,1018.11) .. controls (367.47,1043.14) and (379.94,1049.81) .. (381.25,1065.22) ;
							\draw [black, very thick]    (381,1063.08) .. controls (389.25,1118.95) and (358.5,1109.25) .. (369.46,1161.65) ;
							\draw [black, very thick]    (374.38,1177.28) .. controls (384.71,1204.04) and (413.73,1213.9) .. (441.23,1206.07) ;
							\draw [black, very thick]    (484.77,1163.28) .. controls (476.45,1160.73) and (471.23,1159.5) .. (468.7,1164.19) .. controls (466.17,1168.88) and (466.9,1199.35) .. (452.98,1201.73) ;
							\draw [black, very thick]    (495.73,1166.83) .. controls (508.52,1171.99) and (509.27,1174.77) .. (514.37,1182.69) .. controls (519.46,1190.62) and (526.86,1206.15) .. (519.09,1220.51) .. controls (511.32,1234.87) and (494.96,1241.08) .. (478.63,1233.59) .. controls (462.3,1226.09) and (458.86,1219.24) .. (451.65,1207.79) .. controls (444.43,1196.35) and (439.87,1188.5) .. (435.14,1191.91) .. controls (430.42,1195.32) and (417.18,1235.36) .. (407.64,1234.54) .. controls (398.09,1233.73) and (402.97,1218.32) .. (403.42,1211.67) ;
							\draw [black, very thick]    (503.91,1126.57) .. controls (514.23,1132.17) and (533.38,1142.88) .. (537.59,1151.46) .. controls (541.8,1160.04) and (542.97,1172.38) .. (522.81,1188.6) ;
							\draw [black, very thick]    (513.82,1194.88) .. controls (507.61,1200.55) and (498.41,1200.35) .. (499.2,1193.46) .. controls (499.99,1186.57) and (515.54,1163.92) .. (511.12,1156.83) .. controls (506.69,1149.73) and (496.32,1160.47) .. (484.73,1168.82) .. controls (473.13,1177.16) and (457.12,1171.11) .. (452.95,1174.16) .. controls (448.78,1177.21) and (453.24,1185.93) .. (457.69,1190.79) ;
							\draw [black, very thick]    (464.75,1199.83) .. controls (469.05,1205.69) and (473.73,1209.19) .. (481.69,1210.87) ;
							\draw [black, very thick]    (406.1,1201.13) .. controls (406.8,1195.64) and (409.45,1185.41) .. (403.51,1184.82) .. controls (397.57,1184.23) and (376.28,1206.08) .. (368.92,1194.5) .. controls (361.56,1182.91) and (369.05,1171.95) .. (383.32,1152.78) .. controls (397.59,1133.62) and (417.3,1119.63) .. (431.55,1114.49) .. controls (445.8,1109.34) and (459.17,1108.95) .. (471,1112.29) ;
							\draw [black, very thick]    (390.47,1077.29) .. controls (399.2,1081.45) and (424.69,1086.49) .. (457.01,1084.63) .. controls (489.32,1082.77) and (515.2,1069.05) .. (519.91,1058.67) .. controls (524.62,1048.3) and (525.01,1039.63) .. (518.82,1033.52) ;
							\draw [black, very thick]    (483,1116.95) .. controls (488.37,1119.62) and (489.4,1120.08) .. (491.83,1121.23) ;
							\draw  [black, very thick]  (416.69,1117.75) -- (425.15,1117.21) -- (420.4,1123.81) ;
							\draw  [black, very thick]  (412.91,979.78) .. controls (412.91,977.28) and (414.93,975.26) .. (417.43,975.26) .. controls (419.93,975.26) and (421.95,977.28) .. (421.95,979.78) .. controls (421.95,982.27) and (419.93,984.3) .. (417.43,984.3) .. controls (414.93,984.3) and (412.91,982.27) .. (412.91,979.78) -- cycle ;
							\draw  [black, very thick]  (453,995.19) .. controls (453,992.69) and (455.02,990.67) .. (457.52,990.67) .. controls (460.02,990.67) and (462.04,992.69) .. (462.04,995.19) .. controls (462.04,997.68) and (460.02,999.71) .. (457.52,999.71) .. controls (455.02,999.71) and (453,997.68) .. (453,995.19) -- cycle ;
							\draw  [black, very thick]  (439.68,1025.32) .. controls (439.68,1022.83) and (441.71,1020.8) .. (444.2,1020.8) .. controls (446.7,1020.8) and (448.72,1022.83) .. (448.72,1025.32) .. controls (448.72,1027.82) and (446.7,1029.84) .. (444.2,1029.84) .. controls (441.71,1029.84) and (439.68,1027.82) .. (439.68,1025.32) -- cycle ;
							\draw  [black, very thick]  (518.36,1005.37) .. controls (518.36,1002.87) and (520.39,1000.85) .. (522.88,1000.85) .. controls (525.38,1000.85) and (527.4,1002.87) .. (527.4,1005.37) .. controls (527.4,1007.86) and (525.38,1009.89) .. (522.88,1009.89) .. controls (520.39,1009.89) and (518.36,1007.86) .. (518.36,1005.37) -- cycle ;
							\draw  [black, very thick]  (534.86,1025.82) .. controls (534.86,1023.33) and (536.89,1021.3) .. (539.38,1021.3) .. controls (541.88,1021.3) and (543.9,1023.33) .. (543.9,1025.82) .. controls (543.9,1028.32) and (541.88,1030.34) .. (539.38,1030.34) .. controls (536.89,1030.34) and (534.86,1028.32) .. (534.86,1025.82) -- cycle ;
							\draw  [black, very thick]  (515.18,1016.37) .. controls (515.18,1013.87) and (517.21,1011.85) .. (519.7,1011.85) .. controls (522.2,1011.85) and (524.22,1013.87) .. (524.22,1016.37) .. controls (524.22,1018.86) and (522.2,1020.89) .. (519.7,1020.89) .. controls (517.21,1020.89) and (515.18,1018.86) .. (515.18,1016.37) -- cycle ;
							\draw  [black, very thick]  (471.05,1082.09) .. controls (471.05,1079.6) and (473.07,1077.57) .. (475.57,1077.57) .. controls (478.06,1077.57) and (480.09,1079.6) .. (480.09,1082.09) .. controls (480.09,1084.59) and (478.06,1086.61) .. (475.57,1086.61) .. controls (473.07,1086.61) and (471.05,1084.59) .. (471.05,1082.09) -- cycle ;
							\draw  [black, very thick]  (381.18,1194.42) .. controls (381.18,1191.92) and (383.21,1189.9) .. (385.7,1189.9) .. controls (388.2,1189.9) and (390.22,1191.92) .. (390.22,1194.42) .. controls (390.22,1196.92) and (388.2,1198.94) .. (385.7,1198.94) .. controls (383.21,1198.94) and (381.18,1196.92) .. (381.18,1194.42) -- cycle ;
							\draw  [black, very thick]  (421.72,1208.65) .. controls (421.72,1206.16) and (423.74,1204.13) .. (426.24,1204.13) .. controls (428.73,1204.13) and (430.76,1206.16) .. (430.76,1208.65) .. controls (430.76,1211.15) and (428.73,1213.17) .. (426.24,1213.17) .. controls (423.74,1213.17) and (421.72,1211.15) .. (421.72,1208.65) -- cycle ;
							\draw  [black, very thick]  (462.62,1173.32) .. controls (462.62,1170.82) and (464.64,1168.8) .. (467.14,1168.8) .. controls (469.63,1168.8) and (471.66,1170.82) .. (471.66,1173.32) .. controls (471.66,1175.82) and (469.63,1177.84) .. (467.14,1177.84) .. controls (464.64,1177.84) and (462.62,1175.82) .. (462.62,1173.32) -- cycle ;
							\draw  [black, very thick]  (503.25,1173.75) .. controls (503.25,1171.26) and (505.27,1169.23) .. (507.77,1169.23) .. controls (510.27,1169.23) and (512.29,1171.26) .. (512.29,1173.75) .. controls (512.29,1176.25) and (510.27,1178.27) .. (507.77,1178.27) .. controls (505.27,1178.27) and (503.25,1176.25) .. (503.25,1173.75) -- cycle ;
							\draw  [fill={rgb, 255:red, 0; green, 0; blue, 0 }  ,fill opacity=1 ][black, very thick]  (490.64,1058.49) .. controls (490.64,1056) and (492.67,1053.97) .. (495.16,1053.97) .. controls (497.66,1053.97) and (499.68,1056) .. (499.68,1058.49) .. controls (499.68,1060.99) and (497.66,1063.01) .. (495.16,1063.01) .. controls (492.67,1063.01) and (490.64,1060.99) .. (490.64,1058.49) -- cycle ;
							\draw  [fill={rgb, 255:red, 0; green, 0; blue, 0 }  ,fill opacity=1 ][black, very thick]  (477.17,1210.87) .. controls (477.17,1208.38) and (479.19,1206.35) .. (481.69,1206.35) .. controls (484.18,1206.35) and (486.21,1208.38) .. (486.21,1210.87) .. controls (486.21,1213.37) and (484.18,1215.39) .. (481.69,1215.39) .. controls (479.19,1215.39) and (477.17,1213.37) .. (477.17,1210.87) -- cycle ;
							\draw (450,1250) node [anchor=north west][inner sep=0.75pt] {$K_4$};
						\end{tikzpicture}
					\end{subfigure}
					\begin{subfigure}{0.34\textwidth}
						\centering
						\tikzset{every picture/.style={line width=0.75pt}} 
						\begin{tikzpicture}[x=0.75pt,y=0.75pt,yscale=-1,xscale=1]
							\draw  [draw opacity=0][black, very thick]  (460.94,764.41) .. controls (453.54,806.66) and (416.67,838.78) .. (372.3,838.78) .. controls (322.59,838.78) and (282.3,798.48) .. (282.3,748.78) .. controls (282.3,699.07) and (322.59,658.78) .. (372.3,658.78) .. controls (416.67,658.78) and (453.54,690.89) .. (460.94,733.15) -- (372.3,748.78) -- cycle ; \draw  [line width=1.5]  (460.94,764.41) .. controls (453.54,806.66) and (416.67,838.78) .. (372.3,838.78) .. controls (322.59,838.78) and (282.3,798.48) .. (282.3,748.78) .. controls (282.3,699.07) and (322.59,658.78) .. (372.3,658.78) .. controls (416.67,658.78) and (453.54,690.89) .. (460.94,733.15) ;
							\draw  [fill={rgb, 255:red, 0; green, 0; blue, 0 }  ,fill opacity=1 ][black, very thick]  (456.42,764.41) .. controls (456.42,761.91) and (458.45,759.89) .. (460.94,759.89) .. controls (463.44,759.89) and (465.46,761.91) .. (465.46,764.41) .. controls (465.46,766.9) and (463.44,768.93) .. (460.94,768.93) .. controls (458.45,768.93) and (456.42,766.9) .. (456.42,764.41) -- cycle ;
							\draw  [fill={rgb, 255:red, 0; green, 0; blue, 0 }  ,fill opacity=1 ][black, very thick]  (456.42,733.15) .. controls (456.42,730.65) and (458.45,728.63) .. (460.94,728.63) .. controls (463.44,728.63) and (465.46,730.65) .. (465.46,733.15) .. controls (465.46,735.64) and (463.44,737.67) .. (460.94,737.67) .. controls (458.45,737.67) and (456.42,735.64) .. (456.42,733.15) -- cycle ;
							\draw [black, very thick, ->]    (380.79,659.1) -- (451.37,706.24) ;
							\draw [black, very thick, ->]    (457.37,719.67) -- (366.22,659.38) ;
							\draw [black, very thick, ->]    (443.37,693.38) -- (313.08,681.35) ;
							\draw [black, very thick, ->]    (325.11,672.34) -- (433.92,682.98) ;
							\draw [black, very thick, ->]    (422.78,674.26) -- (337.35,665.69) ;
							\draw [black, very thick, ->]    (356.76,660.33) -- (284.44,729.19) ;
							\draw [black, very thick, ->]    (287.3,777.79) -- (305.3,688.94) ;
							\draw [red, very thick, ->]    (292.48,707.07) -- (295.02,794.65) ;
							\draw [red, very thick, ->]    (415.3,828.11) -- (283.02,737.83) ;
							\draw [black, very thick, ->]    (282.76,751.82) -- (431.87,816.11) ;
							\draw [black, very thick, ->]    (404.02,833.22) -- (298.6,800.43) ;
							\draw [black, very thick, ->]    (307.88,811.53) -- (390.44,837.03) ;
							\draw [black, very thick, ->]    (376.05,838.61) -- (317.74,820.3) ;
							\draw [black, very thick, ->]    (331.8,829.43) -- (458.6,775.23) ;
							\draw [black, very thick, ->]    (454.48,785.65) -- (343.8,834.03) ;
							\draw [black, very thick, ->]    (358.63,837.96) -- (446.17,800.42) ;
							\draw (453.67,695.65) node [anchor=north west][inner sep=0.75pt]  [font=\large]  {$c_{2}$};
							\draw (294.42,666.94) node [anchor=north west][inner sep=0.75pt]  [font=\large]  {$c_{3}$};
							\draw (436.33,672.78) node [anchor=north west][inner sep=0.75pt]  [font=\large]  {$c_{4}$};
							\draw (460.56,714.09) node [anchor=north west][inner sep=0.75pt]  [font=\scriptsize]  {$-$};
							\draw (348.03,646.56) node [anchor=north west][inner sep=0.75pt]  [font=\scriptsize]  {$+$};
							\draw (277.61,698.77) node [anchor=north west][inner sep=0.75pt]  [font=\scriptsize]  {$-$};
							\draw (309.92,664.1) node [anchor=north west][inner sep=0.75pt]  [font=\scriptsize]  {$-$};
							\draw (324.35,648.97) node [anchor=north west][inner sep=0.75pt]  [font=\large]  {$c_{5}$};
							\draw (356.92,642.52) node [anchor=north west][inner sep=0.75pt]  [font=\large]  {$c_{1}$};
							\draw (445.77,686.53) node [anchor=north west][inner sep=0.75pt]  [font=\scriptsize]  {$+$};
							\draw (286.7,675.48) node [anchor=north west][inner sep=0.75pt]  [font=\large]  {$c_{7}$};
							\draw (425.16,665.73) node [anchor=north west][inner sep=0.75pt]  [font=\scriptsize]  {$+$};
							\draw (273.1,786.6) node [anchor=north west][inner sep=0.75pt]  [font=\large]  {$c_{8}$};
							\draw (376.74,645.38) node [anchor=north west][inner sep=0.75pt]  [font=\scriptsize]  {$+$};
							\draw (265.1,717.28) node [anchor=north west][inner sep=0.75pt]  [font=\large]  {$c_{6}$};
							\draw (262.64,729.28) node [anchor=north west][inner sep=0.75pt]  [font=\large]  {$c_{9}$};
							\draw (435.5,815.68) node [anchor=north west][inner sep=0.75pt]  [font=\large]  {$c_{10}$};
							\draw (274.1,796.48) node [anchor=north west][inner sep=0.75pt]  [font=\large]  {$c_{11}$};
							\draw (384.55,842.48) node [anchor=north west][inner sep=0.75pt]  [font=\large]  {$c_{12}$};
							\draw (290.7,818.08) node [anchor=north west][inner sep=0.75pt]  [font=\large]  {$c_{13}$};
							\draw (460.94,770.93) node [anchor=north west][inner sep=0.75pt]  [font=\large]  {$c_{14}$};
							\draw (325.5,836.48) node [anchor=north west][inner sep=0.75pt]  [font=\large]  {$c_{15}$};
							\draw (451.1,796.88) node [anchor=north west][inner sep=0.75pt]  [font=\large]  {$c_{16}$};
							\draw (290.89,808.66) node [anchor=north west][inner sep=0.75pt]  [font=\scriptsize]  {$+$};
							\draw (271.01,773.97) node [anchor=north west][inner sep=0.75pt]  [font=\scriptsize]  {$-$};
							\draw (418.21,829.37) node [anchor=north west][inner sep=0.75pt]  [font=\scriptsize]  {$-$};
							\draw (266.01,746.72) node [anchor=north west][inner sep=0.75pt]  [font=\scriptsize]  {$-$};
							\draw (407.81,833.73) node [anchor=north west][inner sep=0.75pt]  [font=\scriptsize]  {$-$};
							\draw (370.61,841.93) node [anchor=north west][inner sep=0.75pt]  [font=\scriptsize]  {$-$};
							\draw (318.41,830.93) node [anchor=north west][inner sep=0.75pt]  [font=\scriptsize]  {$-$};
							\draw (350.21,840.73) node [anchor=north west][inner sep=0.75pt]  [font=\scriptsize]  {$-$};
							\draw (458.69,782.46) node [anchor=north west][inner sep=0.75pt]  [font=\scriptsize]  {$+$};
							\draw (360,860) node [anchor=north west][inner sep=0.75pt]    {(a)};
						\end{tikzpicture}						
					\end{subfigure}
					\qquad 
					\begin{subfigure}{0.34\textwidth}
						\centering
						\tikzset{every picture/.style={line width=0.75pt}} 
						\begin{tikzpicture}[x=0.75pt,y=0.75pt,yscale=-1,xscale=1]
							\draw  [draw opacity=0][black, very thick]  (460.94,764.41) .. controls (453.54,806.66) and (416.67,838.78) .. (372.3,838.78) .. controls (322.59,838.78) and (282.3,798.48) .. (282.3,748.78) .. controls (282.3,699.07) and (322.59,658.78) .. (372.3,658.78) .. controls (416.67,658.78) and (453.54,690.89) .. (460.94,733.15) -- (372.3,748.78) -- cycle ; \draw  [line width=1.5]  (460.94,764.41) .. controls (453.54,806.66) and (416.67,838.78) .. (372.3,838.78) .. controls (322.59,838.78) and (282.3,798.48) .. (282.3,748.78) .. controls (282.3,699.07) and (322.59,658.78) .. (372.3,658.78) .. controls (416.67,658.78) and (453.54,690.89) .. (460.94,733.15) ;
							\draw  [fill={rgb, 255:red, 0; green, 0; blue, 0 }  ,fill opacity=1 ][black, very thick]  (456.42,764.41) .. controls (456.42,761.91) and (458.45,759.89) .. (460.94,759.89) .. controls (463.44,759.89) and (465.46,761.91) .. (465.46,764.41) .. controls (465.46,766.9) and (463.44,768.93) .. (460.94,768.93) .. controls (458.45,768.93) and (456.42,766.9) .. (456.42,764.41) -- cycle ;
							\draw  [fill={rgb, 255:red, 0; green, 0; blue, 0 }  ,fill opacity=1 ][black, very thick]  (456.42,733.15) .. controls (456.42,730.65) and (458.45,728.63) .. (460.94,728.63) .. controls (463.44,728.63) and (465.46,730.65) .. (465.46,733.15) .. controls (465.46,735.64) and (463.44,737.67) .. (460.94,737.67) .. controls (458.45,737.67) and (456.42,735.64) .. (456.42,733.15) -- cycle ;
							\draw [black, very thick, ->]    (380.79,659.1) -- (451.37,706.24) ;
							\draw [black, very thick, ->]    (457.37,719.67) -- (366.22,659.38) ;
							\draw [black, very thick, ->]    (443.37,693.38) -- (313.08,681.35) ;
							\draw [black, very thick, ->]    (325.11,672.34) -- (433.92,682.98) ;
							\draw [black, very thick, ->]    (422.78,674.26) -- (337.35,665.69) ;
							\draw [black, very thick, ->]    (356.76,660.33) -- (284.44,729.19) ;
							\draw [black, very thick, ->]    (287.3,777.79) -- (305.3,688.94) ;
							\draw [red, very thick, <-]    (292.48,707.07) -- (295.02,794.65) ;
							\draw [red, very thick, <-]    (415.3,828.11) -- (283.02,737.83) ;
							\draw [black, very thick, ->]    (282.76,751.82) -- (431.87,816.11) ; 
							\draw [black, very thick, ->]    (404.02,833.22) -- (298.6,800.43) ;
							\draw [black, very thick, ->]    (307.88,811.53) -- (390.44,837.03) ;
							\draw [black, very thick, ->]    (376.05,838.61) -- (317.74,820.3) ;
							\draw [black, very thick, ->]    (331.8,829.43) -- (458.6,775.23) ;
							\draw [black, very thick, ->]    (454.48,785.65) -- (343.8,834.03) ;
							\draw [black, very thick, ->]    (358.63,837.96) -- (446.17,800.42) ;
							\draw (456,696) node [anchor=north west][inner sep=0.75pt]  [font=\large]  {$c_{2}'$};
							\draw (292,661) node [anchor=north west][inner sep=0.75pt]  [font=\large]  {$c_{3}'$};
							\draw (438,668) node [anchor=north west][inner sep=0.75pt]  [font=\large]  {$c_{4}'$};
							\draw (460.56,714.09) node [anchor=north west][inner sep=0.75pt]  [font=\scriptsize]  {$-$};
							\draw (348.03,646.56) node [anchor=north west][inner sep=0.75pt]  [font=\scriptsize]  {$+$};
							\draw (273,688) node [anchor=north west][inner sep=0.75pt]  [font=\large]  {$c_{8}'$};
							\draw (309.92,664.1) node [anchor=north west][inner sep=0.75pt]  [font=\scriptsize]  {$-$};
							\draw (318,645) node [anchor=north west][inner sep=0.75pt]  [font=\large]  {$c_{5}'$};
							\draw (356.92,636) node [anchor=north west][inner sep=0.75pt]  [font=\large]  {$c_{1}'$};
							\draw (445.77,686.53) node [anchor=north west][inner sep=0.75pt]  [font=\scriptsize]  {$+$};
							\draw (284,675.48) node [anchor=north west][inner sep=0.75pt]  [font=\large]  {$c_{7}'$};
							\draw (425.16,665.73) node [anchor=north west][inner sep=0.75pt]  [font=\scriptsize]  {$+$};
							\draw (282,792) node [anchor=north west][inner sep=0.75pt]  [font=\scriptsize]  {$+$};
							\draw (376.74,645.38) node [anchor=north west][inner sep=0.75pt]  [font=\scriptsize]  {$+$};
							\draw (266,716) node [anchor=north west][inner sep=0.75pt]  [font=\large]  {$c_{6}'$};
							\draw (267,732) node [anchor=north west][inner sep=0.75pt]  [font=\scriptsize]  {$+$};
							\draw (434,810) node [anchor=north west][inner sep=0.75pt]  [font=\large]  {$c_{10}'$};
							\draw (270,792) node [anchor=north west][inner sep=0.75pt]  [font=\large]  {$c_{11}'$};
							\draw (390,836) node [anchor=north west][inner sep=0.75pt]  [font=\large]  {$c_{12}'$};
							\draw (292,815) node [anchor=north west][inner sep=0.75pt]  [font=\large]  {$c_{13}'$};
							\draw (461,763) node [anchor=north west][inner sep=0.75pt]  [font=\large]  {$c_{14}'$};
							\draw (324,835) node [anchor=north west][inner sep=0.75pt]  [font=\large]  {$c_{15}'$};
							\draw (450,790) node [anchor=north west][inner sep=0.75pt]  [font=\large]  {$c_{16}'$};
							\draw (290.89,808.66) node [anchor=north west][inner sep=0.75pt]  [font=\scriptsize]  {$+$};
							\draw (271.01,773.97) node [anchor=north west][inner sep=0.75pt]  [font=\scriptsize]  {$-$};
							\draw (418.21,824) node [anchor=north west][inner sep=0.75pt]  [font=\large]  {$c_{9}'$};
							\draw (266.01,746.72) node [anchor=north west][inner sep=0.75pt]  [font=\scriptsize]  {$-$};
							\draw (407.81,833.73) node [anchor=north west][inner sep=0.75pt]  [font=\scriptsize]  {$-$};
							\draw (370.61,841.93) node [anchor=north west][inner sep=0.75pt]  [font=\scriptsize]  {$-$};
							\draw (318.41,830.93) node [anchor=north west][inner sep=0.75pt]  [font=\scriptsize]  {$-$};
							\draw (350.21,840.73) node [anchor=north west][inner sep=0.75pt]  [font=\scriptsize]  {$-$};
							\draw (458.69,782.46) node [anchor=north west][inner sep=0.75pt]  [font=\scriptsize]  {$+$};
							\draw (367,860) node [anchor=north west][inner sep=0.75pt]    {(b)};
						\end{tikzpicture}			
					\end{subfigure}		
					\caption{Diagrams and Gauss diagrams of planar virtual knotoids $K_3$ and $K_4$.}
					\label{fig23}
				\end{center}
			\end{figure} 	
			From the Gauss diagrams of these two planar virtual knotoids, we can get the index of each chord of the Gauss diagram of $K_3$,
			$$
			\Ind(c_1)=-3,\ \Ind(c_2)=3,\ \Ind(c_3)=-3,\ \Ind(c_4)=3,\ \Ind(c_5)=-3, \ \Ind(c_6)=3, 
			$$
			$$
			\Ind(c_7)=-2, \ \Ind(c_8)=2,\ \Ind(c_9)=-4,\ \Ind(c_{10})=4, \, \Ind(c_{11})=-3,\ \Ind(c_{12})=3,  
			$$
			$$
			\Ind(c_{13})=-3, \, \Ind(c_{14})=3,\ \Ind(c_{15})=-3, \Ind(c_{16})=3;
			$$
			and the indices of all chords in the Gauss diagrams of $K_4$,
			$$
			\Ind(c_1')=-3,\ \Ind(c_2')=3,\ \Ind(c_3')=-3,\ \Ind(c_4')=3,\ \Ind(c_5')=-3, \ \Ind(c_6')=3,
			$$
			$$ 
			\Ind(c_7')=-2,\ \Ind(c_8')=-2,\ \Ind(c_9')=4,\ \Ind(c_{10}')=4, \ \Ind(c_{11}')=-3,\ \Ind(c_{12}')=3, 
			$$ 
			$$				
			\Ind(c_{13}')=-3,\ \Ind(c_{14}')=3,\ \Ind(c_{15}')=-3, \ \Ind(c_{16}')=3.
			$$ 
			From this we observe that the maximum absolute value of the chord indices for this pair of knotoids is $4$. Thus, for any integer $r>4$, the $r$-covering virtualizes all crossings in the diagrams of $K_3$ and $K_4$, so both $K_3^{(r)}$ and $K_4^{(r)}$ are trivial knotoids, which yields $d_G(K_3^{(r)}, K_4^{(r)}) = 0$. We now consider the cases $r=1,2,3,4$. 
			
			When $r = 1$, the $1$-covering virtualizes no crossings in the diagrams of $K_3$ and $K_4$, so \linebreak $d_G(K_3^{(1)}, K_4^{(1)}) = d_G(K_3, K_4)$. 
			
			When $r = 2$, for the $2$-covering, it is noted that the diagram of $K_3^{(2)}$ is also shown in Figure~\ref{fig16}, and $K_4^{(2)}$ is the trivial knotoid. According to Examples~\ref{ex:3}, we have $d_G(K_3^{(2)}, K_4^{(2)}) = 2$. 
			
			When $r=3$, the $3$-covering virtualizes all the distinct crossings between the diagrams of $K_3$ and $K_4$, so $K_3^{(3)}=K_4^{(3)}$, that is, $d_G(K_3^{(3)}, K_4^{(3)}) = 0$.
			
			When $r = 4$, we consider the $4$-covering and can find that the diagram of $K_3^{(4)}$ is given in Figure~\ref{fig20}, and $K_4^{(4)}$ is the trivial knotoid. Although the diagrams of $K_3^{(4)}$ and $K_4^{(4)}$ also differ by a crossing change and an $\Omega_2$-move, a direct computation of their normalized bracket polynomials yields
			$$
			\langle K_3^{(4)}\rangle_\circ = A^{4} + A^{6} - A^{10}, \qquad \langle K_4^{(4)}\rangle_\circ = 1, 
			$$
			which implies that $K_3^{(4)}$ and $K_4^{(4)}$ are homotopic yet non-equivalent. Hence, $d_G(K_3^{(4)}, K_4^{(4)})\geq 1$. Therefore, we get $d_G(K_3^{(4)}, K_4^{(4)}) = 1$.				
			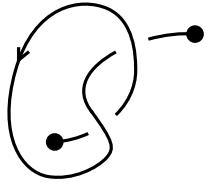
\begin{figure}[htbp]
				\begin{center}
					\tikzset{every picture/.style={line width=0.7pt}} 
					\begin{tikzpicture}[x=0.75pt,y=0.75pt,yscale=-0.8,xscale=0.8]
						\draw [black, very thick]    (584.05,784.11) .. controls (598.36,803.66) and (599.13,807.69) .. (593.75,813.69) .. controls (588.38,819.69) and (573,828.1) .. (557,825.4) .. controls (541,822.69) and (525,800.4) .. (536.67,759.06) .. controls (548.33,717.73) and (577.82,712.71) .. (594.52,720.11) .. controls (611.22,727.51) and (612.02,749.74) .. (612.06,756.59) ;
						\draw [black, very thick]    (612,755.73) .. controls (612.83,764.56) and (610.07,776.71) .. (599.62,786.8) ;
						\draw [black, very thick]    (585,785.4) .. controls (572.64,770.21) and (582.83,756.42) .. (599.7,747.26) ;
						\draw  [black, very thick]  (545.23,747.07) -- (538.65,752.41) -- (538.7,744.27) ;
						\draw  [fill={rgb, 255:red, 0; green, 0; blue, 0 }  ,fill opacity=1 ][black, very thick]  (644.81,736.06) .. controls (644.81,733.57) and (646.84,731.54) .. (649.33,731.54) .. controls (651.83,731.54) and (653.85,733.57) .. (653.85,736.06) .. controls (653.85,738.56) and (651.83,740.58) .. (649.33,740.58) .. controls (646.84,740.58) and (644.81,738.56) .. (644.81,736.06) -- cycle ;
						\draw  [fill={rgb, 255:red, 0; green, 0; blue, 0 }  ,fill opacity=1 ][black, very thick]  (556.81,803.73) .. controls (556.81,801.23) and (558.84,799.21) .. (561.33,799.21) .. controls (563.83,799.21) and (565.85,801.23) .. (565.85,803.73) .. controls (565.85,806.23) and (563.83,808.25) .. (561.33,808.25) .. controls (558.84,808.25) and (556.81,806.23) .. (556.81,803.73) -- cycle ;
						\draw [black, very thick]    (582.36,798.45) .. controls (576.07,801.29) and (569,803.2) .. (561.33,803.73) ;
						\draw [black, very thick]    (620.04,739.25) .. controls (630.31,736.45) and (640.83,735.23) .. (649.33,736.06) ;
					\end{tikzpicture}		
					\caption{The diagram of $K_3^{(4)}$.}
					\label{fig20}
				\end{center}
			\end{figure}
			
			To sum up,  the $r$-coverings yield $d_G(K_3, K_4) \geq 2$, and we can also observe that the diagrams of $K_3$ and $K_4$ differ by two crossing changes, so $d_G(K_3, K_4) \le 2$.  Therefore, $d_G(K_3, K_4) = d_G(K_3^{(2)}, K_4^{(2)}) = 2$.
		}
	\end{example}
	
	Nevertheless, the precise Gordian distance cannot always be computed via the $r$-coverings, as illustrated by the following example.
	
	\begin{example} \label{ex:8} {\rm
			Let $K_5$ and $K_{6}$ be two planar virtual knotoids whose diagrams and Gauss diagrams are shown in Figure~\ref{fig22}~(a) and~(b). 
			\begin{figure}[htbp]
				\begin{center}
					\begin{subfigure}{0.34\textwidth}
						\centering
						\tikzset{every picture/.style={line width=0.75pt}} 
						\begin{tikzpicture}[x=0.75pt,y=0.75pt,yscale=-1,xscale=1]
							\draw [black, very thick]    (120.95,72.98) .. controls (129.4,79.71) and (132.64,90.46) .. (132.41,98.1) ; 
							\draw [black, very thick]    (101.98,55.73) .. controls (87.28,42.37) and (69.71,56.16) .. (69.72,77.19) .. controls (69.73,98.21) and (83.91,118.13) .. (105.26,128.9) ;
							\draw [black, very thick]    (79.55,87.59) .. controls (118.72,51.48) and (141.69,52.15) .. (155.88,59.88) .. controls (170.07,67.61) and (172,84.19) .. (165.32,104.85) .. controls (158.64,125.51) and (133.7,140.69) .. (122.1,134.75) ;
							\draw [black, very thick]    (64.16,99.44) .. controls (48.37,115.65) and (68.87,159.13) .. (95.38,150.05) .. controls (121.89,140.96) and (109.4,95.74) .. (149.95,116.03) ;
							\draw [black, very thick]    (166.05,123.51) .. controls (196.32,139.29) and (193.2,183.02) .. (153,189.36) ;
							\draw [black, very thick]    (136.76,191.72) .. controls (121.34,193.31) and (102.11,192.52) .. (89.02,189.83) ;
							\draw [black, very thick]    (69.46,185.34) .. controls (56.35,180.52) and (48.78,196.19) .. (49.01,207.41) .. controls (49.23,218.63) and (52.67,237.97) .. (92.45,246.22) .. controls (132.24,254.47) and (156.22,246.75) .. (159.85,229.9) .. controls (163.48,213.06) and (136.52,195.86) .. (145.74,180.58) .. controls (154.96,165.3) and (197.15,182.12) .. (208.86,175) .. controls (220.57,167.88) and (204.81,154.85) .. (194.12,151.19) ;
							\draw [black, very thick]    (176.77,145.32) .. controls (166.48,141.76) and (153.72,135.96) .. (137.17,125.99) .. controls (120.62,116.02) and (101.94,99.82) .. (92.96,110.08) ;
							\draw [black, very thick]    (130.05,200.83) .. controls (132.95,210.97) and (132.51,217.92) .. (127.16,226.64) ;
							\draw [black, very thick]    (124.56,183.82) .. controls (122.38,177.35) and (119.39,166.45) .. (109.37,169.83) .. controls (99.34,173.21) and (112.72,206.51) .. (98.02,209.89) .. controls (83.31,213.28) and (78.63,174.46) .. (66.6,166.27) .. controls (54.56,158.09) and (42.98,166.46) .. (50.78,186.58) .. controls (58.58,206.7) and (89.33,217.55) .. (86.21,235.36) ;
							\draw [black, very thick]    (80.47,120.04) .. controls (67.13,133.04) and (28.48,162.36) .. (28.71,207.23) .. controls (28.94,252.1) and (74.61,276.31) .. (80.85,254.77) ;
							\draw  [fill={rgb, 255:red, 0; green, 0; blue, 0 }  ,fill opacity=1 ][black, very thick]  (127.89,98.1) .. controls (127.89,95.6) and (129.91,93.58) .. (132.41,93.58) .. controls (134.9,93.58) and (136.93,95.6) .. (136.93,98.1) .. controls (136.93,100.59) and (134.9,102.62) .. (132.41,102.62) .. controls (129.91,102.62) and (127.89,100.59) .. (127.89,98.1) -- cycle ;
							\draw  [black, very thick]  (140.11,130.31) .. controls (140.11,127.81) and (142.14,125.79) .. (144.63,125.79) .. controls (147.13,125.79) and (149.15,127.81) .. (149.15,130.31) .. controls (149.15,132.8) and (147.13,134.83) .. (144.63,134.83) .. controls (142.14,134.83) and (140.11,132.8) .. (140.11,130.31) -- cycle ;
							\draw  [fill={rgb, 255:red, 0; green, 0; blue, 0 }  ,fill opacity=1 ][black, very thick]  (122.64,226.64) .. controls (122.64,224.15) and (124.66,222.12) .. (127.16,222.12) .. controls (129.66,222.12) and (131.68,224.15) .. (131.68,226.64) .. controls (131.68,229.14) and (129.66,231.16) .. (127.16,231.16) .. controls (124.66,231.16) and (122.64,229.14) .. (122.64,226.64) -- cycle ;
							\draw  [black, very thick]  (174.68,174.88) .. controls (174.68,172.38) and (176.71,170.36) .. (179.2,170.36) .. controls (181.7,170.36) and (183.72,172.38) .. (183.72,174.88) .. controls (183.72,177.37) and (181.7,179.4) .. (179.2,179.4) .. controls (176.71,179.4) and (174.68,177.37) .. (174.68,174.88) -- cycle ;
							\draw  [black, very thick]  (101.16,191.89) .. controls (101.16,189.4) and (103.18,187.37) .. (105.68,187.37) .. controls (108.18,187.37) and (110.2,189.4) .. (110.2,191.89) .. controls (110.2,194.39) and (108.18,196.41) .. (105.68,196.41) .. controls (103.18,196.41) and (101.16,194.39) .. (101.16,191.89) -- cycle ;
							\draw  [black, very thick]  (48.87,191.89) .. controls (48.87,189.4) and (50.9,187.37) .. (53.39,187.37) .. controls (55.89,187.37) and (57.91,189.4) .. (57.91,191.89) .. controls (57.91,194.39) and (55.89,196.41) .. (53.39,196.41) .. controls (50.9,196.41) and (48.87,194.39) .. (48.87,191.89) -- cycle ;
							\draw  [black, very thick]  (60.08,135.04) .. controls (60.08,132.54) and (62.1,130.52) .. (64.6,130.52) .. controls (67.1,130.52) and (69.12,132.54) .. (69.12,135.04) .. controls (69.12,137.53) and (67.1,139.56) .. (64.6,139.56) .. controls (62.1,139.56) and (60.08,137.53) .. (60.08,135.04) -- cycle ;
							\draw  [black, very thick]  (117.1,115.89) .. controls (117.1,113.4) and (119.12,111.37) .. (121.62,111.37) .. controls (124.11,111.37) and (126.14,113.4) .. (126.14,115.89) .. controls (126.14,118.39) and (124.11,120.41) .. (121.62,120.41) .. controls (119.12,120.41) and (117.1,118.39) .. (117.1,115.89) -- cycle ;
							\draw  [black, very thick]  (117.24,245.92) -- (125.01,249.3) -- (117.76,253.01) ;
							\draw (116,265) node [anchor=north west][inner sep=0.75pt] {$K_5$};
						\end{tikzpicture}						
					\end{subfigure}
					\qquad 
					\begin{subfigure}{0.34\textwidth}
						\centering
						\tikzset{every picture/.style={line width=0.75pt}} 
						\begin{tikzpicture}[x=0.75pt,y=0.75pt,yscale=-1,xscale=1] 
							\draw [color={rgb, 255:red, 0; green, 0; blue, 0 }  ,draw opacity=1 ][black, very thick]    (308.54,99.37) .. controls (346.6,60.76) and (384.06,46.75) .. (401.47,58.29) .. controls (418.89,69.84) and (418.25,83.06) .. (410.89,103.44) .. controls (403.54,123.82) and (382.04,138.9) .. (366.84,133.18) ;
							\draw [color={rgb, 255:red, 0; green, 0; blue, 0 }  ,draw opacity=1 ][black, very thick]    (322.36,100.96) .. controls (326.8,108.81) and (337.93,123.34) .. (369.21,134.01) ;
							\draw [color={rgb, 255:red, 0; green, 0; blue, 0 }  ,draw opacity=1 ][black, very thick]    (366.59,71.44) .. controls (375.03,78.2) and (378.27,88.99) .. (378.04,96.66) ;
							\draw [color={rgb, 255:red, 0; green, 0; blue, 0 }  ,draw opacity=1 ][black, very thick]    (347.66,54.12) .. controls (332.98,40.72) and (311.76,54.32) .. (316.13,84.09) ;
							\draw [color={rgb, 255:red, 0; green, 0; blue, 0 }  ,draw opacity=1 ][black, very thick]    (309.9,98) .. controls (290.55,115.54) and (325.47,177.39) .. (354.83,136.81) ;
							\draw [color={rgb, 255:red, 0; green, 0; blue, 0 }  ,draw opacity=1 ][black, very thick]    (411.63,122.17) .. controls (441.84,138.02) and (438.73,181.92) .. (398.6,188.28) ;
							\draw [color={rgb, 255:red, 0; green, 0; blue, 0 }  ,draw opacity=1 ][black, very thick]    (382.38,190.65) .. controls (366.99,192.25) and (347.79,191.45) .. (334.72,188.75) ;
							\draw [color={rgb, 255:red, 0; green, 0; blue, 0 }  ,draw opacity=1 ][black, very thick]    (315.19,184.25) .. controls (302.1,179.41) and (294.54,195.14) .. (294.77,206.4) .. controls (294.99,217.66) and (298.42,237.08) .. (338.15,245.37) .. controls (377.87,253.65) and (401.81,245.89) .. (405.43,228.98) .. controls (409.06,212.07) and (382.14,194.81) .. (391.34,179.47) .. controls (400.55,164.13) and (442.68,181.02) .. (454.37,173.87) .. controls (466.06,166.71) and (450.32,153.64) .. (439.65,149.96) ;
							\draw [color={rgb, 255:red, 0; green, 0; blue, 0 }  ,draw opacity=1 ][black, very thick]    (422.33,144.07) .. controls (412.05,140.49) and (399.31,134.67) .. (382.79,124.66) .. controls (366.26,114.66) and (347.62,98.39) .. (338.65,108.69) ;
							\draw [color={rgb, 255:red, 0; green, 0; blue, 0 }  ,draw opacity=1 ][black, very thick]    (375.68,199.79) .. controls (378.58,209.98) and (378.13,216.95) .. (372.8,225.71) ;
							\draw [color={rgb, 255:red, 0; green, 0; blue, 0 }  ,draw opacity=1 ][black, very thick]    (370.2,182.72) .. controls (368.02,176.23) and (365.04,165.28) .. (355.03,168.68) .. controls (345.02,172.07) and (358.38,205.5) .. (343.7,208.9) .. controls (329.02,212.3) and (324.34,173.33) .. (312.33,165.11) .. controls (300.32,156.88) and (288.75,165.29) .. (296.54,185.49) .. controls (304.33,205.69) and (335.03,216.59) .. (331.92,234.46) ;
							\draw [color={rgb, 255:red, 0; green, 0; blue, 0 }  ,draw opacity=1 ][black, very thick]    (326.18,118.69) .. controls (312.86,131.74) and (274.27,161.18) .. (274.5,206.23) .. controls (274.73,251.27) and (320.33,275.58) .. (326.56,253.95) ;
							\draw [color={rgb, 255:red, 0; green, 0; blue, 0 }  ,draw opacity=1 ][black, very thick]    (361.72,123.33) .. controls (366.32,113.75) and (372.4,102.67) .. (395.55,114.66) ;
							\draw  [black, very thick]  (365.68,245.45) -- (373.61,248.45) -- (366.55,252.5) ;
							\draw  [fill={rgb, 255:red, 0; green, 0; blue, 0 }  ,fill opacity=1 ][black, very thick]  (373.52,96.66) .. controls (373.52,94.16) and (375.54,92.14) .. (378.04,92.14) .. controls (380.53,92.14) and (382.56,94.16) .. (382.56,96.66) .. controls (382.56,99.16) and (380.53,101.18) .. (378.04,101.18) .. controls (375.54,101.18) and (373.52,99.16) .. (373.52,96.66) -- cycle ;
							\draw  [fill={rgb, 255:red, 0; green, 0; blue, 0 }  ,fill opacity=1 ][black, very thick]  (368.28,225.71) .. controls (368.28,223.22) and (370.3,221.19) .. (372.8,221.19) .. controls (375.29,221.19) and (377.32,223.22) .. (377.32,225.71) .. controls (377.32,228.21) and (375.29,230.23) .. (372.8,230.23) .. controls (370.3,230.23) and (368.28,228.21) .. (368.28,225.71) -- cycle ;
							\draw  [black, very thick]  (305.91,133.51) .. controls (305.91,131.01) and (307.94,128.99) .. (310.43,128.99) .. controls (312.93,128.99) and (314.95,131.01) .. (314.95,133.51) .. controls (314.95,136) and (312.93,138.03) .. (310.43,138.03) .. controls (307.94,138.03) and (305.91,136) .. (305.91,133.51) -- cycle ;
							\draw  [black, very thick]  (362.71,114.51) .. controls (362.71,112.01) and (364.74,109.99) .. (367.23,109.99) .. controls (369.73,109.99) and (371.75,112.01) .. (371.75,114.51) .. controls (371.75,117) and (369.73,119.03) .. (367.23,119.03) .. controls (364.74,119.03) and (362.71,117) .. (362.71,114.51) -- cycle ;
							\draw  [black, very thick]  (385.71,128.91) .. controls (385.71,126.41) and (387.74,124.39) .. (390.23,124.39) .. controls (392.73,124.39) and (394.75,126.41) .. (394.75,128.91) .. controls (394.75,131.4) and (392.73,133.43) .. (390.23,133.43) .. controls (387.74,133.43) and (385.71,131.4) .. (385.71,128.91) -- cycle ;
							\draw  [black, very thick]  (294.91,190.51) .. controls (294.91,188.01) and (296.94,185.99) .. (299.43,185.99) .. controls (301.93,185.99) and (303.95,188.01) .. (303.95,190.51) .. controls (303.95,193) and (301.93,195.03) .. (299.43,195.03) .. controls (296.94,195.03) and (294.91,193) .. (294.91,190.51) -- cycle ;
							\draw  [black, very thick]  (346.71,190.71) .. controls (346.71,188.21) and (348.74,186.19) .. (351.23,186.19) .. controls (353.73,186.19) and (355.75,188.21) .. (355.75,190.71) .. controls (355.75,193.2) and (353.73,195.23) .. (351.23,195.23) .. controls (348.74,195.23) and (346.71,193.2) .. (346.71,190.71) -- cycle ;
							\draw  [black, very thick]  (420.31,173.71) .. controls (420.31,171.21) and (422.34,169.19) .. (424.83,169.19) .. controls (427.33,169.19) and (429.35,171.21) .. (429.35,173.71) .. controls (429.35,176.2) and (427.33,178.23) .. (424.83,178.23) .. controls (422.34,178.23) and (420.31,176.2) .. (420.31,173.71) -- cycle ;
							\draw (366,265) node [anchor=north west][inner sep=0.75pt] {$K_{6}$};
						\end{tikzpicture}						
					\end{subfigure}
					\begin{subfigure}{0.34\textwidth}
						\centering
						\tikzset{every picture/.style={line width=0.75pt}} 
						\begin{tikzpicture}[x=0.75pt,y=0.75pt,yscale=-1,xscale=1]
							\draw  [draw opacity=0][black, very thick]  (237.94,444.48) .. controls (230.54,486.74) and (193.67,518.85) .. (149.3,518.85) .. controls (99.59,518.85) and (59.3,478.56) .. (59.3,428.85) .. controls (59.3,379.15) and (99.59,338.85) .. (149.3,338.85) .. controls (193.67,338.85) and (230.54,370.97) .. (237.94,413.22) -- (149.3,428.85) -- cycle ; \draw  [line width=1.5]  (237.94,444.48) .. controls (230.54,486.74) and (193.67,518.85) .. (149.3,518.85) .. controls (99.59,518.85) and (59.3,478.56) .. (59.3,428.85) .. controls (59.3,379.15) and (99.59,338.85) .. (149.3,338.85) .. controls (193.67,338.85) and (230.54,370.97) .. (237.94,413.22) ;  
							\draw  [fill={rgb, 255:red, 0; green, 0; blue, 0 }  ,fill opacity=1 ][black, very thick]  (233.42,444.48) .. controls (233.42,441.99) and (235.45,439.96) .. (237.94,439.96) .. controls (240.44,439.96) and (242.46,441.99) .. (242.46,444.48) .. controls (242.46,446.98) and (240.44,449) .. (237.94,449) .. controls (235.45,449) and (233.42,446.98) .. (233.42,444.48) -- cycle ;
							\draw  [fill={rgb, 255:red, 0; green, 0; blue, 0 }  ,fill opacity=1 ][black, very thick]  (233.42,413.22) .. controls (233.42,410.72) and (235.45,408.7) .. (237.94,408.7) .. controls (240.44,408.7) and (242.46,410.72) .. (242.46,413.22) .. controls (242.46,415.72) and (240.44,417.74) .. (237.94,417.74) .. controls (235.45,417.74) and (233.42,415.72) .. (233.42,413.22) -- cycle ;
							\draw [black, very thick, ->]    (93.25,358.9) -- (230.04,389.83) ;
							\draw [red, very thick, ->]    (207.67,360.48) -- (77.25,374.9) ;
							\draw [black, very thick, ->]    (72.37,473.91) -- (231.17,466.6) ;
							\draw [black, very thick, ->]    (218.25,486.71) -- (85.04,491.57) ;
							\draw [black, very thick, ->]    (102.37,505.91) -- (198.7,503.57) ;
							\draw [black, very thick, ->]    (136.5,517.71) -- (64.25,458.21) ;
							\draw [black, very thick, ->]    (60.14,439.94) -- (154.75,518.51) ;
							\draw [black, very thick, ->]    (177.5,343.15) -- (174.25,515.21) ;
							\draw [red, very thick, ->]    (65.25,396.65) -- (154.75,338.9) ;
							\draw [black, very thick, ->]    (124.39,342.63) -- (60,418.9) ;
							\draw (53.17,365) node [anchor=north west][inner sep=0.75pt]  [font=\large]  {$c_{2}$};
							\draw (165.17,522.02) node [anchor=north west][inner sep=0.75pt]  [font=\large]  {$c_{3}$};
							\draw (152.58,322.6) node [anchor=north west][inner sep=0.75pt]  [font=\large]  {$c_{4}$};
							\draw (214.56,351.91) node [anchor=north west][inner sep=0.75pt]  [font=\scriptsize]  {$-$};
							\draw (116.36,328.59) node [anchor=north west][inner sep=0.75pt]  [font=\scriptsize]  {$-$};
							\draw (172.92,329.18) node [anchor=north west][inner sep=0.75pt]  [font=\scriptsize]  {$-$};
							\draw (39.1,411.55) node [anchor=north west][inner sep=0.75pt]  [font=\large]  {$c_{5}$};
							\draw (236.92,385) node [anchor=north west][inner sep=0.75pt]  [font=\large]  {$c_{1}$};
							\draw (41.29,449.56) node [anchor=north west][inner sep=0.75pt]  [font=\large]  {$c_{7}$};
							\draw (238.35,461.67) node [anchor=north west][inner sep=0.75pt]  [font=\large]  {$c_{8}$};
							\draw (144.1,525.36) node [anchor=north west][inner sep=0.75pt]  [font=\large]  {$c_{6}$};
							\draw (59.64,485.61) node [anchor=north west][inner sep=0.75pt]  [font=\large]  {$c_{9}$};
							\draw (208.5,499.51) node [anchor=north west][inner sep=0.75pt]  [font=\large]  {$c_{10}$};
							\draw (81.89,502.23) node [anchor=north west][inner sep=0.75pt]  [font=\scriptsize]  {$+$};
							\draw (77.76,348.29) node [anchor=north west][inner sep=0.75pt]  [font=\scriptsize]  {$-$};
							\draw (225.21,482.19) node [anchor=north west][inner sep=0.75pt]  [font=\scriptsize]  {$-$};
							\draw (48.16,390.54) node [anchor=north west][inner sep=0.75pt]  [font=\scriptsize]  {$-$};
							\draw (42.16,434.76) node [anchor=north west][inner sep=0.75pt]  [font=\scriptsize]  {$-$};
							\draw (128.46,524.81) node [anchor=north west][inner sep=0.75pt]  [font=\scriptsize]  {$+$};
							\draw (52.44,469.03) node [anchor=north west][inner sep=0.75pt]  [font=\scriptsize]  {$+$};
							\draw (133,540) node [anchor=north west][inner sep=0.75pt]    {(a)};
						\end{tikzpicture}						
					\end{subfigure}
					\qquad 
					\begin{subfigure}{0.34\textwidth}
						\centering
						\tikzset{every picture/.style={line width=0.75pt}} 
						\begin{tikzpicture}[x=0.75pt,y=0.75pt,yscale=-1,xscale=1]
							\draw  [draw opacity=0][black, very thick]  (237.94,444.48) .. controls (230.54,486.74) and (193.67,518.85) .. (149.3,518.85) .. controls (99.59,518.85) and (59.3,478.56) .. (59.3,428.85) .. controls (59.3,379.15) and (99.59,338.85) .. (149.3,338.85) .. controls (193.67,338.85) and (230.54,370.97) .. (237.94,413.22) -- (149.3,428.85) -- cycle ; \draw  [black, very thick]  (237.94,444.48) .. controls (230.54,486.74) and (193.67,518.85) .. (149.3,518.85) .. controls (99.59,518.85) and (59.3,478.56) .. (59.3,428.85) .. controls (59.3,379.15) and (99.59,338.85) .. (149.3,338.85) .. controls (193.67,338.85) and (230.54,370.97) .. (237.94,413.22) ;  
							\draw  [fill={rgb, 255:red, 0; green, 0; blue, 0 }  ,fill opacity=1 ][black, very thick]  (233.42,444.48) .. controls (233.42,441.99) and (235.45,439.96) .. (237.94,439.96) .. controls (240.44,439.96) and (242.46,441.99) .. (242.46,444.48) .. controls (242.46,446.98) and (240.44,449) .. (237.94,449) .. controls (235.45,449) and (233.42,446.98) .. (233.42,444.48) -- cycle ;
							\draw  [fill={rgb, 255:red, 0; green, 0; blue, 0 }  ,fill opacity=1 ][black, very thick]  (233.42,413.22) .. controls (233.42,410.72) and (235.45,408.7) .. (237.94,408.7) .. controls (240.44,408.7) and (242.46,410.72) .. (242.46,413.22) .. controls (242.46,415.72) and (240.44,417.74) .. (237.94,417.74) .. controls (235.45,417.74) and (233.42,415.72) .. (233.42,413.22) -- cycle ;
							\draw [black, very thick, ->]    (93.25,358.9) -- (230.04,389.83) ;
							\draw [red, very thick, <-]    (207.67,360.48) -- (77.25,374.9) ;
							\draw [black, very thick, ->]    (72.37,473.91) -- (231.17,466.6) ;
							\draw [black, very thick, ->]    (218.25,486.71) -- (85.04,491.57) ;
							\draw [black, very thick, ->]    (102.37,505.91) -- (198.7,503.57) ;
							\draw [black, very thick, ->]    (136.5,517.71) -- (64.25,458.21) ;
							\draw [black, very thick, ->]    (60.14,439.94) -- (154.75,518.51) ;
							\draw [black, very thick, ->]    (177.5,343.15) -- (174.25,515.21) ;
							\draw [red, very thick, <-]    (65.25,396.65) -- (154.75,338.9) ;
							\draw [black, very thick, ->]    (124.39,342.63) -- (60,418.9) ;
							\draw (62,368) node [anchor=north west][inner sep=0.75pt]  [font=\scriptsize]  {$+$};
							\draw (166,518) node [anchor=north west][inner sep=0.75pt]  [font=\large]  {$c_{3}'$};
							\draw (150,326) node [anchor=north west][inner sep=0.75pt]  [font=\scriptsize]  {$+$};
							\draw (214,348) node [anchor=north west][inner sep=0.75pt]  [font=\large]  {$c_{2}'$};
							\draw (116.36,328.59) node [anchor=north west][inner sep=0.75pt]  [font=\scriptsize]  {$-$};
							\draw (172.92,329.18) node [anchor=north west][inner sep=0.75pt]  [font=\scriptsize]  {$-$};
							\draw (39.1,411.55) node [anchor=north west][inner sep=0.75pt]  [font=\large]  {$c_{5}'$};
							\draw (234,380.59) node [anchor=north west][inner sep=0.75pt]  [font=\large]  {$c_{1}'$};
							\draw (42,446) node [anchor=north west][inner sep=0.75pt]  [font=\large]  {$c_{7}'$};
							\draw (234,456) node [anchor=north west][inner sep=0.75pt]  [font=\large]  {$c_{8}'$};
							\draw (148,522) node [anchor=north west][inner sep=0.75pt]  [font=\large]  {$c_{6}'$};
							\draw (61,484) node [anchor=north west][inner sep=0.75pt]  [font=\large]  {$c_{9}'$};
							\draw (206,493) node [anchor=north west][inner sep=0.75pt]  [font=\large]  {$c_{10}'$};
							\draw (81.89,502.23) node [anchor=north west][inner sep=0.75pt]  [font=\scriptsize]  {$+$};
							\draw (77.76,348.29) node [anchor=north west][inner sep=0.75pt]  [font=\scriptsize]  {$-$};
							\draw (225.21,482.19) node [anchor=north west][inner sep=0.75pt]  [font=\scriptsize]  {$-$};
							\draw (42,388) node [anchor=north west][inner sep=0.75pt]  [font=\large]  {$c_{4}'$};
							\draw (42.16,434.76) node [anchor=north west][inner sep=0.75pt]  [font=\scriptsize]  {$-$};
							\draw (128.46,524.81) node [anchor=north west][inner sep=0.75pt]  [font=\scriptsize]  {$+$};
							\draw (52.44,469.03) node [anchor=north west][inner sep=0.75pt]  [font=\scriptsize]  {$+$};
							\draw (138,545) node [anchor=north west][inner sep=0.75pt]    {(b)};
						\end{tikzpicture}						
					\end{subfigure}		
					\caption{Diagrams and Gauss diagrams of planar virtual knotoids $K_5$ and $K_{6}$.}
					\label{fig22}
				\end{center}
			\end{figure}
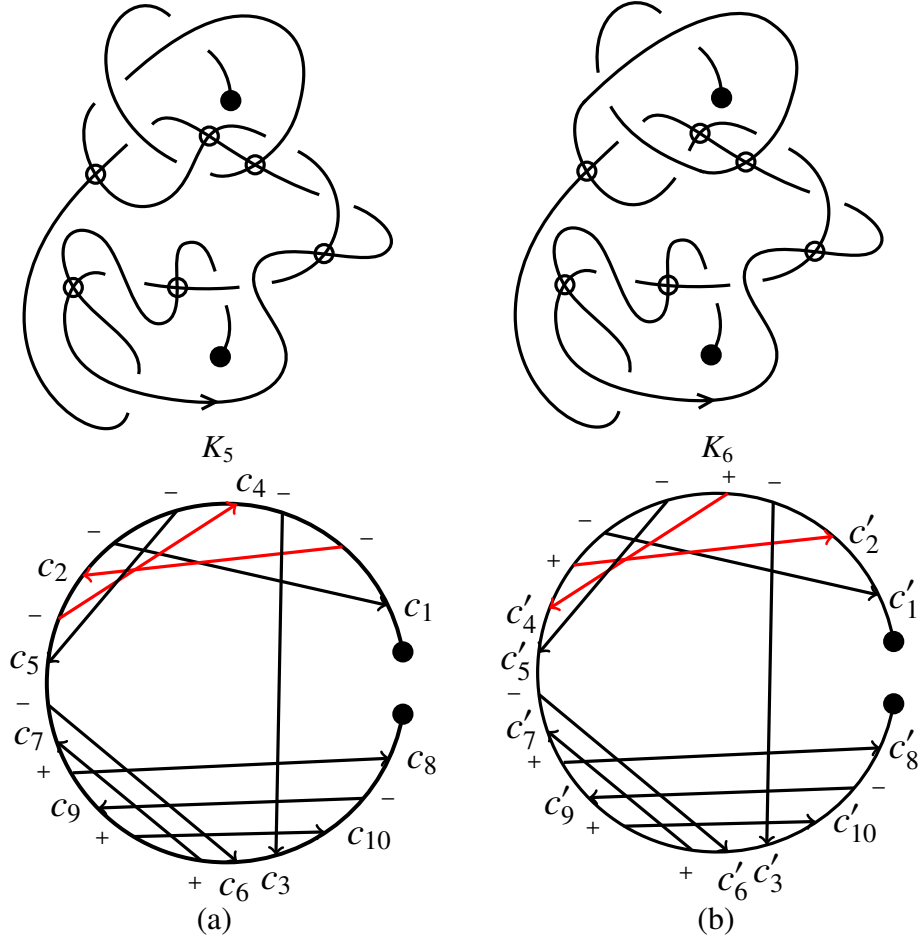
			
			Through the Gauss diagrams of $K_5$ and $K_{6}$, we can obtain the indices of all chords in the Gauss diagram of $K_5$,
			$$
			\Ind(c_1)=-2,\ \Ind(c_2)=2,\ \Ind(c_3)=-3,\ \Ind(c_4)=-1,\ \Ind(c_5)=1, 
			$$
			$$
			\Ind(c_6)=-3,\ \Ind(c_7)=3,\ \Ind(c_8)=-3,\ \Ind(c_9)=3,\ \Ind(c_{10})=-3;
			$$
			and the index of each chord of the Gauss diagram of $K_{6}$,
			$$
			\Ind(c_1')=-2,\ \Ind(c_2')=-2,\ \Ind(c_3')=-3,\ \Ind(c_4')=1,\ \Ind(c_5')=1, 
			$$
			$$
			\Ind(c_6')=-3,\ \Ind(c_7')=3,\ \Ind(c_8')=-3,\ \Ind(c_9')=3,\ \Ind(c_{10}')=-3.
			$$
			
			It follows that the maximum absolute chord value of the two knotoids equals $4$. Consequently, every crossing in the diagrams of $K_5$ and $K_{6}$ becomes virtual under the $r$-covering whenever $r>4$, making both $K_5^{(r)}$ and $K_{6}^{(r)}$ trivial knotoids and hence $d_G(K_5^{(r)}, K_{6}^{(r)}) = 0$. We then examine the cases $r=1,2,3,4$. 
			
			When $r = 1$, no crossings of the diagrams of $K_5$ and $K_{6}$ are virtualized by the $1$-covering, and thus $d_G(K_5^{(1)}, K_{6}^{(1)}) = d_G(K_5, K_{6})$.
			
			When $r = 2$, we consider the $2$-covering, and find that the diagram of $K_5^{(2)}$ is illustrated in Figure~\ref{fig20}, and $K_{6}^{(2)}$ is the trivial knotoid. So $d_G(K_5^{(2)}, K_{6}^{(2)}) = 1$.
			
			When $r=3$, the $3$-covering virtualizes all the distinct crossings between the diagrams of $K_5$ and $K_{6}$, so $K_5^{(3)}=K_{6}^{(3)}$, that is, $d_G(K_5^{(3)}, K_{6}^{(3)}) = 0$.
			
			In summary, the $r$-coverings yield $d_G(K_5, K_6) \geq 1$, and we can find that the diagrams of $K_5$ and $K_{6}$ differ by two crossing changes, so $d_G(K_5, K_{6}) \le 2$. Therefore, $1\le d_G(K_5, K_{6}) \le 2$.
		}
	\end{example}
	
	It thus follows that for any pair of homotopic planar virtual knotoids, evaluating their exact Gordian distance may require auxiliary methods in addition to the $r$-covering.
	
	\section{$\Delta$-Gordian distance: proof of Theorem~\ref{thm-decov}.} \label{sec6}
	
	This section opens with the proof of Theorem~\ref{thm-decov}. We also supply an explicit example for computing the $\Delta$-Gordian distance between $\Delta$-homotopic planar virtual knotoids. 
	
	\begin{proof}[Proof of Theorem~\ref{thm-decov}.]
		Since $K$ and $K'$ are $\Delta$-homotopic, set $d_\Delta(K,K')=n$. There exists a sequence of planar virtual knotoids $K_0, \dots, K_n$ such that $K_0=K$, $K_n=K'$, where for each $i=1,\dots,n$, one diagram of $K_{i-1}$ can be converted into one diagram of $K_i$ via exactly one $\Delta$-move. Let $D_{i-1}$ and $D_{i}$ be such diagrams for $K_{i-1}$ and $K_i$, and let $G_{i-1}$, $G_{i}$ denote their respective Gauss diagrams, such that the $\Delta$-move for $D_{i-1}$ converts it into $D_{i}$. Let $c_1$, $c_2$, and $c_3$ be the crossings of  $D_{i-1}$ involved in the $\Delta$-move, and we denote the crossings of $D_{i}$ corresponding to $c_1$, $c_2$, $c_3$ by $c_1'$, $c_2'$, $c_3'$. We see from the $\Delta$-moves on a Gauss diagram depicted in Figure~\ref{fig10} that these $\Delta$-moves leave the index of every chord unchanged. Combining this with the $\Delta$-moves on diagrams shown in Figure~\ref {fig7}, we find that when we proceed to apply the $\Delta$-move, all crossings of $D_{i-1}$ remain unchanged with the exception of $c_1$, $c_2$, $c_3$, and $\Ind(c_j) = \Ind(c_j')$ ($j = 1, 2, 3$). We may verify that $\Ind(c_1) + \Ind(c_2) + \Ind(c_3) = 0$. Four cases are now discussed separately.
		
		\noindent\textbf{Case (a).}
		The numbers $\Ind(c_1)$, $\Ind(c_2)$ and $\Ind(c_3)$ are all divisible by $r$. The $r$-covering preserves the crossings $c_1$, $c_2$, $c_3$ in $D_{i-1}$ and the crossings $c_1'$, $c_2'$, $c_3'$ in  $D_{i}$. Then, a $\Delta$-move from $D_{i-1}$ to $D_{i}$ corresponds to a $\Delta$-move from $D_{i-1}^{(r)}$ to $D_{i}^{(r)}$. Thus, $D_{i-1}^{(r)}$ and $D_i^{(r)}$ differ by a $\Delta$-move.
		
		\noindent\textbf{Case (b).}
		One of the numbers is divisible by $r$ and the other two are not. In this case, assume $\Ind(c_1)$ is divisible by $r$. The $r$-covering virtualizes crossings $c_2, c_3$ among $c_1, c_2, c_3$ in $D_{i-1}$ and  crossings $c_2', c_3'$ among $c_1', c_2', c_3'$ in $D_{i}$. Then, a $\Delta$-move from $D_{i-1}$ to $D_{i}$  corresponds to an $\Omega_3^m$-move from $D_{i-1}^{(r)}$ to $D_{i}^{(r)}$. Hence, $D_{i-1}^{(r)}$ and $D_i^{(r)}$ are equivalent.
		
		\noindent\textbf{Case (c).}
		Two of the numbers are divisible by $r$ and the other is not. However, since $\Ind(c_1) + \Ind(c_2) + \Ind(c_3) = 0$, this case does not occur.
		
		\noindent\textbf{Case (d).}
		None of these numbers is divisible by $r$. Hence, the $r$-covering virtualizes the crossings $c_1$, $c_2$, $c_3$ in $D_{i-1}$ and the crossings $c_1'$, $c_2'$, $c_3'$ in $D_{i}$. Then, a $\Delta$-move from $D_{i-1}$ to $D_{i}$ corresponds to an $\Omega_3^v$-move from $D_{i-1}^{(r)}$ to $D_{i}^{(r)}$. Therefore,  $D_{i-1}^{(r)}$ and $D_i^{(r)}$ are equivalent.
		
		Thus, we obtain a sequence of planar virtual knotoid diagrams $D_0^{(r)}, \dots, D_n^{(r)}$ satisfying $D_0^{(r)}=D^{(r)}$, $D_n^{(r)}=D'^{(r)}$, where $D_{i-1}^{(r)}$ and $D_i^{(r)}$ differ by a $\Delta$-move or are identical ($i = 1 , \dots , n$). Therefore, $K^{(r)}$ is $\Delta$-homotopic to $K'^{(r)}$. Moreover, $d_\Delta(K, K') \geq d_\Delta(K^{(r)}, K'^{(r)})$. 
	\end{proof}
	
	One readily verifies that a $\Delta$-move can be realized by two crossing changes and an $\Omega_3$-move. Thus, we obtain the following.
	
	\begin{proposition}\label{pro 5.1}
		If $K$ and $K'$ are $\Delta$-homotopic planar virtual knotoids, then $K$ and $K'$ are homotopic.
	\end{proposition}
	
	Theorem~\ref{thm-UVW} implies all planar virtual knotoids in $\mathcal{U}$ are non-homotopic to those in $\mathcal{W}$. By Proposition~\ref{pro 5.1}, no planar virtual knotoid in $\mathcal{U}$ is $\Delta$-homotopic to any planar virtual knotoid in $\mathcal{W}$.
	
	Theorem~\ref{thm-decov} provides a method for calculating the $\Delta$-Gordian distances of $\Delta$-homotopic planar virtual knotoids. We supply an explicit example to illustrate this construction below.
	
	\begin{example} \label{ex:10} {\rm
			The diagrams and Gauss diagrams of $K_7$ and $K_8$ are shown in Figure~\ref{fig32}~(a) and ~(b).
			\begin{figure}[htbp]
				\begin{center}
					\begin{subfigure}{0.34\textwidth}
						\centering
						\tikzset{every picture/.style={line width=0.75pt}}
						\begin{tikzpicture}[x=0.75pt,y=0.75pt,yscale=-0.8,xscale=0.8]
							\draw [color={rgb, 255:red, 208; green, 2; blue, 27 }  ,draw opacity=1 ][red, very thick]    (54.46,123.15) .. controls (56.62,106.38) and (78.33,95.67) .. (93.13,106.93) ;
							\draw [color={rgb, 255:red, 208; green, 2; blue, 27 }  ,draw opacity=1 ][red, very thick]    (108.53,117.53) .. controls (124.2,129.6) and (157.07,159.8) .. (156.6,180) ; 
							\draw [black, very thick]    (156.2,199.2) .. controls (155.8,214.8) and (133.6,222.86) .. (118,220.87) .. controls (102.4,218.87) and (80.2,204.75) .. (78.35,182.6) ;
							\draw [black, very thick]    (77.38,161.92) .. controls (79,140) and (87.23,127.62) .. (100.61,112.52) .. controls (113.99,97.42) and (131.08,94.85) .. (136.62,99.15) .. controls (142.15,103.46) and (141.38,116.08) .. (134.77,126.54) ; 
							\draw [black, very thick]    (98.3,222) .. controls (94.55,237.25) and (115.49,242.42) .. (138.35,242.25) .. controls (161.2,242.09) and (186.4,238.2) .. (198.4,214.6) .. controls (210.4,191) and (208.6,151.8) .. (196.05,129) .. controls (183.5,106.2) and (168.22,89.92) .. (147.6,81.6) .. controls (126.98,73.28) and (101.2,71) .. (80,77) .. controls (58.8,83) and (37.4,100.2) .. (31.6,122.2) .. controls (25.8,144.2) and (35.47,168.79) .. (52.55,174.12) .. controls (69.63,179.45) and (99.88,161.54) .. (121.33,145.46) ;
							\draw [black, very thick]    (102.35,205.92) .. controls (104.75,195.25) and (114.08,187.65) .. (121.42,191.65) .. controls (128.75,195.65) and (129.02,202.18) .. (141.02,212.58) .. controls (153.02,222.98) and (173.42,227.92) .. (177.35,217.65) .. controls (181.28,207.38) and (169.55,196.58) .. (155.28,190.45) .. controls (141.02,184.32) and (131.3,184.73) .. (128.15,177.05) .. controls (125,169.38) and (130.63,162.63) .. (142.38,157.38) .. controls (154.13,152.13) and (167.21,147.99) .. (184.86,135.29) ;
							\draw [black, very thick]    (204.47,121.42) .. controls (210.06,117.59) and (213.94,113.82) .. (217.47,109) ;
							\draw  [black, very thick]  (155.06,80.97) -- (159.98,87.88) -- (151.86,87.31) ;
							\draw  [fill={rgb, 255:red, 0; green, 0; blue, 0 }  ,fill opacity=1 ][black, very thick]  (212.95,109) .. controls (212.95,106.5) and (214.97,104.48) .. (217.47,104.48) .. controls (219.97,104.48) and (221.99,106.5) .. (221.99,109) .. controls (221.99,111.5) and (219.97,113.52) .. (217.47,113.52) .. controls (214.97,113.52) and (212.95,111.5) .. (212.95,109) -- cycle ;
							\draw  [fill={rgb, 255:red, 0; green, 0; blue, 0 }  ,fill opacity=1 ][black, very thick]  (49.94,123.15) .. controls (49.94,120.66) and (51.97,118.63) .. (54.46,118.63) .. controls (56.96,118.63) and (58.98,120.66) .. (58.98,123.15) .. controls (58.98,125.65) and (56.96,127.67) .. (54.46,127.67) .. controls (51.97,127.67) and (49.94,125.65) .. (49.94,123.15) -- cycle ;
							\draw  [black, very thick]  (140.24,215.48) .. controls (140.24,212.98) and (142.26,210.96) .. (144.76,210.96) .. controls (147.25,210.96) and (149.28,212.98) .. (149.28,215.48) .. controls (149.28,217.98) and (147.25,220) .. (144.76,220) .. controls (142.26,220) and (140.24,217.98) .. (140.24,215.48) -- cycle ;
							\draw  [black, very thick]  (141.83,155.73) .. controls (141.83,153.24) and (143.85,151.21) .. (146.35,151.21) .. controls (148.84,151.21) and (150.87,153.24) .. (150.87,155.73) .. controls (150.87,158.23) and (148.84,160.25) .. (146.35,160.25) .. controls (143.85,160.25) and (141.83,158.23) .. (141.83,155.73) -- cycle ;
							\draw (125,260) node [anchor=north west][inner sep=0.75pt]  {$K_7$};
						\end{tikzpicture}
					\end{subfigure}
					\qquad 
					\begin{subfigure}{0.34\textwidth}
						\centering
						\tikzset{every picture/.style={line width=0.75pt}}
						\begin{tikzpicture}[x=0.75pt,y=0.75pt,yscale=-0.8,xscale=0.8]
							\draw [color={rgb, 255:red, 208; green, 2; blue, 27 }  ,draw opacity=1 ][red, very thick]    (375.76,191.21) .. controls (387.35,180.87) and (402.71,172.15) .. (415.76,170.04) .. controls (428.82,167.92) and (443.55,174.75) .. (443.79,183.08) ;
							\draw [black, very thick]    (443.2,199.18) .. controls (442.8,214.78) and (420.6,222.84) .. (405,220.84) .. controls (389.4,218.85) and (367.2,204.73) .. (365.35,182.58) ;
							\draw [black, very thick]    (364.38,161.9) .. controls (366,139.98) and (374.23,127.59) .. (387.61,112.5) .. controls (400.99,97.4) and (416.6,95.38) .. (421.8,99.58) .. controls (427,103.78) and (424.92,121.48) .. (408,141.18) .. controls (391.08,160.88) and (370.2,173.98) .. (353.2,175.18) ;
							\draw [black, very thick]    (385.3,221.98) .. controls (381.55,237.23) and (402.49,242.39) .. (425.35,242.23) .. controls (448.2,242.07) and (473.4,238.18) .. (485.4,214.58) .. controls (497.4,190.98) and (495.6,151.78) .. (483.05,128.98) .. controls (470.5,106.18) and (455.22,89.9) .. (434.6,81.58) .. controls (413.98,73.25) and (388.2,70.98) .. (367,76.98) .. controls (345.8,82.98) and (324.4,100.18) .. (318.6,122.18) .. controls (312.8,144.18) and (322.47,168.77) .. (339.55,174.1) ;
							\draw [black, very thick]    (389.35,205.89) .. controls (391.75,195.23) and (401.08,187.63) .. (408.42,191.63) .. controls (415.75,195.63) and (416.02,202.16) .. (428.02,212.56) .. controls (440.02,222.96) and (460.42,227.89) .. (464.35,217.63) .. controls (468.28,207.36) and (456.55,196.56) .. (442.28,190.43) .. controls (428.02,184.29) and (418.3,184.7) .. (415.15,177.03) .. controls (412,169.35) and (417.63,162.6) .. (429.38,157.35) .. controls (441.13,152.1) and (454.21,147.97) .. (471.86,135.26) ; 
							\draw [color={rgb, 255:red, 208; green, 2; blue, 27 }  ,draw opacity=1 ][red, very thick]    (349.6,139.78) .. controls (347,153.26) and (345.4,169.58) .. (346.57,182.12) .. controls (347.74,194.66) and (353.14,214.83) .. (364.2,203.38) ;
							\draw  [fill={rgb, 255:red, 0; green, 0; blue, 0 }  ,fill opacity=1 ][black, very thick]  (345.08,139.78) .. controls (345.08,137.28) and (347.1,135.26) .. (349.6,135.26) .. controls (352.1,135.26) and (354.12,137.28) .. (354.12,139.78) .. controls (354.12,142.27) and (352.1,144.3) .. (349.6,144.3) .. controls (347.1,144.3) and (345.08,142.27) .. (345.08,139.78) -- cycle ;
							\draw  [black, very thick]  (427.43,215.52) .. controls (427.43,213.02) and (429.46,211) .. (431.95,211) .. controls (434.45,211) and (436.47,213.02) .. (436.47,215.52) .. controls (436.47,218.02) and (434.45,220.04) .. (431.95,220.04) .. controls (429.46,220.04) and (427.43,218.02) .. (427.43,215.52) -- cycle ;
							\draw [black, very thick]    (491.4,121.4) .. controls (496.99,117.57) and (500.87,113.8) .. (504.4,108.98) ; 
							\draw  [fill={rgb, 255:red, 0; green, 0; blue, 0 }  ,fill opacity=1 ][black, very thick]  (499.88,108.98) .. controls (499.88,106.48) and (501.9,104.46) .. (504.4,104.46) .. controls (506.9,104.46) and (508.92,106.48) .. (508.92,108.98) .. controls (508.92,111.48) and (506.9,113.5) .. (504.4,113.5) .. controls (501.9,113.5) and (499.88,111.48) .. (499.88,108.98) -- cycle ;
							\draw  [black, very thick]  (410.39,170.04) .. controls (410.39,167.54) and (412.41,165.52) .. (414.91,165.52) .. controls (417.4,165.52) and (419.43,167.54) .. (419.43,170.04) .. controls (419.43,172.53) and (417.4,174.56) .. (414.91,174.56) .. controls (412.41,174.56) and (410.39,172.53) .. (410.39,170.04) -- cycle ;
							\draw  [black, very thick]  (441.5,80.52) -- (446.03,87.69) -- (437.95,86.68) ;
							\draw (410,260) node [anchor=north west][inner sep=0.75pt]  {$K_8$};
						\end{tikzpicture}
					\end{subfigure}
					\begin{subfigure}{0.34\textwidth}
						\centering
						\tikzset{every picture/.style={line width=0.75pt}}
						\begin{tikzpicture}[x=0.75pt,y=0.75pt,yscale=-1,xscale=1]
							\draw  [draw opacity=0][black, very thick]  (211.44,604.35) .. controls (201.68,628.35) and (178.45,645.24) .. (151.34,645.24) .. controls (115.44,645.24) and (86.34,615.62) .. (86.34,579.09) .. controls (86.34,542.55) and (115.44,512.94) .. (151.34,512.94) .. controls (178.63,512.94) and (201.98,530.04) .. (211.62,554.29) -- (151.34,579.09) -- cycle ; \draw  [black, very thick]  (211.44,604.35) .. controls (201.68,628.35) and (178.45,645.24) .. (151.34,645.24) .. controls (115.44,645.24) and (86.34,615.62) .. (86.34,579.09) .. controls (86.34,542.55) and (115.44,512.94) .. (151.34,512.94) .. controls (178.63,512.94) and (201.98,530.04) .. (211.62,554.29) ;
							\draw  [fill={rgb, 255:red, 0; green, 0; blue, 0 }  ,fill opacity=1 ][black, very thick]  (206.92,604.35) .. controls (206.92,601.85) and (208.94,599.83) .. (211.44,599.83) .. controls (213.93,599.83) and (215.96,601.85) .. (215.96,604.35) .. controls (215.96,606.84) and (213.93,608.87) .. (211.44,608.87) .. controls (208.94,608.87) and (206.92,606.84) .. (206.92,604.35) -- cycle ; 
							\draw  [fill={rgb, 255:red, 0; green, 0; blue, 0 }  ,fill opacity=1 ][black, very thick]  (207.1,554.29) .. controls (207.1,551.79) and (209.13,549.77) .. (211.62,549.77) .. controls (214.12,549.77) and (216.14,551.79) .. (216.14,554.29) .. controls (216.14,556.79) and (214.12,558.81) .. (211.62,558.81) .. controls (209.13,558.81) and (207.1,556.79) .. (207.1,554.29) -- cycle ;
							\draw  [black, very thick, ->]  (112.38,526.08) -- (150,645) ; 
							\draw  [black, very thick, ->]  (171.77,641.92) -- (133,516) ;
							\draw  [black, very thick, ->]  (119.31,636.69) -- (196,627) ;
							\draw [color={rgb, 255:red, 208; green, 2; blue, 27 }  ,draw opacity=1 ]  [red, very thick, ->] (105.46,626.23) -- (95,547) ;
							\draw [color={rgb, 255:red, 208; green, 2; blue, 27 }  ,draw opacity=1 ] [red, very thick, ->]  (86.69,572.54) -- (192,528) ;
							\draw [color={rgb, 255:red, 208; green, 2; blue, 27 }  ,draw opacity=1 ] [red, very thick, ->]  (163.69,514.03) -- (90,600) ;
							\draw (70,598) node [anchor=north west][inner sep=0.75pt]  [font=\large]  {$c_{2}$};
							\draw (121,498) node [anchor=north west][inner sep=0.75pt]  [font=\large]  {$c_{3}$};
							\draw (142,648) node [anchor=north west][inner sep=0.75pt]  [font=\large]  {$c_{4}$};
							\draw (72,568) node [anchor=north west][inner sep=0.75pt]  [font=\scriptsize]  {$-$};
							\draw (160,503) node [anchor=north west][inner sep=0.75pt]  [font=\scriptsize]  {$-$};
							\draw (99,628) node [anchor=north west][inner sep=0.75pt]  [font=\scriptsize]  {$-$};
							\draw (104,513) node [anchor=north west][inner sep=0.75pt]  [font=\scriptsize]  {$-$};
							\draw (81,525) node [anchor=north west][inner sep=0.75pt]  [font=\large]  {$c_{5}$};
							\draw (193,516) node [anchor=north west][inner sep=0.75pt]  [font=\large]  {$c_{1}$};
							\draw (200.02,620) node [anchor=north west][inner sep=0.75pt]  [font=\large]  {$c_{6}$};
							\draw (168,645) node [anchor=north west][inner sep=0.75pt]  [font=\scriptsize]  {$+$};
							\draw (105.69,633.98) node [anchor=north west][inner sep=0.75pt]  [font=\scriptsize]  {$+$};
							\draw (142,666) node [anchor=north west][inner sep=0.75pt]    {(a)};
						\end{tikzpicture}
					\end{subfigure}
					\qquad 
					\begin{subfigure}{0.34\textwidth}
						\centering
						\tikzset{every picture/.style={line width=0.75pt}}
						\begin{tikzpicture}[x=0.75pt,y=0.75pt,yscale=-1,xscale=1]
							\draw  [draw opacity=0][black, very thick]  (407.44,603.01) .. controls (397.68,627.02) and (374.45,643.9) .. (347.34,643.9) .. controls (311.44,643.9) and (282.34,614.29) .. (282.34,577.75) .. controls (282.34,541.22) and (311.44,511.6) .. (347.34,511.6) .. controls (374.63,511.6) and (397.98,528.71) .. (407.62,552.96) -- (347.34,577.75) -- cycle ; \draw  [black, very thick]  (407.44,603.01) .. controls (397.68,627.02) and (374.45,643.9) .. (347.34,643.9) .. controls (311.44,643.9) and (282.34,614.29) .. (282.34,577.75) .. controls (282.34,541.22) and (311.44,511.6) .. (347.34,511.6) .. controls (374.63,511.6) and (397.98,528.71) .. (407.62,552.96) ;
							\draw  [fill={rgb, 255:red, 0; green, 0; blue, 0 }  ,fill opacity=1 ][black, very thick]  (402.92,603.01) .. controls (402.92,600.52) and (404.94,598.49) .. (407.44,598.49) .. controls (409.93,598.49) and (411.96,600.52) .. (411.96,603.01) .. controls (411.96,605.51) and (409.93,607.53) .. (407.44,607.53) .. controls (404.94,607.53) and (402.92,605.51) .. (402.92,603.01) -- cycle ;
							\draw  [fill={rgb, 255:red, 0; green, 0; blue, 0 }  ,fill opacity=1 ][black, very thick]  (403.1,552.96) .. controls (403.1,550.46) and (405.13,548.44) .. (407.62,548.44) .. controls (410.12,548.44) and (412.14,550.46) .. (412.14,552.96) .. controls (412.14,555.45) and (410.12,557.48) .. (407.62,557.48) .. controls (405.13,557.48) and (403.1,555.45) .. (403.1,552.96) -- cycle ; 
							\draw  [black, very thick,->]  (308.38,524.74) -- (345,643) ;
							\draw  [black, very thick,->]  (367.77,640.59) -- (328,515) ; 
							\draw  [black, very thick,->]  (315.31,635.36) -- (392,626) ;
							\draw [color={rgb, 255:red, 208; green, 2; blue, 27 }  ,draw opacity=1 ]  [red, very thick, ->] (292.2,612.8) -- (284.89,562.58) ;
							\draw [color={rgb, 255:red, 208; green, 2; blue, 27 }  ,draw opacity=1 ]  [red, very thick, ->] (291.46,543.97) -- (358,513) ;
							\draw [color={rgb, 255:red, 208; green, 2; blue, 27 }  ,draw opacity=1 ]  [red, very thick, ->] (385.2,524.2) -- (301,625) ;
							\draw (283,623) node [anchor=north west][inner sep=0.75pt]  [font=\large]  {$c_{2}'$};
							\draw (316,490) node [anchor=north west][inner sep=0.75pt]  [font=\large]  {$c_{3}'$};
							\draw (336,646) node [anchor=north west][inner sep=0.75pt]  [font=\large]  {$c_{4}'$};
							\draw (278,540) node [anchor=north west][inner sep=0.75pt]  [font=\scriptsize]  {$-$};
							\draw (383,514) node [anchor=north west][inner sep=0.75pt]  [font=\scriptsize]  {$-$};
							\draw (283,615) node [anchor=north west][inner sep=0.75pt]  [font=\scriptsize]  {$-$};
							\draw (301,514) node [anchor=north west][inner sep=0.75pt]  [font=\scriptsize]  {$-$};
							\draw (264,548) node [anchor=north west][inner sep=0.75pt]  [font=\large]  {$c_{5}'$};
							\draw (356,492) node [anchor=north west][inner sep=0.75pt]  [font=\large]  {$c_{1}'$};
							\draw (395,616) node [anchor=north west][inner sep=0.75pt]  [font=\large]  {$c_{6}'$};
							\draw (365,643) node [anchor=north west][inner sep=0.75pt]  [font=\scriptsize]  {$+$};
							\draw (301.69,632.65) node [anchor=north west][inner sep=0.75pt]  [font=\scriptsize]  {$+$};
							\draw (336,666) node [anchor=north west][inner sep=0.75pt]    {(b)};
						\end{tikzpicture}
					\end{subfigure}		
					\caption{Diagrams and Gauss diagrams of planar virtual knotoids $K_7$ and $K_8$.}
					\label{fig32}
				\end{center}
			\end{figure}
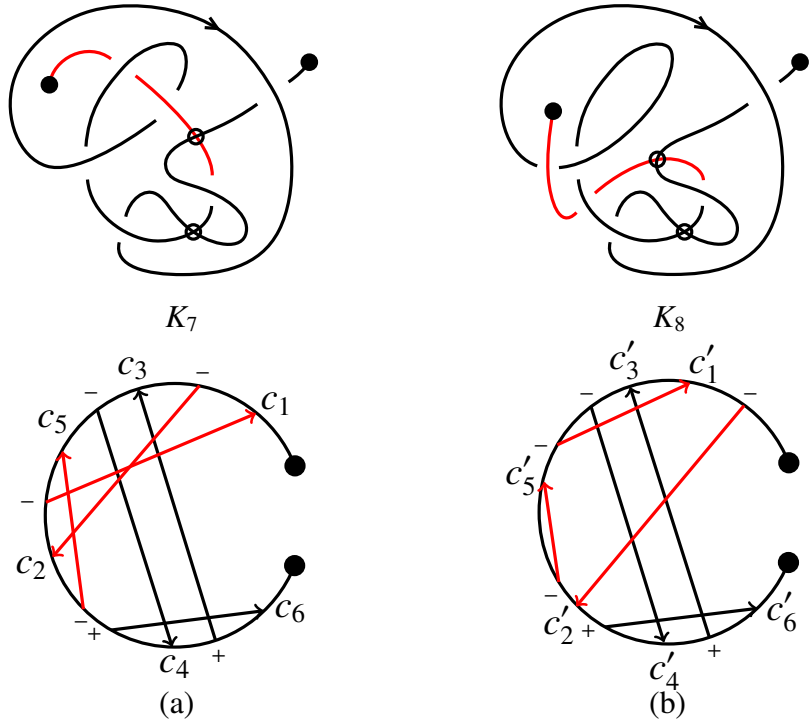
			Calculation gives all chord indices of the Gauss diagram of $K_7$,
			\begin{align*}
				&
				\Ind(c_1)=-2,\ \Ind(c_2)=2,\ \Ind(c_3)=1,\ \Ind(c_4)=-1,\ \Ind(c_5)=0,\ \Ind(c_6)=-2;
			\end{align*}
			likewise, we compute the chord indices associated with the Gauss diagram of $K_8$,
			\begin{align*}
				&
				\Ind(c_1')=-2,\ \Ind(c_2')=2,\ \Ind(c_3')=1,\ \Ind(c_4')=-1,\ \Ind(c_5\cyan{\sout{}{'}})=0,\ \Ind(c_6\cyan{\sout{}{'}})=-2.
			\end{align*}
			
			However, the diagram of $K_7^{(2)}$ is presented in Figure~\ref{fig34}, while $K_8^{(2)}$ is the trivial knotoid. Although the diagrams of $K_7^{(2)}$ and $K_8^{(2)}$ differ by a $\Delta$-move and and four $\Omega_1$-moves, a direct computation of their normalized bracket polynomials yields
			$$
			\langle K_7^{(2)}\rangle_\circ = -A^{16} + A^{12} + A^{4},\ \langle K_8^{(2)}\rangle_\circ = 1, 
			$$
			which implies that $K_7^{(2)}$ and $K_8^{(2)}$ are $\Delta$-homotopic but not equivalent, so $d_\Delta(K_7^{(2)}, K_8^{(2)})\geq 1$. Thus, $d_\Delta(K_7^{(2)}, K_8^{(2)}) = 1$. By Theorem~\ref{thm-decov}, $d_\Delta(K_7, K_8)\geq d_\Delta(K_7^{(2)}, K_8^{(2)}) = 1$. We can also observe that the diagrams of $K_7$ and $K_8$ differ by a $\Delta$-move, so $d_\Delta(K_7, K_8) \le 1$. Therefore, $d_\Delta(K_7, K_8) = d_\Delta(K_7^{(2)}, K_8^{(2)}) = 1$.
			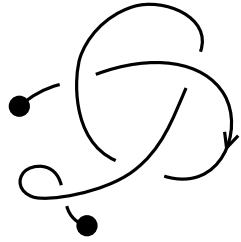
\begin{figure}[htbp]
				\begin{center}
					\tikzset{every picture/.style={line width=0.7pt}}
					\begin{tikzpicture}[x=0.75pt,y=0.75pt,yscale=-1,xscale=1]
						\draw  [black, very thick]  (392.56,395.46) -- (386.72,401.6) -- (385.73,393.52) ;
						\draw [black, very thick]    (321.13,364.71) .. controls (348.58,354.34) and (375.38,357.39) .. (385.6,374.95) .. controls (395.82,392.52) and (384,424.52) .. (355.45,415.79) ;
						\draw [black, very thick]    (373.63,353.33) .. controls (379.75,335.83) and (354.89,326) .. (339.88,330.71) .. controls (324.86,335.42) and (316.16,347.3) .. (313.25,356.21) .. controls (310.34,365.12) and (308.24,397.9) .. (331.09,407.97) ; 
						\draw [black, very thick]    (303.8,420.33) .. controls (300.92,408.33) and (289,409.67) .. (285.25,413.58) .. controls (281.5,417.5) and (282.5,423.83) .. (289.67,426.25) .. controls (296.83,428.67) and (325.02,424.27) .. (339.64,413.42) .. controls (354.25,402.58) and (359.5,387.58) .. (366.4,371.53) ;
						\draw [black, very thick]   (282.75,380.7) .. controls (288.75,375.5) and (295.3,371.79) .. (303,369.79) ; 
						\draw [black, very thick]    (306.57,430.8) .. controls (307.88,436.29) and (312.29,439.79) .. (316.7,440.66) ;
						\draw  [fill={rgb, 255:red, 0; green, 0; blue, 0 }  ,fill opacity=1 ][black, very thick]  (287.27,380.69) .. controls (287.27,383.19) and (285.25,385.21) .. (282.76,385.22) .. controls (280.26,385.22) and (278.23,383.2) .. (278.23,380.7) .. controls (278.23,378.21) and (280.25,376.18) .. (282.74,376.18) .. controls (285.24,376.17) and (287.27,378.19) .. (287.27,380.69) -- cycle ;
						\draw  [fill={rgb, 255:red, 0; green, 0; blue, 0 }  ,fill opacity=1 ][black, very thick]  (321.22,440.65) .. controls (321.22,443.15) and (319.2,445.17) .. (316.71,445.18) .. controls (314.21,445.18) and (312.18,443.16) .. (312.18,440.66) .. controls (312.18,438.17) and (314.2,436.14) .. (316.69,436.14) .. controls (319.19,436.13) and (321.22,438.15) .. (321.22,440.65) -- cycle ;
					\end{tikzpicture}
					\caption{The diagram of $K_7^{(2)}$.}
					\label{fig34}
				\end{center}
			\end{figure}
			
			From the diagrams of $K_7^{(2)}$ and $K_8^{(2)}$, we see they differ by a crossing change, an $\Omega_1$-move and an $\Omega_2$-move, so $d_G(K_7^{(2)}, K_8^{(2)}) = 1$. By Theorem~\ref{thm-ccov}, $d_G(K_7, K_8)\geq d_G(K_7^{(2)}, K_8^{(2)}) = 1$. The diagrams of $K_7$ and $K_8$ differ by one $\Delta$-move composed of two crossing changes and one $\Omega_3$-move, hence $d_G (K_7, K_8) \le 2$. Therefore, $1 \le d_G(K_7, K_8) \le 2$. 
			
			To summarize, $d_\Delta(K_7, K_8) = 1$ and $1 \le d_G(K_7, K_8) \le 2$.
		}
	\end{example}

	\section{VRCC-Gordian distance: proof of Theorem~\ref{thm-vrcov}} \label{sec7}
	
	In this section, the proof of Theorem~\ref{thm-vrcov} is first presented and an example of computing the VRCC-Gordian distance between VRCC-homotopic planar virtual knotoids are provided.
	
	\begin{proof}[Proof of Theorem~\ref{thm-vrcov}.]
		Since $K$ and $K'$ are VRCC-homotopic, set $d_\text{VRCC}(K,K')=n$. There exists a sequence of planar virtual knotoids $K_0, \dots, K_n$ satisfying $K_0=K$, $K_n=K'$, where $d_\text{VRCC}(K_{i-1}, K_i) = 1$ ($i = 1 , \dots , n$). For planar virtual knotoids $K_{i-1}$ and $K_i$, let $D_{i-1}$ be a diagram of $K_{i-1}$ and $D_i$ a diagram of $K_i$. The diagrams $D_{i-1}$ and $D_i$ differ by a VRCC-move.  Let $M_{i-1}$ be the set of all crossings involved in this VRCC-move in $D_{i-1}$, and $M_i$ the set of corresponding crossings in $D_i$. Since a VRCC-move consists of multiple crossing changes, when we apply the VRCC-move at $D_{i-1}$, the indices of all crossings of $D_{i-1}$ remain unchanged with the exception of $M_{i-1}$. Moreover, the index of each  crossing in $M_{i-1}$ is the negative of that of its counterpart in $M_i$. Thus, all crossings of $D_{i-1}^{(r)}$ and $D_{i}^{(r)}$ coincide with the exception of the crossings corresponding to $M_{i-1}$ and $M_{i}$. We now discuss the relationship between $D_{i-1}^{(r)}$ and $D_i^{(r)}$ in three cases.
		
		\noindent\textbf{Case (a).}
		The index of every crossing in $M_{i-1}$ is divisible by $r$. There are crossings in $D_{i-1}^{(r)}$ and $D_{i}^{(r)}$ corresponding to $M_{i-1}$ and $M_{i}$, respectively. Then, the VRCC-move from $D_{i-1}$ to  $D_i$ corresponds to a VRCC-move from $D_{i-1}^{(r)}$ to $D_{i}^{(r)}$. Thus, $D_{i-1}^{(r)}$ and $D_i^{(r)}$ differ by a VRCC-move.
		
		\noindent\textbf{Case (b).}
		Some indices of crossings in $M_{i-1}$ are divisible by $r$, and the rest are not. Then all crossings with indices indivisible by $r$ in $M_{i-1}$ and $M_i$ are virtualized. The underlying $4$-valent graph, and hence the region at which the move is performed, is unchanged by virtualization. Its boundary contains precisely the retained classical crossings from $M_{i-1}$. Therefore, $D_{i-1}^{(r)}$ and $D_i^{(r)}$ still differ by a VRCC-move, which now involves fewer classical crossings.
		
		\noindent\textbf{Case (c).}
		None of the indices of crossings in $M_{i-1}$ is divisible by $r$. Then all crossings in $M_{i-1}$ of  $D_{i-1}$ and all crossings in $M_i$ of $D_i$ are virtualized. Consequently, $D_{i-1}^{(r)}$ coincides exactly with $D_{i}^{(r)}$. 
		
		Thus, we obtain a sequence of planar virtual knotoids $K_0^{(r)}, \dots, K_n^{(r)}$ satisfying $K_0^{(r)}=K^{(r)}$, $K_n^{(r)}=K'^{(r)}$, where $d_\text{VRCC}(K_{i-1}^{(r)}, K_{i}^{(r)})\le 1$ ($i = 1 , \dots , n$). Therefore,  $K^{(r)}$ is VRCC-homotopic to $K'^{(r)}$. Moreover, $d_\text{VRCC}(K, K') \geq d_\text{VRCC}(K^{(r)}, K'^{(r)})$. 
	\end{proof}
	
	We can confirm directly that any VRCC-move can be decomposed into finitely many crossing changes. This leads to the following result.
	
	\begin{proposition}\label{pro 6.1}
		If two planar virtual knotoids $K$ and $K'$ are VRCC-homotopic, then they are homotopic.
	\end{proposition}
	
	Theorem~\ref{thm-UVW} implies all planar virtual knotoids in $\mathcal{V}$ are non-homotopic to those in $\mathcal{W}$. By Proposition ~\ref{pro 6.1}, no planar virtual knotoid in $\mathcal{V}$ is VRCC-homotopic to any planar virtual knotoid in~$\mathcal{W}$.
	
	Theorem~\ref{thm-vrcov} allows us to compute the VRCC-Gordian distances between VRCC-homotopic planar virtual knotoids. We demonstrate this with an example below.
	
	\begin{example} \label{ex:7} {\rm
			The diagrams and Gauss diagrams of $K_9$ and $K_{10}$ are illustrated in Figure~\ref{fig21}~(a) and ~(b). 
			\begin{figure}[htbp]
				\begin{center}
					\begin{subfigure}{0.34\textwidth}
						\centering
						\tikzset{every picture/.style={line width=0.75pt}} 
						\begin{tikzpicture}[x=0.75pt,y=0.75pt,yscale=-1,xscale=1]
							\draw  [black, very thick]  (160.32,246.24) -- (154.64,252.53) -- (153.44,244.48) ;
							\draw [black, very thick]    (114.2,204.29) .. controls (110.3,197.37) and (123.85,183.82) .. (140.03,186.35) .. controls (156.2,188.87) and (160.06,208.48) .. (159.28,228.28) .. controls (158.5,248.07) and (150.95,261.79) .. (147.33,265.82) ;
							\draw [black, very thick]    (158.21,173.26) .. controls (138.7,146.52) and (122.75,200.11) .. (151.36,222.53) ;
							\draw [black, very thick]    (222.25,207.33) .. controls (222.15,221.13) and (214.1,225.2) .. (205.61,223.56) ;
							\draw [black, very thick]    (223.62,193.28) .. controls (223.87,192.06) and (225.71,172.73) .. (215.43,176.05) .. controls (205.15,179.37) and (200.18,196.36) .. (190.27,198.72) .. controls (180.36,201.08) and (173.17,195.11) .. (167.45,186.62) ; 
							\draw [black, very thick]    (209.87,190.96) .. controls (232.58,202.73) and (236.74,214.65) .. (233.97,226.4) .. controls (231.2,238.14) and (217.41,242.81) .. (207.77,243.23) .. controls (198.14,243.65) and (181.98,240.26) .. (165.64,230.15) ;
							\draw [black, very thick]    (195.66,182.2) .. controls (181.56,174.93) and (169.9,174.41) .. (162.17,180.65) .. controls (154.45,186.89) and (159.82,205.99) .. (144.62,207.01) ;
							\draw [black, very thick]    (134,208.46) .. controls (117.16,210.1) and (95.41,219.85) .. (90.13,241.63) ;
							\draw [black, very thick]    (87.47,254.11) .. controls (82.05,270.59) and (56.05,258.98) .. (46.59,269.02) .. controls (37.13,279.07) and (47.7,297.06) .. (66.52,301.59) ; 
							\draw [black, very thick]    (121.43,218.13) .. controls (125.27,226.78) and (126.66,237.41) .. (122.85,238.56) .. controls (119.03,239.71) and (116.37,237.2) .. (97.02,226.35) .. controls (77.68,215.5) and (66.61,220.58) .. (68.36,226.46) .. controls (70.11,232.35) and (76.13,239.34) .. (90.45,248.16) .. controls (104.77,256.98) and (145.17,275.72) .. (155.51,282.59) .. controls (165.85,289.47) and (172.34,304.33) .. (155.69,306.43) .. controls (139.04,308.52) and (97.21,306.94) .. (91.48,305.85) ;
							\draw [black, very thick]    (138.13,282.55) .. controls (126.94,296.48) and (101.5,328.77) .. (89.98,328.45) .. controls (78.46,328.13) and (80.16,314.75) .. (76.9,295.14) .. controls (73.64,275.52) and (36.22,266.88) .. (28.41,256.91) .. controls (20.59,246.94) and (32.53,243.71) .. (41.31,244.96) .. controls (50.1,246.21) and (65.05,254.51) .. (70.5,257.86) ;
							\draw [black, very thick]    (105.04,288.76) .. controls (102.83,280.19) and (98.6,276.64) .. (87.35,268.92) ; 
							\draw  [fill={rgb, 255:red, 0; green, 0; blue, 0 }  ,fill opacity=1 ][black, very thick]  (201.09,223.56) .. controls (201.09,221.06) and (203.11,219.04) .. (205.61,219.04) .. controls (208.11,219.04) and (210.13,221.06) .. (210.13,223.56) .. controls (210.13,226.06) and (208.11,228.08) .. (205.61,228.08) .. controls (203.11,228.08) and (201.09,226.06) .. (201.09,223.56) -- cycle ;
							\draw  [black, very thick]  (42.07,269.02) .. controls (42.07,266.53) and (44.09,264.5) .. (46.59,264.5) .. controls (49.09,264.5) and (51.11,266.53) .. (51.11,269.02) .. controls (51.11,271.52) and (49.09,273.54) .. (46.59,273.54) .. controls (44.09,273.54) and (42.07,271.52) .. (42.07,269.02) -- cycle ;
							\draw  [fill={rgb, 255:red, 0; green, 0; blue, 0 }  ,fill opacity=1 ][black, very thick]  (100.52,288.76) .. controls (100.52,286.27) and (102.54,284.24) .. (105.04,284.24) .. controls (107.54,284.24) and (109.56,286.27) .. (109.56,288.76) .. controls (109.56,291.26) and (107.54,293.28) .. (105.04,293.28) .. controls (102.54,293.28) and (100.52,291.26) .. (100.52,288.76) -- cycle ;
							\draw  [black, very thick]  (112.82,306.77) .. controls (112.82,304.28) and (114.84,302.25) .. (117.34,302.25) .. controls (119.84,302.25) and (121.86,304.28) .. (121.86,306.77) .. controls (121.86,309.27) and (119.84,311.29) .. (117.34,311.29) .. controls (114.84,311.29) and (112.82,309.27) .. (112.82,306.77) -- cycle ;
							\draw  [black, very thick]  (92.5,226.35) .. controls (92.5,223.85) and (94.53,221.83) .. (97.02,221.83) .. controls (99.52,221.83) and (101.54,223.85) .. (101.54,226.35) .. controls (101.54,228.85) and (99.52,230.87) .. (97.02,230.87) .. controls (94.53,230.87) and (92.5,228.85) .. (92.5,226.35) -- cycle ;
							\draw  [black, very thick]  (150.57,199.02) .. controls (150.57,196.53) and (152.59,194.5) .. (155.09,194.5) .. controls (157.59,194.5) and (159.61,196.53) .. (159.61,199.02) .. controls (159.61,201.52) and (157.59,203.54) .. (155.09,203.54) .. controls (152.59,203.54) and (150.57,201.52) .. (150.57,199.02) -- cycle ;
							\draw  [black, very thick]  (131.53,186.18) .. controls (131.53,183.69) and (133.56,181.66) .. (136.05,181.66) .. controls (138.55,181.66) and (140.57,183.69) .. (140.57,186.18) .. controls (140.57,188.68) and (138.55,190.7) .. (136.05,190.7) .. controls (133.56,190.7) and (131.53,188.68) .. (131.53,186.18) -- cycle ;
							\draw  [color={rgb, 255:red, 208; green, 2; blue, 27 }  ,draw opacity=1 ][fill={rgb, 255:red, 208; green, 2; blue, 27 }  ,fill opacity=1 ] (172.82,178.27) .. controls (176.91,177.54) and (191.64,181.18) .. (196.36,185.27) .. controls (201.09,189.36) and (194.27,194.73) .. (189.73,196.73) .. controls (185.18,198.73) and (176.73,196.27) .. (170.91,188.18) .. controls (165.09,180.09) and (168.73,179) .. (172.82,178.27) -- cycle ;
							\draw (125,330) node [anchor=north west][inner sep=0.75pt]  {$K_9$};
						\end{tikzpicture}				
					\end{subfigure}
					\qquad 
					\begin{subfigure}{0.34\textwidth}
						\centering
						\tikzset{every picture/.style={line width=0.75pt}} 
						\begin{tikzpicture}[x=0.75pt,y=0.75pt,yscale=-1,xscale=1]
							\draw [color={rgb, 255:red, 0; green, 0; blue, 0 }  ,draw opacity=1 ][black, very thick]    (352.24,207.67) .. controls (348.3,200.66) and (361.91,187.54) .. (378.19,190.5) .. controls (394.46,193.46) and (398.38,213.11) .. (397.63,232.82) .. controls (396.88,252.52) and (389.31,265.98) .. (385.68,269.9) ;
							\draw [color={rgb, 255:red, 0; green, 0; blue, 0 }  ,draw opacity=1 ][black, very thick]    (436.43,199.38) .. controls (428.54,207.43) and (420.03,205.67) .. (412.99,200.59) .. controls (405.95,195.51) and (402.93,186.84) .. (395.82,177.07) ;
							\draw [color={rgb, 255:red, 0; green, 0; blue, 0 }  ,draw opacity=1 ][black, very thick]    (448.46,197.03) .. controls (471.32,209.38) and (475.52,221.39) .. (472.75,233.01) .. controls (469.99,244.64) and (456.12,248.91) .. (446.43,249.06) .. controls (436.75,249.21) and (420.48,245.39) .. (404.03,234.86) ;
							\draw [color={rgb, 255:red, 0; green, 0; blue, 0 }  ,draw opacity=1 ][black, very thick]    (460.93,213.69) .. controls (460.85,227.44) and (452.77,231.27) .. (444.22,229.4) ;
							\draw [color={rgb, 255:red, 0; green, 0; blue, 0 }  ,draw opacity=1 ][black, very thick]    (462.28,199.73) .. controls (463.74,187.76) and (459.82,179.69) .. (454.02,182.32) .. controls (448.22,184.95) and (448.37,185.76) .. (446.32,187.54) ;
							\draw [color={rgb, 255:red, 0; green, 0; blue, 0 }  ,draw opacity=1 ][black, very thick]    (396.46,177.97) .. controls (376.79,150.77) and (360.84,203.74) .. (389.65,226.87) ;
							\draw [color={rgb, 255:red, 0; green, 0; blue, 0 }  ,draw opacity=1 ][black, very thick]    (450.94,198.4) .. controls (428.65,185.75) and (414.73,180.39) .. (405.84,182.65) ;
							\draw [color={rgb, 255:red, 0; green, 0; blue, 0 }  ,draw opacity=1 ][black, very thick]    (372.17,212.37) .. controls (355.23,213.54) and (333.37,222.66) .. (328.1,244.21) ;
							\draw [color={rgb, 255:red, 0; green, 0; blue, 0 }  ,draw opacity=1 ][black, very thick]    (325.45,256.57) .. controls (320.02,272.85) and (293.85,260.56) .. (284.35,270.31) .. controls (274.85,280.05) and (285.51,298.28) .. (304.45,303.31) ;
							\draw [color={rgb, 255:red, 0; green, 0; blue, 0 }  ,draw opacity=1 ][black, very thick]    (359.54,221.65) .. controls (363.42,230.38) and (364.84,241.01) .. (361,242.05) .. controls (357.17,243.09) and (354.48,240.52) .. (335.01,229.17) .. controls (315.53,217.82) and (304.4,222.58) .. (306.17,228.49) .. controls (307.94,234.4) and (314.01,241.54) .. (328.43,250.73) .. controls (342.85,259.92) and (383.52,279.71) .. (393.93,286.84) .. controls (404.34,293.98) and (410.9,308.97) .. (394.15,310.6) .. controls (377.41,312.23) and (335.33,309.49) .. (329.57,308.25) ;
							\draw [color={rgb, 255:red, 0; green, 0; blue, 0 }  ,draw opacity=1 ][black, very thick]    (376.45,286.32) .. controls (365.22,299.89) and (339.69,331.36) .. (328.1,330.73) .. controls (316.51,330.09) and (318.2,316.8) .. (314.88,297.16) .. controls (311.57,277.53) and (273.92,267.88) .. (266.04,257.73) .. controls (258.16,247.58) and (270.16,244.69) .. (279,246.18) .. controls (287.84,247.67) and (302.89,256.35) .. (308.38,259.84) ;
							\draw [color={rgb, 255:red, 0; green, 0; blue, 0 }  ,draw opacity=1 ][black, very thick]    (343.18,291.59) .. controls (340.94,282.98) and (336.67,279.33) .. (325.35,271.32) ;
							\draw [color={rgb, 255:red, 0; green, 0; blue, 0 }  ,draw opacity=1 ][black, very thick]    (396.72,191.3) .. controls (395.9,193.57) and (394.83,210.75) .. (382.84,211.22) ;
							\draw  [fill={rgb, 255:red, 0; green, 0; blue, 0 }  ,fill opacity=1 ][black, very thick]  (338.66,291.59) .. controls (338.66,289.1) and (340.68,287.07) .. (343.18,287.07) .. controls (345.67,287.07) and (347.7,289.1) .. (347.7,291.59) .. controls (347.7,294.09) and (345.67,296.11) .. (343.18,296.11) .. controls (340.68,296.11) and (338.66,294.09) .. (338.66,291.59) -- cycle ;
							\draw  [fill={rgb, 255:red, 0; green, 0; blue, 0 }  ,fill opacity=1 ][black, very thick]  (439.7,229.4) .. controls (439.7,226.9) and (441.73,224.88) .. (444.22,224.88) .. controls (446.72,224.88) and (448.74,226.9) .. (448.74,229.4) .. controls (448.74,231.9) and (446.72,233.92) .. (444.22,233.92) .. controls (441.73,233.92) and (439.7,231.9) .. (439.7,229.4) -- cycle ;
							\draw  [black, very thick]  (369.76,190.39) .. controls (369.76,187.89) and (371.78,185.87) .. (374.28,185.87) .. controls (376.77,185.87) and (378.8,187.89) .. (378.8,190.39) .. controls (378.8,192.89) and (376.77,194.91) .. (374.28,194.91) .. controls (371.78,194.91) and (369.76,192.89) .. (369.76,190.39) -- cycle ;
							\draw  [black, very thick]  (388.39,203.1) .. controls (388.39,200.6) and (390.41,198.58) .. (392.91,198.58) .. controls (395.41,198.58) and (397.43,200.6) .. (397.43,203.1) .. controls (397.43,205.59) and (395.41,207.62) .. (392.91,207.62) .. controls (390.41,207.62) and (388.39,205.59) .. (388.39,203.1) -- cycle ;
							\draw  [black, very thick]  (330.49,229.17) .. controls (330.49,226.68) and (332.51,224.65) .. (335.01,224.65) .. controls (337.5,224.65) and (339.53,226.68) .. (339.53,229.17) .. controls (339.53,231.67) and (337.5,233.69) .. (335.01,233.69) .. controls (332.51,233.69) and (330.49,231.67) .. (330.49,229.17) -- cycle ;
							\draw  [black, very thick]  (279.83,270.31) .. controls (279.83,267.81) and (281.86,265.79) .. (284.35,265.79) .. controls (286.85,265.79) and (288.87,267.81) .. (288.87,270.31) .. controls (288.87,272.8) and (286.85,274.83) .. (284.35,274.83) .. controls (281.86,274.83) and (279.83,272.8) .. (279.83,270.31) -- cycle ;
							\draw  [black, very thick]  (351.12,310.19) .. controls (351.12,307.69) and (353.14,305.67) .. (355.64,305.67) .. controls (358.13,305.67) and (360.16,307.69) .. (360.16,310.19) .. controls (360.16,312.68) and (358.13,314.71) .. (355.64,314.71) .. controls (353.14,314.71) and (351.12,312.68) .. (351.12,310.19) -- cycle ;
							\draw  [black, very thick]  (399.67,246.19) -- (394.48,252.9) -- (392.67,244.97) ;
							\draw  [color={rgb, 255:red, 208; green, 2; blue, 27 }  ,draw opacity=1 ][fill={rgb, 255:red, 208; green, 2; blue, 27 }  ,fill opacity=1 ] (411.22,184.27) .. controls (417.08,184.54) and (428.86,188.11) .. (433.58,192.2) .. controls (438.31,196.29) and (432.09,201.58) .. (428.13,202.73) .. controls (424.17,203.87) and (415.13,202.27) .. (409.31,194.18) .. controls (403.49,186.09) and (405.35,184.01) .. (411.22,184.27) -- cycle ;
							\draw (363,330) node [anchor=north west][inner sep=0.75pt]   {$K_{10}$};
						\end{tikzpicture}				
					\end{subfigure}
					\begin{subfigure}{0.34\textwidth}
						\centering
						\tikzset{every picture/.style={line width=0.75pt}} 
						\begin{tikzpicture}[x=0.75pt,y=0.75pt,yscale=-1,xscale=1]
							\draw  [draw opacity=0][black, very thick]  (235.44,526.86) .. controls (225.68,550.86) and (202.45,567.75) .. (175.34,567.75) .. controls (139.44,567.75) and (110.34,538.13) .. (110.34,501.6) .. controls (110.34,465.06) and (139.44,435.45) .. (175.34,435.45) .. controls (202.63,435.45) and (225.98,452.55) .. (235.62,476.8) -- (175.34,501.6) -- cycle ; \draw  [black, very thick]  (235.44,526.86) .. controls (225.68,550.86) and (202.45,567.75) .. (175.34,567.75) .. controls (139.44,567.75) and (110.34,538.13) .. (110.34,501.6) .. controls (110.34,465.06) and (139.44,435.45) .. (175.34,435.45) .. controls (202.63,435.45) and (225.98,452.55) .. (235.62,476.8) ;  
							\draw  [fill={rgb, 255:red, 0; green, 0; blue, 0 }  ,fill opacity=1 ][black, very thick]  (230.92,526.86) .. controls (230.92,524.36) and (232.94,522.34) .. (235.44,522.34) .. controls (237.93,522.34) and (239.96,524.36) .. (239.96,526.86) .. controls (239.96,529.35) and (237.93,531.38) .. (235.44,531.38) .. controls (232.94,531.38) and (230.92,529.35) .. (230.92,526.86) -- cycle ;
							\draw  [fill={rgb, 255:red, 0; green, 0; blue, 0 }  ,fill opacity=1 ][black, very thick]  (231.1,476.8) .. controls (231.1,474.3) and (233.13,472.28) .. (235.62,472.28) .. controls (238.12,472.28) and (240.14,474.3) .. (240.14,476.8) .. controls (240.14,479.3) and (238.12,481.32) .. (235.62,481.32) .. controls (233.13,481.32) and (231.1,479.3) .. (231.1,476.8) -- cycle ;
							\draw  [black, very thick, ->]  (132.59,451.85) -- (230.27,466.29) ;
							\draw  [red, very thick, ->]  (221.34,454.95) -- (118.86,468.81) ;
							\draw  [red, very thick, ->]  (113.34,481.64) -- (210.43,445.88) ;
							\draw  [black, very thick, ->]  (187.16,436.24) -- (111.4,490.44) ;
							\draw  [black, very thick, ->]  (180.64,567.38) -- (169.4,435.74) ;
							\draw  [black, very thick, ->]  (110.43,498.26) -- (170.14,567.38) ;
							\draw  [black, very thick, ->]  (158.64,565.63) -- (110.89,510.13) ;
							\draw  [black, very thick, ->]  (114.13,524.14) -- (230.39,536.71) ;
							\draw  [black, very thick, ->]  (222.49,547.05) -- (119.59,535.43) ;
							\draw  [black, very thick, ->]  (128.86,547.8) -- (210.68,557.27) ;
							\draw (96.97,460.35) node [anchor=north west][inner sep=0.75pt]  [font=\large]  {$c_{2}$};
							\draw (210.22,430.71) node [anchor=north west][inner sep=0.75pt]  [font=\large]  {$c_{3}$};
							\draw (90.33,484.65) node [anchor=north west][inner sep=0.75pt]  [font=\large]  {$c_{4}$};
							\draw (225.63,447.43) node [anchor=north west][inner sep=0.75pt]  [font=\scriptsize]  {$-$};
							\draw (97.63,518.36) node [anchor=north west][inner sep=0.75pt]  [font=\scriptsize]  {$+$};
							\draw (184.01,423.17) node [anchor=north west][inner sep=0.75pt]  [font=\scriptsize]  {$-$};
							\draw (115.26,444.5) node [anchor=north west][inner sep=0.75pt]  [font=\scriptsize]  {$-$};
							\draw (157.88,418.24) node [anchor=north west][inner sep=0.75pt]  [font=\large]  {$c_{5}$};
							\draw (235.12,459.95) node [anchor=north west][inner sep=0.75pt]  [font=\large]  {$c_{1}$};
							\draw (160.38,572.71) node [anchor=north west][inner sep=0.75pt]  [font=\large]  {$c_{6}$};
							\draw (90.44,501.67) node [anchor=north west][inner sep=0.75pt]  [font=\large]  {$c_{7}$};
							\draw (177,573.79) node [anchor=north west][inner sep=0.75pt]  [font=\scriptsize]  {$+$};
							\draw (236.3,531.77) node [anchor=north west][inner sep=0.75pt]  [font=\large]  {$c_{8}$};
							\draw (152,571.38) node [anchor=north west][inner sep=0.75pt]  [font=\scriptsize]  {$+$};
							\draw (215.78,554.62) node [anchor=north west][inner sep=0.75pt]  [font=\large]  {$c_{10}$};
							\draw (97.85,528.95) node [anchor=north west][inner sep=0.75pt]  [font=\large]  {$c_{9}$};
							\draw (96.01,476.5) node [anchor=north west][inner sep=0.75pt]  [font=\scriptsize]  {$-$};
							\draw (94.35,494.5) node [anchor=north west][inner sep=0.75pt]  [font=\scriptsize]  {$-$};
							\draw (112.29,544.09) node [anchor=north west][inner sep=0.75pt]  [font=\scriptsize]  {$+$};
							\draw (228.68,543.17) node [anchor=north west][inner sep=0.75pt]  [font=\scriptsize]  {$-$};
							\draw (172,595) node [anchor=north west][inner sep=0.75pt]    {(a)};
						\end{tikzpicture}					
					\end{subfigure}
					\qquad 
					\begin{subfigure}{0.34\textwidth}
						\centering
						\tikzset{every picture/.style={line width=0.75pt}} 
						\begin{tikzpicture}[x=0.75pt,y=0.75pt,yscale=-1,xscale=1]
							\draw  [draw opacity=0][black, very thick]  (235.44,526.86) .. controls (225.68,550.86) and (202.45,567.75) .. (175.34,567.75) .. controls (139.44,567.75) and (110.34,538.13) .. (110.34,501.6) .. controls (110.34,465.06) and (139.44,435.45) .. (175.34,435.45) .. controls (202.63,435.45) and (225.98,452.55) .. (235.62,476.8) -- (175.34,501.6) -- cycle ; \draw  [line width=1.5]  (235.44,526.86) .. controls (225.68,550.86) and (202.45,567.75) .. (175.34,567.75) .. controls (139.44,567.75) and (110.34,538.13) .. (110.34,501.6) .. controls (110.34,465.06) and (139.44,435.45) .. (175.34,435.45) .. controls (202.63,435.45) and (225.98,452.55) .. (235.62,476.8) ;  
							\draw  [fill={rgb, 255:red, 0; green, 0; blue, 0 }  ,fill opacity=1 ][black, very thick]  (230.92,526.86) .. controls (230.92,524.36) and (232.94,522.34) .. (235.44,522.34) .. controls (237.93,522.34) and (239.96,524.36) .. (239.96,526.86) .. controls (239.96,529.35) and (237.93,531.38) .. (235.44,531.38) .. controls (232.94,531.38) and (230.92,529.35) .. (230.92,526.86) -- cycle ;
							\draw  [fill={rgb, 255:red, 0; green, 0; blue, 0 }  ,fill opacity=1 ][black, very thick]  (231.1,476.8) .. controls (231.1,474.3) and (233.13,472.28) .. (235.62,472.28) .. controls (238.12,472.28) and (240.14,474.3) .. (240.14,476.8) .. controls (240.14,479.3) and (238.12,481.32) .. (235.62,481.32) .. controls (233.13,481.32) and (231.1,479.3) .. (231.1,476.8) -- cycle ;
							\draw  [black, very thick, ->]  (132.59,451.85) -- (230.27,466.29) ;
							\draw  [red, very thick, <-]  (221.34,454.95) -- (118.86,468.81) ;
							\draw  [red, very thick, <-]  (113.34,481.64) -- (210.43,445.88) ;
							\draw  [black, very thick, ->]  (187.16,436.24) -- (111.4,490.44) ; 
							\draw  [black, very thick, ->]  (180.64,567.38) -- (169.4,435.74) ; 
							\draw  [black, very thick, ->]  (110.43,498.26) -- (170.14,567.38) ;
							\draw  [black, very thick, ->]  (158.64,565.63) -- (110.89,510.13) ;
							\draw  [black, very thick, ->]  (114.13,524.14) -- (230.39,536.71) ;
							\draw  [black, very thick, ->]  (222.49,547.05) -- (119.59,535.43) ;
							\draw  [black, very thick, ->]  (128.86,547.8) -- (210.68,557.27) ;
							\draw (105,462) node [anchor=north west][inner sep=0.75pt]  [font=\scriptsize]  {$+$};
							\draw (212,436) node [anchor=north west][inner sep=0.75pt]  [font=\scriptsize]  {$+$};
							\draw (88,478) node [anchor=north west][inner sep=0.75pt]  [font=\large]  {$c_{4}'$};
							\draw (225,440) node [anchor=north west][inner sep=0.75pt]  [font=\large]  {$c_{2}'$};
							\draw (97.63,518.36) node [anchor=north west][inner sep=0.75pt]  [font=\scriptsize]  {$+$};
							\draw (184.01,423.17) node [anchor=north west][inner sep=0.75pt]  [font=\scriptsize]  {$-$};
							\draw (115.26,444.5) node [anchor=north west][inner sep=0.75pt]  [font=\scriptsize]  {$-$};
							\draw (160,412) node [anchor=north west][inner sep=0.75pt]  [font=\large]  {$c_{5}'$};
							\draw (235.12,456) node [anchor=north west][inner sep=0.75pt]  [font=\large]  {$c_{1}'$};
							\draw (164,570) node [anchor=north west][inner sep=0.75pt]  [font=\large]  {$c_{6}'$};
							\draw (92,500) node [anchor=north west][inner sep=0.75pt]  [font=\large]  {$c_{7}'$};
							\draw (177,573.79) node [anchor=north west][inner sep=0.75pt]  [font=\scriptsize]  {$+$};
							\draw (232,528) node [anchor=north west][inner sep=0.75pt]  [font=\large]  {$c_{8}'$};
							\draw (152,571.38) node [anchor=north west][inner sep=0.75pt]  [font=\scriptsize]  {$+$};
							\draw (212,550) node [anchor=north west][inner sep=0.75pt]  [font=\large]  {$c_{10}'$};
							\draw (99,526) node [anchor=north west][inner sep=0.75pt]  [font=\large]  {$c_{9}'$};
							\draw (90.8,465) node [anchor=north west][inner sep=0.75pt]  [font=\large]  {$c_{3}'$};
							\draw (94.35,496) node [anchor=north west][inner sep=0.75pt]  [font=\scriptsize]  {$-$};
							\draw (112.29,544.09) node [anchor=north west][inner sep=0.75pt]  [font=\scriptsize]  {$+$};
							\draw (228.68,543.17) node [anchor=north west][inner sep=0.75pt]  [font=\scriptsize]  {$-$};
							\draw (171,595) node [anchor=north west][inner sep=0.75pt]    {(b)};
						\end{tikzpicture}							
					\end{subfigure}		
					\caption{Diagrams and Gauss diagrams of planar virtual knotoids $K_9$ and $K_{10}$.}
					\label{fig21}
				\end{center}
			\end{figure}
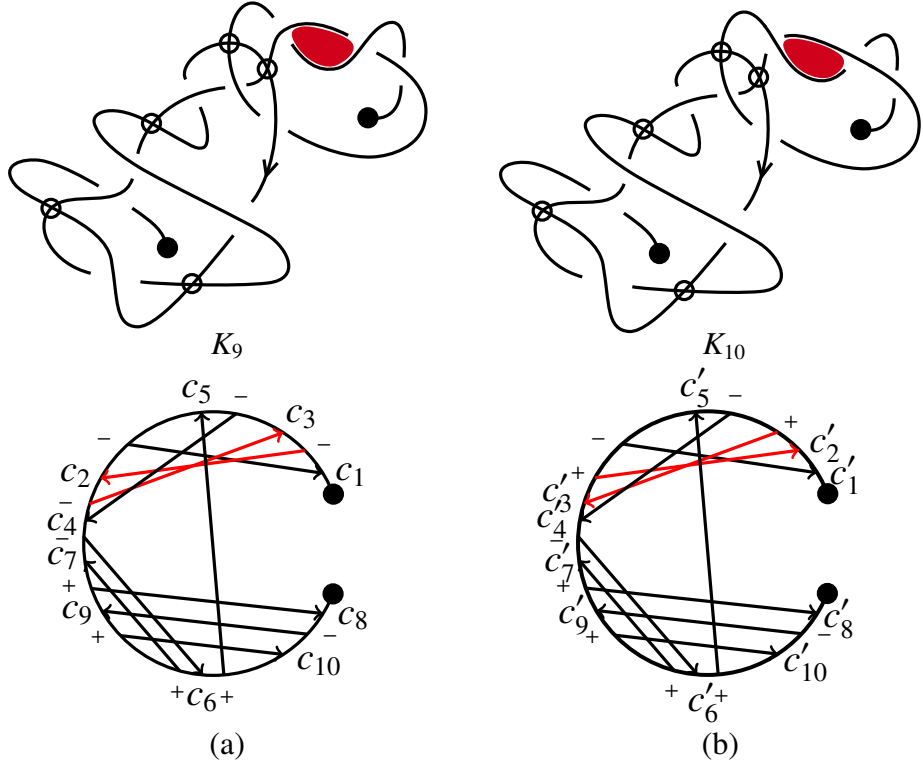
			We now determine the index of each chord from the Gauss diagrams of $K_9$ and $K_{10}$. The corresponding values for $K_9$ are
			\begin{align*}
				&
				\Ind(c_1)=-2,\ \Ind(c_2)=2,\ \Ind(c_3)=-2,\ \Ind(c_4)=2,\ \Ind(c_5)=3, \\
				&
				\Ind(c_6)=-3,\ \Ind(c_7)=3,\ \Ind(c_8)=-3,\ \Ind(c_9)=3,\ \Ind(c_{10})=-3;
			\end{align*}
			and those for $K_{10}$ read
			\begin{align*}
				&
				\Ind(c_1')=-2,\ \Ind(c_2')=-2,\ \Ind(c_3')=2,\ \Ind(c_4')=2,\ \Ind(c_5')=3, \\
				&
				\Ind(c_6')=-3,\ \Ind(c_7')=3,\ \Ind(c_8')=-3,\ \Ind(c_9')=3,\ \Ind(c_{10}')=-3.
			\end{align*}
			
			Therefore, Figure~\ref{fig16} displays the diagram of $K_9^{(2)}$, and $K_{10}^{(2)}$ is the trivial knotoid, so we can see that the diagrams of $K_9^{(2)}$ and $K_{10}^{(2)}$ differ by a VRCC-move and two $\Omega_2$-moves. According to Example~\ref{ex:3}, $d_G(K_9^{(2)}, K_{10}^{(2)}) = 2$. We conclude that $K_9^{(2)}$ and $K_{10}^{(2)}$ are not equivalent, which implies $d_\text{VRCC}(K_9^{(2)}, K_{10}^{(2)}) = 1$. By Theorem~\ref{thm-vrcov}, we obtain $d_\text{VRCC}(K_9, K_{10})\geq d_\text{VRCC}(K_9^{(2)}, K_{10}^{(2)}) = 1$. We can also find that the diagrams of $K_9$ and $K_{10}$ differ by a VRCC-move, so $d_\text{VRCC}(K_9, K_{10}) \le 1$. Hence, $d_\text{VRCC}(K_9, K_{10}) = d_\text{VRCC}(K_9^{(2)}, K_{10}^{(2)}) = 1$.
			
			By Theorem~\ref{thm-ccov}, $d_G(K_9, K_{10})\geq d_G(K_9^{(2)}, K_{10}^{(2)}) = 2$. We further note that the diagrams of $K_9$ and $K_{10}$ differ by two crossing changes, yielding $d_G(K_9, K_{10}) \le 2$. Therefore, $d_G(K_9, K_{10}) = 2$.
			
			In summary, $d_\text{VRCC}(K_9, K_{10}) = 1$ and $d_G(K_9, K_{10}) = 2$.
		} 
	\end{example}
	
	It is also easy to see that the diagrams of the two planar virtual knotoids $K_1$ and $K_2$ in Example~\ref{ex:3} also differ by a VRCC-move in a bigonal region. Following a similar argument as above, we conclude that their VRCC-Gordian distance is $1$.

\end{document}